\pdfoutput=1 

\documentclass[12pt]{extarticle}
\usepackage[utf8]{inputenc}
\usepackage[T1]{fontenc}

\usepackage{caption}
\oddsidemargin \evensidemargin
\usepackage{float}
\usepackage{resizegather}
\usepackage{amssymb}
\usepackage{pmboxdraw}
\usepackage{amsmath}
\usepackage{bbm}
\usepackage{amsthm}
\usepackage{amssymb}
\usepackage{dsfont}
\usepackage[dvipsnames]{xcolor}
\usepackage{graphicx}
\usepackage{alphabeta}
\usepackage{tikz,pgfplots}
\pgfplotsset{compat=1.18}
\usepackage{pgfplots}
\usepackage{comment}
\usepackage[margin=0.5cm]{caption}
\usepackage{hyperref}
\usepackage{subcaption}
\usepackage[all]{xy}
\usepackage[greek,english]{babel}
\usepackage[export]{adjustbox}
\usepackage{hyperref}
\usepackage{mathrsfs}  
\hypersetup{
    colorlinks=true,
    linkcolor=myblue,
    filecolor=magenta,      
    urlcolor=myblue,
    citecolor=BurntOrange,
}
\usepackage{enumitem}
\usepackage{booktabs}

\definecolor{myblue}{rgb}{0.25,0.45,0.99}
\definecolor{myorange}{rgb}{0.8500, 0.3250, 0.0980}
\definecolor{myyellow}{rgb}{0.9290, 0.6940, 0.1250}
\definecolor{mypurple}{rgb}{0.4940, 0.1840, 0.5560}
\definecolor{mygreen}{rgb}{0.4660, 0.6740, 0.1880}

\usepackage[activate={true,nocompatibility},final,tracking=true,kerning=true,spacing=true,factor=1100,stretch=10,shrink=10]{microtype}

\let\OLDthebibliography\thebibliography
\renewcommand\thebibliography[1]{
  \OLDthebibliography{#1}
  \setlength{\parskip}{0.4pt}
  \setlength{\itemsep}{3.0pt plus 0.3ex}
}

\usepackage{listings}

\usepackage{algorithm}
\usepackage{algpseudocode}

\newtheorem{theorem}{Theorem}[section]
\newtheorem{theorem*}[theorem]{Theorem*}

\newtheorem{lemma}[theorem]{Lemma}
\newtheorem{corollary}[theorem]{Corollary}
\newtheorem{proposition}[theorem]{Proposition}

\newtheorem{thm/conj}[theorem]{Theorem/Conjecture}

\theoremstyle{definition}
\newtheorem{examplex}[theorem]{Example}
\newenvironment{example}
{\pushQED{\qed}\examplex}
{\popQED\endexamplex}

\newtheorem{remark}[theorem]{Remark}

\theoremstyle{remark}

\newcommand\restr[2]{{
  \left.\kern-\nulldelimiterspace 
  #1
  \vphantom{\big|} 
  \right|_{#2}
  }}

\usepackage{titling} 
\title{\bf Toric Amplitudes and Universal Adjoints}
\author{Simon Telen}
\date{}

\begin{document}

\maketitle

\begin{abstract}
\noindent A toric amplitude is a rational function associated to a simplicial polyhedral fan. The definition is inspired by scattering amplitudes in particle physics. We prove algebraic properties of such amplitudes and study the geometry of their zero loci. These hypersurfaces play the role of Warren's adjoint via a dual volume interpretation. We investigate their Fano schemes and singular loci via the nef cone and toric irrelevant ideal of the~fan.  
\end{abstract}

\section{Introduction}

Let $\Sigma$ be a simplicial polyhedral fan in $\mathbb{R}^d$. 
The set of $k$-dimensional cones of $\Sigma$ is denoted by $\Sigma(k)$. We choose a ray generator $u_\rho \in \mathbb{R}^d$ for each ray $\rho \in \Sigma(1)$ and record these vectors in the rows of an $n \times d$ matrix $U$ in arbitrary order. The \emph{toric amplitude} associated to $\Sigma$ and $U$ is the following rational function in $x_\rho$, $\rho \in \Sigma(1)$ with real, positive coefficients: 
\begin{equation} \label{eq:Amp}
{\rm Amp}_\Sigma(x) \, = \, \sum_{\sigma \in \Sigma(d)} \frac{| \det U_\sigma |}{\prod_{\rho \in \sigma(1)} x_\rho} \, \, \, \, \in \, \, \mathbb{R}(x_\rho \, : \, \rho \in \Sigma(1)).
\end{equation}
Here, $U_\sigma$ is the submatrix of $U$ whose rows are indexed by the rays of $\sigma$, and $|\cdot |$ denotes the absolute value. The product in the denominator ranges over all rays of $\sigma$. 
The dependence on $U$ is discussed at the beginning of Section \ref{sec:firstprop} and left implicit in the notation.
The~\emph{universal adjoint} of $\Sigma$ is a polynomial of degree $n-d$ obtained by clearing the denominator in ${\rm Amp}_\Sigma$: 
\begin{equation} \label{eq:Adj}
{\rm Adj}_\Sigma(x)  \, = \, \Big( \prod_{\rho \in \Sigma(1)} x_\rho \Big ) \cdot {\rm Amp}_\Sigma(x) \, = \, \sum_{\sigma \in \Sigma(d)} |\det U_\sigma| \cdot \prod_{\rho \notin \sigma(1)} x_\rho . 
\end{equation}
This paper studies the geometry of the hypersurface defined by ${\rm Adj}_\Sigma$ in $\mathbb{P}^{n-1} = \mathbb{CP}^{n-1}$. We denote this hypersurface by ${\cal A}_\Sigma$. If $\Sigma = \Sigma_P$ is the normal fan of a convex polyhedron $P \subset \mathbb{R}^d$, then our sum is over the vertices of $P$. In that case, we will also write ${\rm Amp}_P$, ${\rm Adj}_P$ and ${\cal A}_P$.

\begin{example} \label{ex:pentagonintro}
Figure \ref{fig:fanpentagon} shows the complete fan $\Sigma$ in $\mathbb{R}^2$ corresponding to $U = \left ( \begin{smallmatrix}
     1 & 0 & -1 & -1 & 0 \\ 0 & 1 & 1 & 0 & -1
\end{smallmatrix} \right)^t$. 
\begin{figure}
    \centering
    \includegraphics[height = 3.5cm]{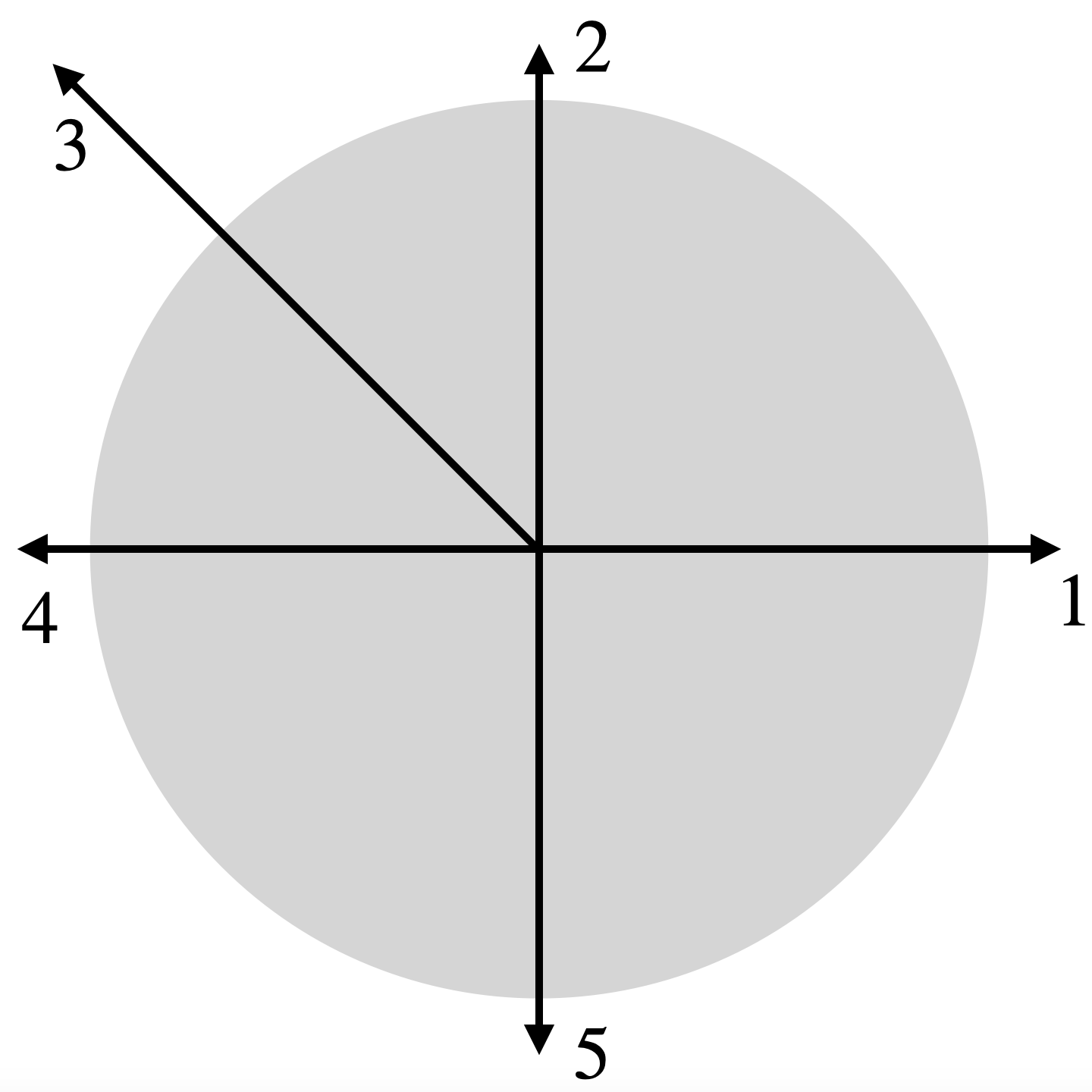} \quad \quad  \quad \quad \quad  
    \includegraphics[height = 3.5cm]{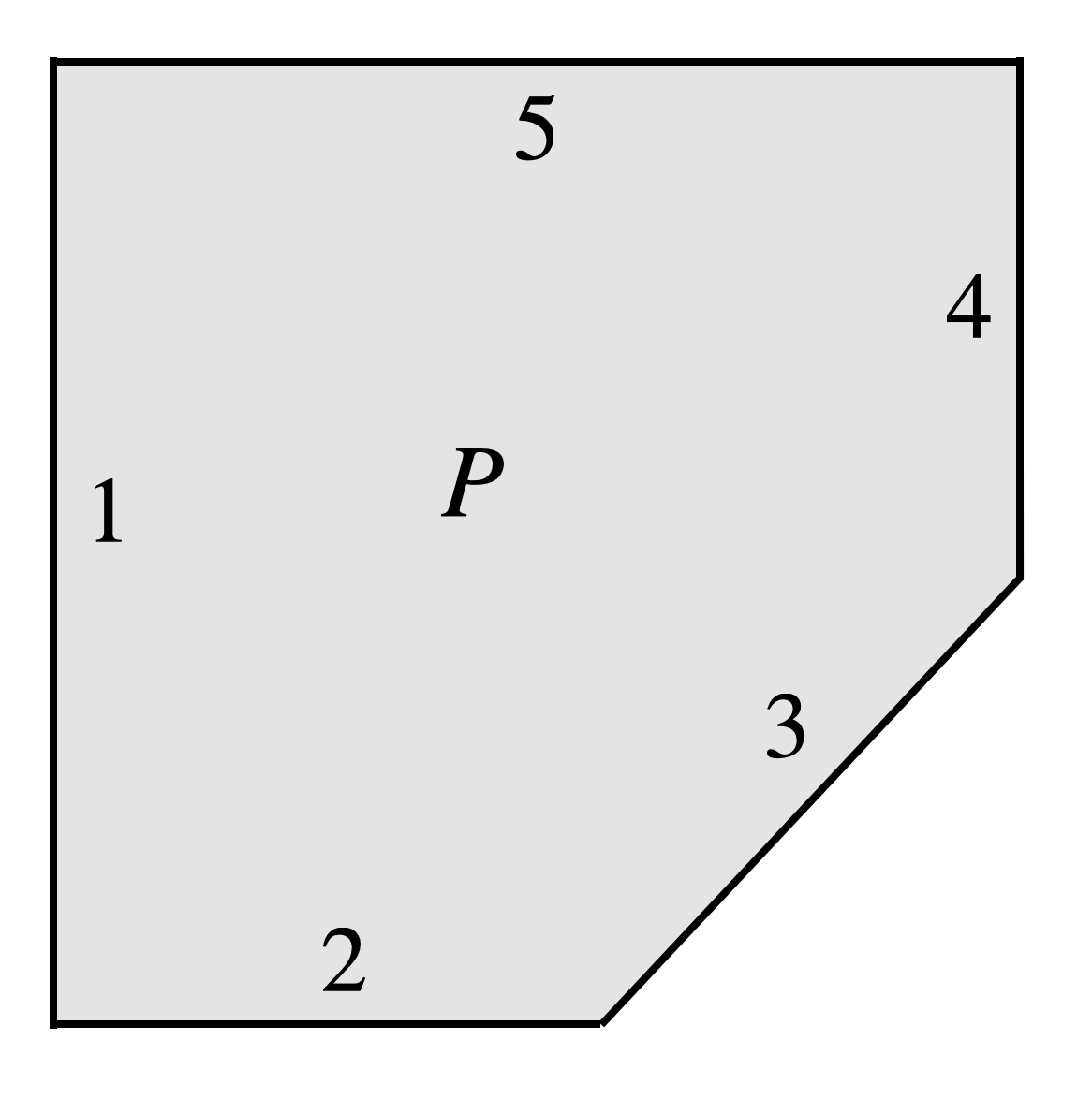}
    \caption{The normal fan of a pentagon.}
    \label{fig:fanpentagon}
\end{figure}
It is the normal fan of the pentagon $P$ shown in the right part of the figure. The amplitude~is
    \begin{equation}  \label{eq:amppentagon}
    {\rm Amp}_P \, = \, \frac{1}{x_1x_2} \, + \, \frac{1}{x_2x_3} \, + \, \frac{1}{x_3x_4} \, + \, \frac{1}{x_4x_5}  \, + \, \frac{1}{x_5x_1}.
    \end{equation}
    The universal adjoint is obtained by multiplying with the product of all $x_i$: 
    \begin{equation} \label{eq:adjpentagon} {\rm Adj}_P \, = \, x_3x_4x_5 \, + \, x_1x_4x_5 \, + \,x_1x_2x_5 \, + \,x_1x_2x_3 \, + \,x_2x_3x_4. \end{equation}
    Its zero locus in $\mathbb{P}^4$ is a cubic threefold with ten isolated singularities. It was pointed out to us by Bernd Sturmfels that this is the \emph{Segre cubic} \cite{Dolgachev}. Such a threefold contains 15 planes, each of which contains four of its nodes. We investigate this configuration in Section \ref{subsec:pentagon}.
\end{example}

In the next paragraphs, we justify the names \emph{toric amplitude} and \emph{universal adjoint}. First, \emph{amplitude} refers to scattering amplitudes in theoretical particle physics. These are important for predictions in particle collider experiments. More specifically, the toric amplitude ${\rm Amp}_P(x)$ corresponding to the \emph{ABHY associahedron} $P$ from \cite{arkani2018scattering} is the bi-adjoint scalar $\phi^3$-amplitude at tree level. In fact, the pentagon $P$ in Figure \ref{fig:fanpentagon} represents such an ABHY associahedron. Substituting $x_1 = X_{1,3}, \, x_2 = X_{1,4}, \, x_3 = X_{2,4}, \, x_4 = X_{2,5}, \, x_5 = X_{3,5}$ in \eqref{eq:amppentagon}, we find the rational function shown in \cite[Equation (3.24)]{arkani2018scattering}. The importance of studying zeros of amplitudes was highlighted in \cite{arkani2024hidden}, where the focus is on ABHY associahedra in any dimension. We come back to this in Section \ref{subsec:assoc}, but we assume no background in physics for the rest of this~article. 

Let $P = \{ y \in \mathbb{R}^d \, : \, u_\rho \cdot y + z_\rho \geq 0 \text{ for } \rho \in \Sigma(1)\}$ be a minimal facet description of $P$. The name \emph{adjoint} refers to the fact that ${\rm Adj}_P(x)$ specializes to Warren's adjoint \cite{kohn2020projective,warren1996barycentric} of the polytope $P$ when setting $x_\rho = u_\rho \cdot y + z_\rho$. This holds for all $z \in \mathbb{R}^n$ contained in an $n$-dimensional cone, see Lemma \ref{lem:universaladjoint}, which justifies the name \emph{universal} adjoint. The study of adjoint hypersurfaces plays a key role in positive geometry, see \cite{ranestad2025positive} or \cite[Section~4.4]{lam2024invitation}.

In \cite{lam2024matroids}, Lam constructs rational amplitude functions from the combinatorics of matroids. In combinatorial algebraic geometry, the geometric objects associated to realizable matroids are hyperplane arrangement complements and their compactifications. On the other hand, to polytopes and fans one associates a toric variety. We propose the name \emph{toric} amplitude to emphasize this analogy. To a certain extent, the combinatorial structure of ${\rm Amp}_\Sigma$ is similar to that of normal toric varieties. For instance, the variables $x_\rho$ indexing the rays of $\Sigma$ are reminiscent of the Cox coordinates on the abstract toric variety $X_\Sigma$ obtained from $\Sigma$ \cite{Cox1995}. Moreover, the monomials of ${\rm Adj}_P$ are the minimal generators of the irrelevant ideal $B(\Sigma_P)$ in the Cox ring of~$X_{\Sigma_P}$. That ideal plays a crucial role in our study of the Fano schemes and singular locus of ${\cal A}_\Sigma$. We will also identify linear spaces contained in ${\cal A}_\Sigma$ from wall inequalities for the deformation cone of $\Sigma$ (Section \ref{sec:chambercomplex}). This is the nef cone of $X_\Sigma$ in toric geometry. We point out that, unlike the toric literature, this paper does not require $\Sigma$ to be rational. 

As indicated above, our goal is to study ${\cal A}_\Sigma$ from the point of view of combinatorial algebraic geometry. In particular, we are interested in describing its Fano schemes, i.e., the linear spaces contained in ${\cal A}_\Sigma$, and in its singular locus in terms of the combinatorics of $\Sigma$. 

\vspace{-0.3cm}
\paragraph{Related work.} In Santal\'o geometry, the toric amplitude ${\rm Amp}_P$ is the universal barrier function for a linear program in standard form \cite[Corollary 2.6]{pavlov2025santalo}. The polytope $P \subset \mathbb{R}^d$ is identified with an affine section of the nonnegative orthant in $\mathbb{R}^n$.
Minimizing ${\rm Amp}_P$ on $P$ amounts to finding the interior point $y$ of $P$ which leads a polar dual polytope $(P-y)^\circ$ of minimal Euclidean volume. This is called the Santal\'o point of $P$. 
The connection with the dual volume function is also explained in \cite[Section 2]{gao2024dual}, where the function ${\rm Amp}_\Sigma$ appears in \cite[Definition 2.1]{gao2024dual} for general fans. In \cite{gao2024dual}, ${\rm Amp}_\Sigma$ is used to define a dual mixed volume function for a tuple $\mathbf{P}$ of polyhedra. It is shown in \cite[Proposition 14.10]{gao2024dual} that this dual mixed volume function evaluates to the tree-level $\phi^3$-amplitude when $\mathbf{P}$ consists of the Minkowski summands in the Loday realization of the associahedron. Such amplitudes exhibit a ``splitting behavior'', which essentially means that certain coordinate restrictions of the amplitude factor into simpler pieces. This was first studied in \cite{cachazo2022smoothly} for the CEGM amplitudes introduced in \cite{cachazo2019scattering}, and recently explored further in \cite{umbert2025splitting}. In these works, the amplitude is expressed as a function of Mandelstam variables $s_{ij}$. Our approach is more directly inspired by \cite{arkani2024hidden}, which expresses the amplitude in terms of the variables $X_{i,j}$, each associated to a facet of the ABHY associahedron. Detecting splitting behavior essentially means finding linear spaces contained in the zero locus of the amplitude. This motivates our study of the Fano schemes~of ${\cal A}_\Sigma$. 

\vspace{-0.3cm}
\paragraph{Outline and contributions.}

We start with motivating examples in Section \ref{sec:examples}. We study the hypersurface ${\cal A}_P$ in detail for the quadrilateral, the pentagon and the three-dimensional ABHY associahedron. For interested readers, we include a discussion of how the toric amplitude arises in physics at the end of Section \ref{sec:examples}. In Section \ref{sec:firstprop}, we prove some useful properties of toric amplitudes and universal adjoints. We show that they behave nicely under taking products (Lemma \ref{lem:factorfans}) and coordinate restrictions (Lemmas \ref{lem:restrict} and \ref{lem:restrictpolytope}). We show that, if $\Sigma$ is complete, then ${\cal A}_\Sigma \subset \mathbb{P}^{n-1}$ contains the projectivized column span of the matrix $U$ (Theorem \ref{thm:containsU}). We spell out the connection to dual volume functions and Warren's adjoint in Section \ref{sec:dualvolume}. Theorem \ref{thm:containsU} implies the well-known fact that Warren's adjoint of $P$ has degree at most $n-d-1$ (Proposition \ref{prop:degadj}) and gives a new geometric interpretation of this hypersurface as a linear section of the universal adjoint (Example \ref{ex:geometric} and Figure \ref{fig:adjquadrilateral}). Section \ref{sec:fano} is about Fano schemes. We observe that ${\cal A}_\Sigma$ contains the zero locus $Z(\Sigma)$ of the toric irrelevant ideal $B(\Sigma)$ (Proposition \ref{prop:containedinB}). We characterize ${\cal A}_P$ as the unique hypersurface of degree $n-d$ containing $Z(\Sigma)$ as well as one $(n-d-1)$-dimensional linear space for each edge of $P$ (Theorem \ref{thm:interpol}). For each face $\Delta$ of $P$ which is a product of simplices, we identify a coordinate subspace $\Lambda_\Delta$ so that the restriction $({\rm Adj}_P)_{|\Lambda_\Delta}$ is a product of linear forms (Corollary \ref{cor:simplexfaces}). This is our interpretation of ``splitting'' \cite{arkani2024hidden,cachazo2022smoothly}. Section \ref{sec:chambercomplex} relates some of the linear spaces contained in ${\cal A}_\Sigma$ to the chamber complex and deformation cone of $\Sigma$ and $U$; see Proposition \ref{prop:linspacefromcham}. This explains vanishing properties of Warren's adjoint of deformations of $P$ (Proposition \ref{prop:degenadjoint}). Section \ref{sec:singlocus} studies the singular locus of ${\cal A}_\Sigma$. We provide an efficient description of ${\rm Sing}({\cal A}_\Sigma) \cap Z(\Sigma)$ in Proposition \ref{prop:singirrel}. We prove a criterion to check whether ${\rm Sing}({\cal A}_\Sigma) \subseteq Z(\Sigma)$ (Corollary \ref{cor:checkIsatB}). We show that for a generic $n$-gon $P$, ${\cal A}_P$ is irreducible and the dimension of its singular locus is at most $n-4$ (Theorem \ref{thm:singngon}). This implies, via a Bertini argument (Theorem \ref{thm:bertini}),  that Warren's adjoint curve for a generic $n$-gon is smooth (Corollary \ref{cor:genericsmoothngon}). Code supporting this paper is found at \cite{mathrepo}. It relies on \texttt{Oscar.jl} for polyhedral and algebraic computations~\cite{OSCAR}.

\vspace{-0.3cm}
\paragraph{Notation.} 

Throughout the text, $\Sigma$ is a simplicial fan in $\mathbb{R}^d$. We write $\Sigma(k)$ for the set of $k$-dimensional cones of $\Sigma$, and $\sigma(1)$ for the rays of a cone $\sigma$. The polyhedron $P \subset \mathbb{R}^d$ is full-dimensional and simple. The linear span of $S \subseteq \mathbb{R}^d$, i.e., the smallest linear subspace of $\mathbb{R}^d$ containing $S$, is ${\rm span}_{\mathbb{R}}(S)$. We call the matrix $U \in \mathbb{R}^{n \times d}$ whose rows are generators of the rays of $\Sigma$ a \emph{ray (generator) matrix}. We assume that ${\rm rank}(U) = d$. We write $\mathbb{P}^{n-1}$ for the $(n-1)$-dimensional complex projective space with coordinates indexed by $\Sigma(1)$. Its homogeneous coordinate ring is $R_\Sigma = \mathbb{C}[x_\rho:  \rho \in \Sigma(1)]$. For an ideal $I \subseteq R_\Sigma$ with homogeneous generators $f_1, \ldots, f_k \in R_\Sigma$ we write $V(I) = V(f_1, \ldots, f_k) = \{ x \in \mathbb{P}^{n-1} \, : \, f_1(x) = \cdots = f_k(x) = 0 \}$.

\section{Polygons and associahedra} \label{sec:examples}
We start with some illustrative examples in which $\Sigma$ is the normal fan of a polygon in $\mathbb{R}^2$ or of an associahedron. These examples motivated this project. They highlight some properties of universal adjoints and set the stage for the general results proved in later sections. Our focus is on linear spaces contained in the adjoint hypersurface ${\cal A}_P$, and on its singular locus.

\subsection{The universal adjoint quadric of a quadrilateral} \label{subsec:quadrilateral}
The normal fan of a quadrilateral has four $2$-dimensional cones and four rays. We have 
\[ {\rm Amp}_P \, = \, \frac{u_{12}}{x_1x_2} + \frac{u_{23}}{x_2x_3} + \frac{u_{34}}{x_3x_4} + \frac{u_{14}}{x_1x_4} \quad \text{and} \quad {\rm Adj}_P \, = \, u_{12} \, x_3x_4 + u_{23} \, x_1x_4 + u_{34} \, x_1x_2 + u_{14} \, x_2x_3, \]
where $u_{ij} = |\det U_{ij}|$ and the rays are ordered in such a way that the cones of $\Sigma(2)$ are generated by $\{\rho_1,\rho_2\}$, $\{ \rho_2,\rho_3\}$, $\{\rho_3,\rho_4\}$ and $\{ \rho_1,\rho_4\}$. The adjoint hypersurface ${\cal A}_P \subset \mathbb{P}^3$ is a quadratic surface in $\mathbb{P}^3$. It is smooth unless the discriminant $(u_{12}u_{34}-u_{14}u_{23})^2$ vanishes. 

If ${\cal A}_P$ is singular, then it is a union of two planes, i.e., ${\rm Adj}_P$ factors. This happens, for instance, for the normal fan of the unit square $[0,1]^2$, see Figure \ref{fig:fansSec2} (left). The singular locus is the intersection of those two planes, which is the line $\mathbb{P}({\rm im}(U)) \subset \mathbb{P}^3$ spanned by the columns of the ray matrix $U \in \mathbb{R}^{4 \times 2}$. This line is always contained in ${\cal A}_P$ by Theorem \ref{thm:containsU}.

A smooth quadratic surface in $\mathbb{P}^3$ is classically ruled by two families of lines. In our setting, these families are described as follows. Let $p_{ij}$ be Pl\"ucker coordinates on ${\rm Gr}(2,4) \subset \mathbb{P}^5$. Set 
\[ {\cal F}_1 \, = \, \Big \{ p \in {\rm Gr}(2,4) \, : \, {\rm rank} \begin{pmatrix}
    p_{12} & p_{23} & p_{34} & p_{14} \\ u_{12} & u_{23} & u_{34} & -u_{14}
\end{pmatrix} \leq 1  \Big \} . \]
For a line $\Lambda \subset \mathbb{P}^3$, let $[\Lambda]$ be its point in ${\rm Gr}(2,4)$. The second family of lines is 
\[ {\cal F}_2 \, = \, \{ [\Lambda] \in {\rm Gr}(2,4) \, : \, \Lambda \text{ intersects the lines } \{x_1=x_3 =0 \}, \{x_2=x_4=0\} \text{ and } \mathbb{P}({\rm im}(U)) \}. \]
The union of (reduced) curves ${\cal F}_1 \cup {\cal F}_2 \subset {\rm Gr}(2,4)$ is the Fano scheme of lines contained in ${\cal A}_P$. The lines $\{x_1=x_3 =0 \}, \,  \{x_2=x_4=0\}$ and $\mathbb{P}({\rm im}(U))$ used in the definition of ${\cal F}_2$ appear in Theorem~\ref{thm:containsU} and Proposition \ref{prop:containedinB}. Notice that these three lines belong to ${\cal F}_1$.

\subsection{The Segre cubic of a pentagon} \label{subsec:pentagon}

The universal adjoint of a pentagon $P$ with normal fan $\Sigma$ is the following quinary cubic: 
\[ {\rm Adj}_P \, \,  = \, \,  u_{45} \, x_1x_2x_3 + u_{15} \, x_2x_3x_4 + u_{12} \, x_3x_4x_5 + u_{23} \, x_1x_4x_5 + u_{34} \, x_1x_2x_5, \quad u_{ij} > 0. \]
The singular locus ${\rm Sing}({\cal A}_P)$ of the threefold ${\cal A}_P = V({\rm Adj}_P) \subset \mathbb{P}^4$ is defined by 
\begin{align*}
&u_{45} \, x_2x_3 + u_{23} \, x_4x_5  + u_{34} \, x_2x_5 \, = \, 0, \quad &  
u_{45} \, x_1x_3 + u_{15} \, x_3x_4 + u_{34} \, x_1x_5 \, = \, 0, \\ 
&u_{45} \, x_1x_2 + u_{15} \, x_2x_4 + u_{12} \, x_4x_5 \, = \, 0, \quad & 
u_{15} \, x_2x_3 + u_{12} \, x_3x_5 + u_{23} \, x_1x_5  \, = \, 0, \\ 
&u_{12} \, x_3x_4 + u_{23} \, x_1x_4 + u_{34} \, x_1x_2  \, = \, 0. 
\end{align*}
This trivially contains the torus invariant points $(1:0:0:0:0), \ldots, (0:0:0:0:1)$ of $\mathbb{P}^4$. We denote these by $e_1, \ldots, e_5$. Five more points contained in ${\rm Sing}({\cal A}_P)$ are identified as follows. Substituting $x_1 = x_3=0$ in our equations, all but the first and the third are trivially satisfied. Additionally setting $u_{23}\, x_4 + u_{34} \, x_2 = u_{15} \, x_2 +u_{12} \, x_5 = 0$ gives $q_{13} = (0:-u_{12}u_{23}:0:u_{12}u_{34}:u_{15}u_{23})$. The points $q_{14},q_{24}, q_{25},q_{35}$ are found in the same way. The Hessian matrix of ${\rm Adj}_P$ has rank $4$ at each of the ten points $e_i, q_{ij}$, so these singular points are isolated.

We recall a classical result from algebraic geometry about irreducible cubic hypersurfaces with isolated singularities in $\mathbb{P}^4$. Such a threefold has at most ten nodes, there exists a threefold with ten nodes, and this is unique up to projective transformations. That threefold is known as the \emph{Segre cubic}. For details we refer to Dolgachev’s historical exposition \cite{Dolgachev}. We conclude from these facts that \emph{the universal adjoint threefold of a pentagon is the Segre cubic}.

Among the ten nodes $e_i$, $q_{ij}$, we find 15 quadruples that are coplanar. These quadruples span 15 planes which are contained in the Segre cubic. They come in three groups of five:
\[
\renewcommand\arraystretch{1.2}
\begin{matrix}
    \Lambda_{i,i+2}  = \{ x_i = x_{i+2} = 0 \}, \quad \quad \quad \quad H_i = \{x_i = u_{i-1,i}\, x_{i+1} + u_{i,i+1} \, x_{i-1} = 0\},\\ 
    L_i = \{u_{i-1,i} x_{i+1} + u_{i,i+1}  x_{i-1} = u_{i-1,i}u_{i+2,i-2}  x_{i+1} -u_{i,i+1}u_{i-2,i-1}  x_{i+2} + u_{i-1,i}u_{i+1,i+2}  x_{i-2} = 0\}.
\end{matrix}
\]
Here $i$ ranges over $\{1, 2, 3, 4, 5\}$. The indexing is cyclic and we use the convention $u_{ij} = u_{ji}$. For instance, if $i = 1$, then $u_{i-2,i-1} = u_{45}$ and $u_{i-1,i}= u_{51} = u_{15}$. In Theorem \ref{thm:interpol} we will characterize ${\cal A}_P$ as the unique cubic threefold containing the ten planes $\Lambda_{i,i+2}$ and $H_i$. 

The configuration of 15 planes contained in ${\cal A}_P$ and the 10 nodal singularities in ${\rm Sing}({\cal A}_P)$ is an abstract configuration $(15_4,10_6)$, meaning that each of the planes contains four nodes and each node is contained in six planes \cite[Proposition 2.2]{Dolgachev}. This is easily checked using the defining equations of our planes and points. For instance, $q_{i,i+2}$ is defined by 
\[ x_i \, = \, x_{i+2} \, = \, u_{i+1, i+2} \, x_{i-2} + u_{i+2,i-2} \, x_{i+1} \, = \, u_{i-1, i} \, x_{i+1} + u_{i,i+1} \, x_{i-1} \, = \, 0,\]
and clearly contained in, for instance, $L_i$. 
All incidences are summarized in Figure \ref{fig:config}. The right part of the figure shows the image of $e_i, q_{ij}$ and $H_i$ under the projection away from the line $\mathbb{P}({\rm im}(U)) \subset \mathbb{P}^4$. We will see in Section \ref{sec:chambercomplex} that this interacts nicely with the \emph{nef cone} of~$\Sigma$. 

\begin{figure}
\centering
\includegraphics[height = 4.5cm]{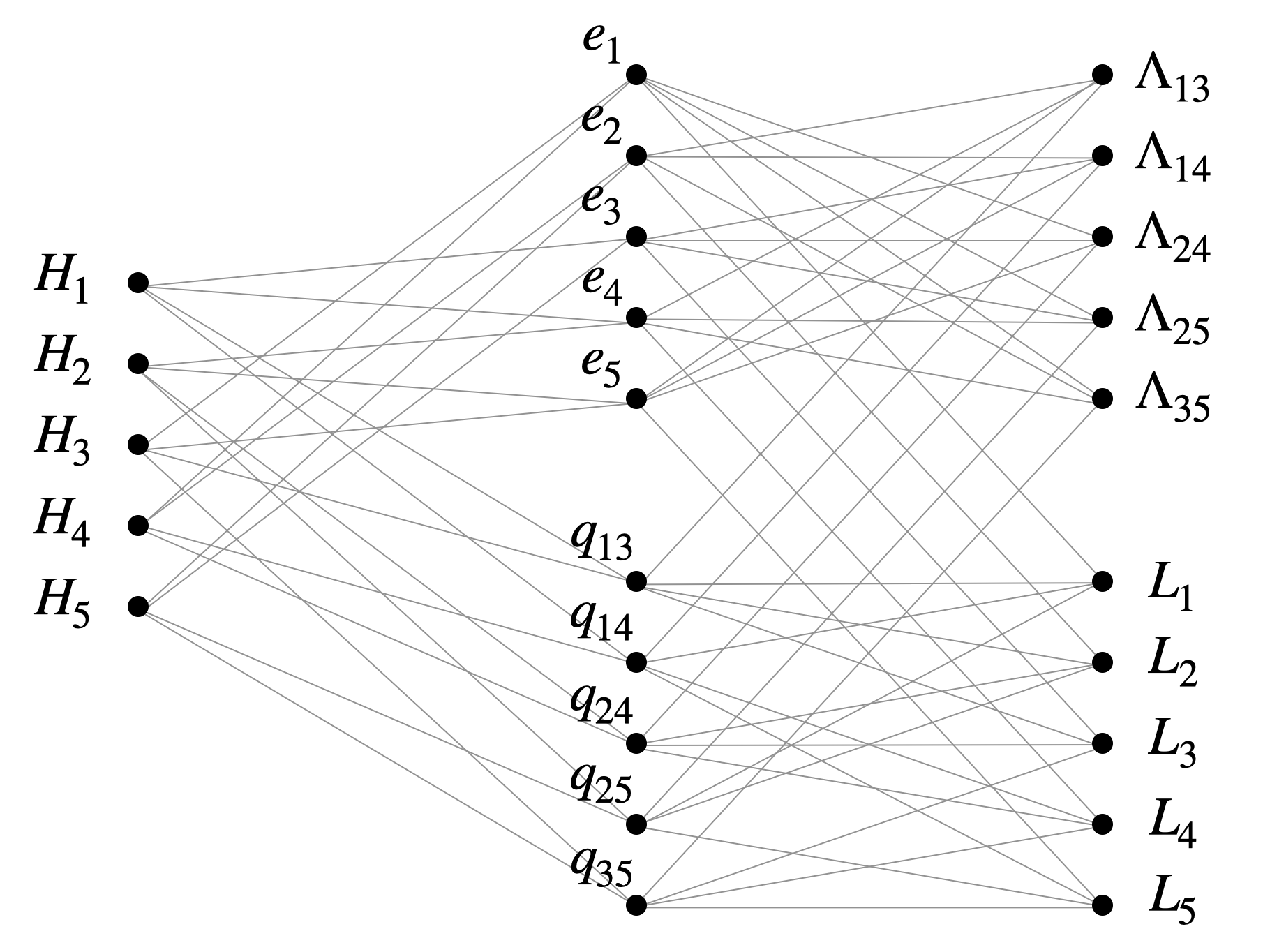} 
\quad \quad 
\includegraphics[height = 4.5cm]{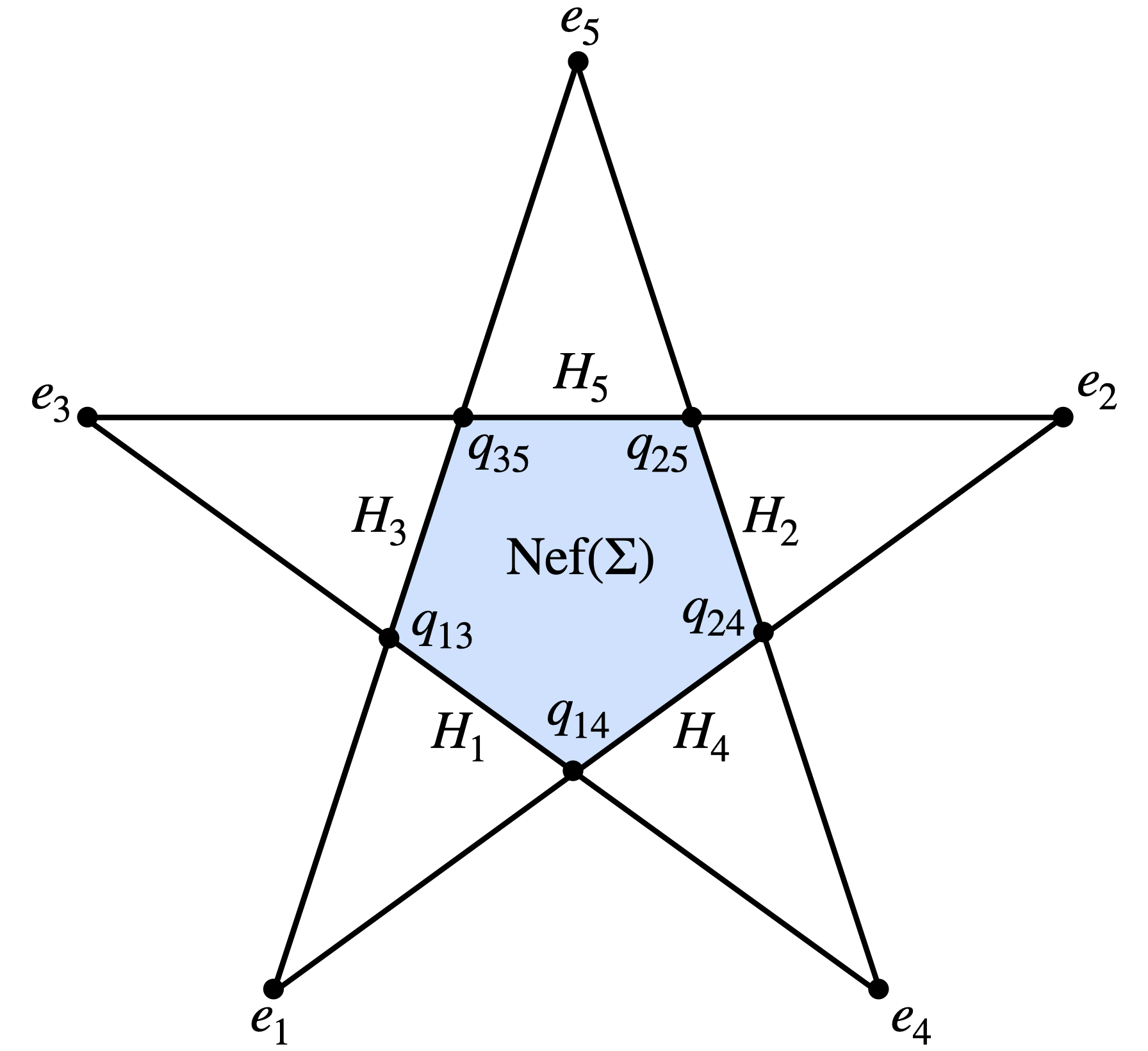}
\caption{The $(15_4,10_6)$ configuration of the universal adjoint of a pentagon.}
\label{fig:config}
\end{figure}
The Fano scheme of planes contained in ${\cal A}_P$ consists of 15 points in ${\rm Gr}(3,5)$, given by $[\Lambda_{i,i+1}], [H_i]$ and $[L_i]$. The Fano scheme of lines in ${\cal A}_P$ is a surface in ${\rm Gr}(2,5)$ with 21 components \cite[Section 4]{Dolgachev}. Out of these, 15 consist of the lines contained in the planes. The other six are degree five del Pezzo surfaces in the Pl\"ucker embedding ${\rm Gr}(2,5) \subset \mathbb{P}^9$.

\subsection{The three-dimensional associahedron} \label{subsec:assoc}
The $d$-dimensional associahedron is a simple convex polytope whose vertices correspond to the triangulations of the $(d+3)$-gon \cite{lee1989associahedron}. The one-dimensional associahedron is a line segment, and the two-dimensional associahedron is a pentagon. As pointed out in the Introduction, the toric amplitude of certain realizations of the associahedron is the bi-adjoint scalar $\phi^3$ amplitude in particle physics. These realizations are called \emph{ABHY associahedra}, after the authors of \cite{arkani2018scattering}. A two-dimensional ABHY associahedron appeared in Example \ref{ex:pentagonintro}. Its singular locus and Fano schemes are described in Section \ref{subsec:pentagon} after setting $u_{ij} = 1$. Here, we analyze a three-dimensional ABHY associahedron, shown in Figure \ref{fig:ABHY3}. 
\begin{figure}[h!]
    \centering
    \includegraphics[height = 3.5cm]{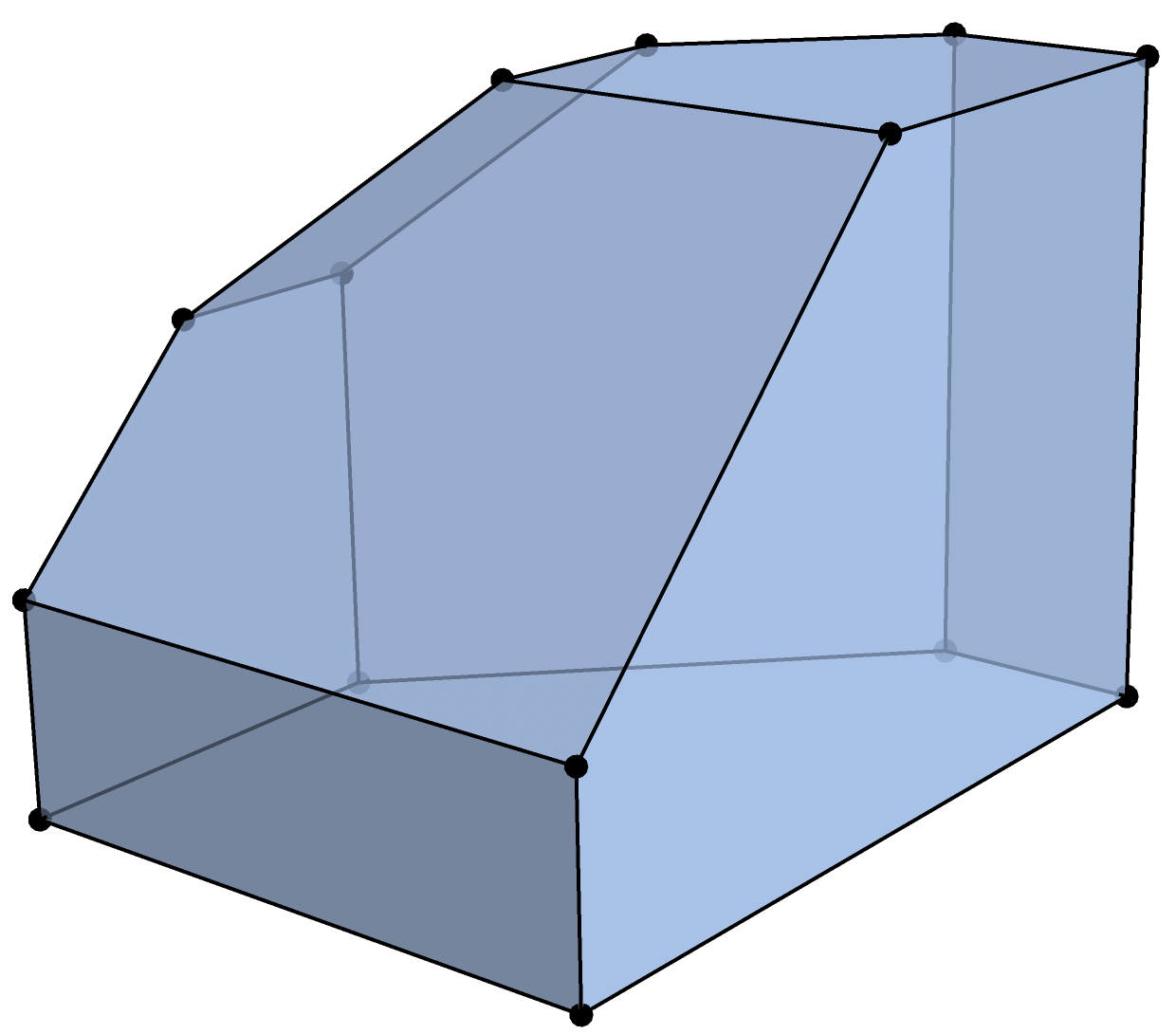}
    \caption{A three-dimensional ABHY associahedron.}
    \label{fig:ABHY3}
\end{figure}
There are nine facets, 21 edges and 14 vertices. 
The nine facet inequalities are $u_i \cdot y + z_i \geq 0$, where $u_i$ are the rows of 
\[ U \, = \, \begin{pmatrix}
    -1 & 0 & 0 & 1 & 1 & 1 & 0 & 0 & 0 \\ 
    0 & -1 & 0 & -1 & 0 & 0 & 1 & 1 & 0 \\ 
    0 & 0 & -1 & 0 & -1 & 0 & -1 & 0 & 1
\end{pmatrix}^t\]
and $z = (3,4,3,2,2,0,1,0,0)$. Let $\Sigma$ be the normal fan of $P$. Each facet of $P$ corresponds to a ray of $\Sigma$, and to a generator of the ring $R_\Sigma = \mathbb{C}[x_{13},x_{14},x_{15},x_{24},x_{25},x_{26},x_{35},x_{36},x_{46}]$. Here the order in which the generators are listed is compatible with the columns of $U$. We use the notation $x_{ij}$ to emphasize that the facet of the variable $x_{ij}$ corresponds to the diagonal $(i,j)$ of the hexagon. The references \cite{arkani2018scattering,arkani2024hidden} use $X_{i,j}$ instead. The pentagonal facets correspond to $x_{13}, x_{24}, x_{35}, x_{46}, x_{15}, x_{26}$, and the quadrilaterals to $x_{14},x_{25},x_{36}$. The toric amplitude ${\rm Amp}_P$ is 
\begin{align*}
    &\frac{1}{x_{15} x_{25} x_{35}}+
 \frac{1}{x_{13} x_{14} x_{46}}+
 \frac{1}{x_{13} x_{14} x_{15}}+
 \frac{1}{x_{13} x_{15} x_{35}}+
 \frac{1}{x_{13} x_{35} x_{36}}+
 \frac{1}{x_{14} x_{15} x_{24}}+
 \frac{1}{x_{13} x_{36} x_{46}}
 \\ + &
 \frac{1}{x_{14} x_{24} x_{46}}+
 \frac{1}{x_{15} x_{24} x_{25}}+
 \frac{1}{x_{24} x_{25} x_{26}}+
 \frac{1}{x_{24} x_{26} x_{46}}+
 \frac{1}{x_{25} x_{26} x_{35}}+
 \frac{1}{x_{26} x_{35} x_{36}}+
 \frac{1}{x_{26} x_{36} x_{46}},
\end{align*}  
where each term is a triangulation of the hexagon. The numerator of this rational function defines the universal adjoint hypersurface ${\cal A}_P \subset \mathbb{P}^8$. It has degree six, and its defining equation ${\rm Adj}_P(x)=0$ has 14 squarefree terms. There are 21 six-planes contained in ${\cal A}_P$. The first 15 of them are coordinate subspaces defined by the following radical monomial ideal: 
\[
\renewcommand\arraystretch{1.2}
\begin{matrix}
B(\Sigma) \, = \, \langle x_{46}, x_{35}\rangle \cap
 \langle x_{46}, x_{25}\rangle \cap
 \langle x_{46}, x_{15}\rangle \cap
 \langle x_{36}, x_{25}\rangle \cap
 \langle x_{36}, x_{24}\rangle \cap
 \langle x_{36}, x_{15}\rangle \cap
 \langle x_{36}, x_{14}\rangle \cap  \\ 
 \langle x_{35}, x_{24}\rangle 
 \cap \langle x_{35}, x_{14}\rangle \cap
 \langle x_{26}, x_{15}\rangle \cap
 \langle x_{26}, x_{14}\rangle \cap
 \langle x_{26}, x_{13}\rangle \cap
 \langle x_{25}, x_{14}\rangle \cap
 \langle x_{25}, x_{13}\rangle \cap
 \langle x_{24}, x_{13}\rangle.
\end{matrix}
\]
This is the Stanley-Reisner ideal of the Alexander dual of $\Sigma$. In toric geometry, $B(\Sigma)$ is the irrelevant ideal in the Cox ring of the normal toric variety $X_\Sigma$ \cite[Chapter 5]{CoxLittleSchenck2011}. Its variety is a union of coordinate subspaces, denoted by $Z(\Sigma)$. By Proposition \ref{prop:containedinB}, we have $Z(\Sigma) \subseteq {\cal A}_\Sigma$. The other six 6-planes in ${\cal A}_P$ come in three pairs; one pair for each quadrilateral facet:
\[  \langle x_{14}, x_{13} + x_{24} \rangle,  
 \langle x_{14}, x_{15} + x_{46} \rangle,  
 \langle x_{25}, x_{24} + x_{36} \rangle, 
 \langle x_{25}, x_{26} + x_{15} \rangle, 
 \langle x_{36}, x_{35} + x_{46} \rangle, 
 \langle x_{36}, x_{13} + x_{26} \rangle. \]
The restriction of ${\rm Adj}_P$ to the coordinate hyperplane $x_{14} = 0$ is, up to a squarefree monomial factor, the universal adjoint of the facet labeled by $x_{14}$, see Lemma \ref{lem:restrictpolytope}. That quadrilateral facet has a degenerate quadratic adjoint which factors as $(x_{13} + x_{24})(x_{15} + x_{46})$ (Section \ref{subsec:quadrilateral}). This explains $\langle x_{14}, x_{13} + x_{24} \rangle,  
 \langle x_{14}, x_{15} + x_{46} \rangle$. The other four are explained in the same way. 

To analyze the singular locus of ${\cal A}_P$, we distinguish between components which are contained in $Z(\Sigma)$ and components which are not. This is motivated in Section \ref{sec:singlocus}. Let $I_{\rm sing} \subset R_\Sigma$ be the ideal generated by the nine partial derivatives of ${\rm Adj}_P$. Let $\mathfrak{p}_i, i = 1, \ldots, 15$ be the 15 minimal primes of $B(\Sigma)$. We compute the primary decomposition of each $I_{\rm sing} + \mathfrak{p}_i$ using a computer algebra system, such as \texttt{Oscar.jl} \cite{OSCAR}. This results in a list of 133 distinct associated primes. One of them defines a five-plane contained in three components of $Z(\Sigma)$: $\langle x_{36},x_{25},x_{14} \rangle$. The other 132 define four-dimensional components of ${\rm Sing}({\cal A}_\Sigma)$. There are 114 four-planes, twelve components of degree three, and six components of degree seven. We shall analyze the combinatorics of this arrangement further in Example \ref{ex:singassoc} using Proposition \ref{prop:singirrel}. 

To find the components of ${\rm Sing}({\cal A}_P)$ which are not contained in $B(\Sigma)$, we compute the saturation $I_{\rm sing}: B(\Sigma)^\infty$ and decompose the result. This gives two planes and three additional four-dimensional components of degree three. All these computations take no more than a few seconds. The code and a list of all components is available at \cite{mathrepo}.

\paragraph{Motivation from physics.} We briefly explain how our construction arises from scattering amplitudes. A central objective in theoretical particle physics is to make predictions for the outcome of collider experiments. For our purpose, one should imagine a total of $m = d+3$ particles entering and exiting the collider. The particles interact or \emph{scatter} inside the collider, and the experiment is called a \emph{scattering process}. 
The \emph{scattering amplitude} ${\rm Amp}(p_1, \ldots, p_m)$ is a function of the momentum vectors $p_1, \ldots, p_m \in \mathbb{R}^{1,D-1}$ associated with each of the particles. Here $D$ is the space-time dimension, which can be an arbitrary positive integer in a theoretical setup. The space $\mathbb{R}^{1,D-1}$ is the $D$-dimensional \emph{Minkowski space}, which is the vector space $\mathbb{R}^D$ endowed with the \emph{Minkowski inner product} $p \cdot q = p_1q_1 - p_2q_2 - \cdots - p_Dq_D$. The squared absolute value of the amplitude is a joint probability distribution describing the scattering process. Computing it analytically is in general an extremely hard task. 

The amplitude function depends on the physical theory governing the scattering process. Our setup is motivated by \emph{biadjoint scalar $\phi^3$-theory with tree-level interactions}, for which the amplitude turns out to be a rational function in the entries of the momentum vectors $p_i$. To compute the amplitude, we must sum over all possible interaction patterns that can happen inside the collider. Such an interaction pattern is conventionally represented by a graph, called \emph{Feynman diagram} \cite[\S 2.2]{WeinzierlFeynmanDiagrams}. Here one should imagine that the in- and outgoing particles, as well as newly created particles inside the collider, travel along the edges of the graph. \emph{Tree level interactions} means that the only graphs which are allowed are trees, and \emph{biadjoint $\phi^3$} restricts us further to trivalent planar trees with $m$ labeled leaves. These are dual to the triangulations of the $m$-gon, which are the vertices of the associahedron of dimension $d = m-3$. A triangulation $T$ uses $d$ diagonals $(i,j)$ of the $m$-gon, each corresponding to a facet of the associahedron. The contribution of the triangulation $T$ to the amplitude is 
\[ \frac{1}{\prod_{(i,j) \in T} x_{ij}}, \quad \text{where} \quad x_{ij} = (p_i + p_{i+1} + \cdots + p_{j-1}) \cdot (p_i + p_{i+1} + \cdots + p_{j-1}). \]
Here $\cdot$ is the Minkowski inner product. This formula is prescribed by the \emph{Feynman rules} \cite[\S 2.3]{WeinzierlFeynmanDiagrams}. The sum of these contributions gives the toric amplitude ${\rm Amp}_P$ of a smooth realization of the associahedron. In particular, one can let $P$ be the ABHY associahedron from \cite{arkani2018scattering}. In fact, the more general \emph{CEGM amplitudes} \cite{cachazo2019scattering} arise as toric amplitudes in a similar manner. 

Figure \ref{fig:scatterm=5} illustrates the above discussion for $m = 5$. The vertices of the $2$-dimensional associahedron are the five triangulations of the pentagon (orange), or the five planar trivalent trees with five labeled leaves (blue).  Its edges are the five diagonals of the pentagon. Relabeling the edge variables $x_1 \rightarrow x_{13}$, \ldots, $x_5 \rightarrow x_{35}$ from Example \ref{ex:pentagonintro} to make this correspondence explicit, we see that \eqref{eq:amppentagon} sums the contribution of each tree to the scattering amplitude. Similarly, the $m = 6$ amplitude is a sum over 14 Feynman diagrams. These are the trivalent trees dual to the 14 triangulations of the hexagon, each corresponding to a vertex in Figure~\ref{fig:ABHY3}.
\begin{figure}
    \centering
    \includegraphics[width=0.65\linewidth]{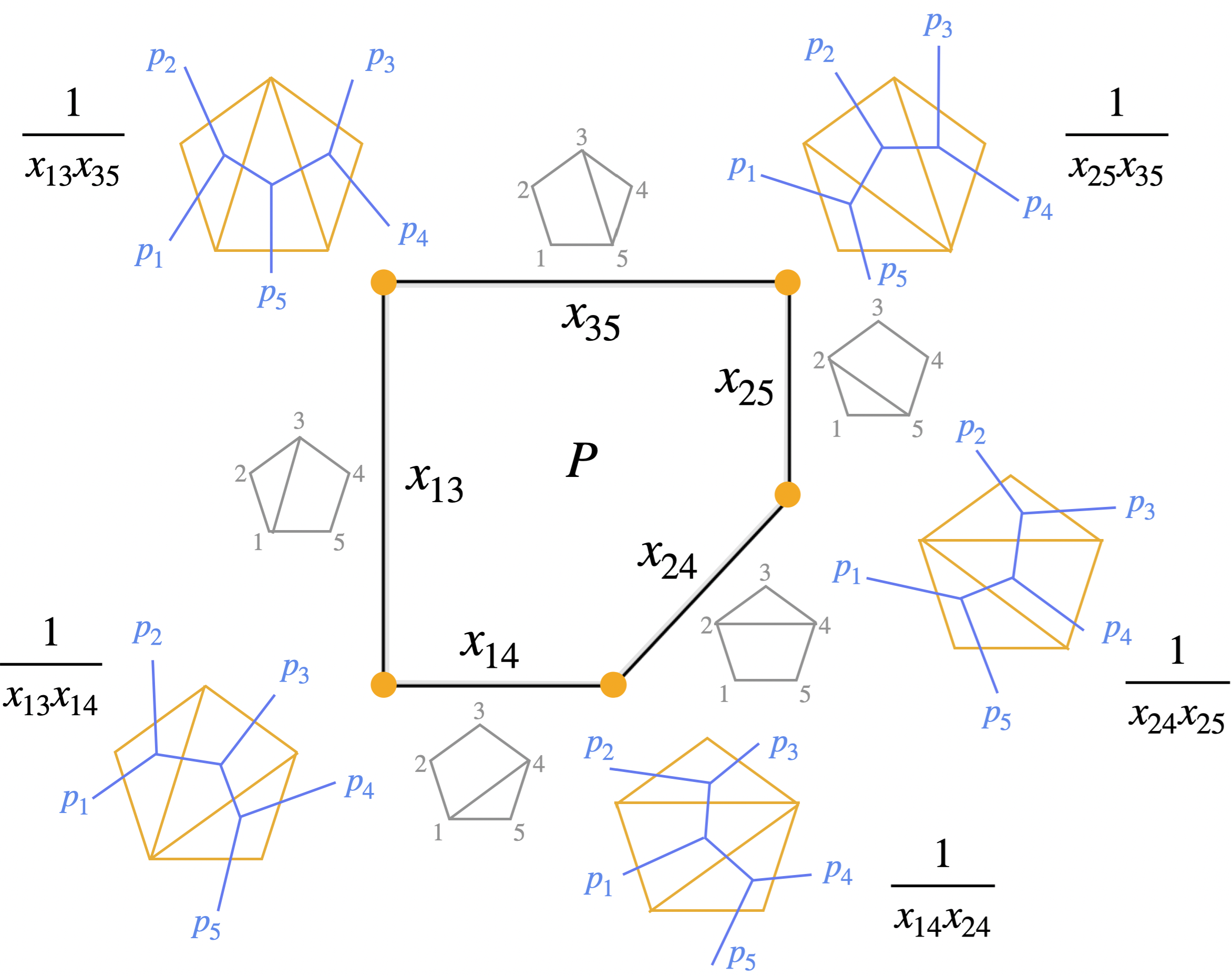}
    \caption{The five terms in \eqref{eq:amppentagon} correspond to the five triangulations of the pentagon.}
    \label{fig:scatterm=5}
\end{figure}

\section{First properties} \label{sec:firstprop}

In this section, we state and illustrate some properties of amplitudes and adjoints. 
Both ${\rm Amp}_\Sigma$ and ${\rm Adj}_\Sigma$ depend on a choice of $U \in \mathbb{R}^{n \times d}$, i.e., a choice of ray generators. The hypersurface ${\cal A}_\Sigma \subset \mathbb{P}^{n-1}$ defined by ${\rm Adj}_\Sigma$ only depends on the ${\rm GL}_d(\mathbb{R})$-orbit of $U$. Scaling the rows of $U$ induces an action of $(\mathbb{R}^\times)^n$ on $R_\Sigma$. This changes our hypersurface, so its singular locus and Fano schemes will depend on the scaling. We shall make this dependence on $U$ explicit when needed by writing ${\rm Amp}_{\Sigma,U}$, ${\rm Adj}_{\Sigma,U}$  and ${\cal A}_{\Sigma,U}$ instead. Here are some more examples. 

\begin{example} \label{ex:simpleex}
If $\Sigma$ has no full-dimensional cones, then ${\rm Adj}_{\Sigma,U} = 0$. 
If $\Sigma$ consists of a $d$-dimensional simplicial cone $\sigma$ and all its faces, then ${\rm Adj}_{\Sigma,U} = |\det U|$. If $P \subset \mathbb{R}^d$ is a $d$-dimensional simplex, i.e., $\Sigma_P$ is complete with $d+1$ rays, then ${\rm Adj}_{\Sigma,U}$ is a linear form in $x_\rho, \rho \in \Sigma(1)$, given by the $(d+1) \times (d+1)$ determinant ${\rm Adj}_{\Sigma,U} =  \pm \det \begin{pmatrix}
    U & x
\end{pmatrix}$. 
\end{example}

A first simple observation is that amplitudes and adjoints behave nicely under taking products. Let $\Sigma_i$ be a simplicial fan in $\mathbb{R}^{d_i}$ for $i = 1,2$. The ray matrices are $U_1 \in \mathbb{R}^{n_1 \times d_1}$ and $U_2 \in \mathbb{R}^{n_2 \times d_2}$. The product fan $\Sigma = \Sigma_1 \times \Sigma_2$ (see \cite[Proposition 3.1.14]{CoxLittleSchenck2011}) has ray matrix 
\[ U_1 \oplus U_2 \, = \, \begin{pmatrix}
    U_1 & 0 \\ 0 & U_2
\end{pmatrix} \, \, \in \, \mathbb{R}^{(n_1+n_2) \times (d_1 + d_2)}.\]
The universal adjoint ${\rm Adj}_{\Sigma_i, U_i}$ is an element of the polynomial ring $R_{\Sigma_i}$ with $n_i$ variables. In the following lemma, this is naturally viewed as a subring of $R_\Sigma = R_{\Sigma_1} \otimes_{\mathbb{C}} R_{\Sigma_2}$. 
\begin{lemma} \label{lem:factorfans}
    The toric amplitude of the product fan $\Sigma = \Sigma_1 \times \Sigma_2$ satisfies ${\rm Amp}_{\Sigma, U_1 \oplus U_2} = {\rm Amp}_{\Sigma_1,U_1} \cdot {\rm Amp}_{\Sigma_2,U_2}$ and we have ${\rm Adj}_{\Sigma,U_1 \oplus U_2} = {\rm Adj}_{\Sigma_1,U_1} \cdot {\rm Adj}_{\Sigma_2,U_2}$. For polytopes $P_1,P_2$ we have ${\rm Amp}_{P_1 \times P_2,U_1 \oplus U_2} = {\rm Amp}_{P_1,U_1} \cdot {\rm Amp}_{P_2,U_2}$ and ${\rm Adj}_{P_1 \times P_2,U_1 \oplus U_2} = {\rm Adj}_{P_1,U_1} \cdot {\rm Adj}_{P_2,U_2}$.
\end{lemma}
\begin{proof}
    This follows easily from the fact that $\Sigma(d_1+d_2) = \{ \sigma_1 \times \sigma_2 \, : \, \sigma_i \in \Sigma_i(d_i)\}$. 
\end{proof}

\begin{example}
    The normal fan of the cube $P \subset \mathbb{R}^3$ has the following matrix of ray generators: 
    \[ U \, = \, \begin{pmatrix}
        e_1 & -e_1 & e_2 & -e_2 & e_3 & -e_3
    \end{pmatrix}^t \, = \, \begin{pmatrix}
        1 & -1 
    \end{pmatrix}^t \oplus \begin{pmatrix}
        1 & -1 
    \end{pmatrix}^t \oplus \begin{pmatrix}
        1 & -1 
    \end{pmatrix}^t \, \, \in \, \mathbb{R}^{6 \times 3}.\]
    The toric amplitude factors as ${\rm Amp}_{P} = (\frac{1}{x_1} + \frac{1}{x_2}) \cdot (\frac{1}{x_3} + \frac{1}{x_4}) \cdot (\frac{1}{x_5} + \frac{1}{x_6}) $ and, similarly, the universal adjoint is the reducible cubic polynomial ${\rm Adj}_P = (x_1 + x_2) \cdot (x_3 + x_4) \cdot (x_5 + x_6)$.
\end{example}

Next, we study restrictions of ${\rm Adj}_\Sigma$ to coordinate subspaces. 
For any $k$-dimensional cone $\tau \in \Sigma(k)$, let $\Sigma_\tau$ denote the \emph{star fan} of $\Sigma$ at $\tau$. This is the simplicial fan given by 
\[ \Sigma_\tau \, = \, \{ \overline{\sigma} \subseteq \mathbb{R}^d/{\rm span}_{\mathbb{R}}(\tau) \, : \, \text{$\tau$ is a face of $\sigma \in \Sigma$} \},\]
where $\overline{\sigma}$ is the image of $\sigma$ under the projection $\mathbb{R}^d \rightarrow \mathbb{R}^d/{\rm span}_{\mathbb{R}}(\tau) \simeq \mathbb{R}^{d-k}$. We define 
\begin{equation} \label{eq:nb}
    {\rm nb}(\tau) \, = \, \{ \rho \in \Sigma(1)  \, : \, \text{ $\tau(1) \cup \{ \rho \} = \sigma(1)$ for some $\sigma \in \Sigma(k+1)$} \}.
\end{equation}
Here nb stands for ``neighbors''. Projecting along ${\rm span}_{\mathbb{R}}(\tau)$ gives a one-to-one correspondence 
\[ {\rm nb}(\tau) \, \overset{1:1}{\longleftrightarrow} \,  \Sigma_\tau(1) \,  = \, \{ \overline{\rho} \, : \, \rho \in {\rm nb}(\tau)  \} .\] 

We compute a ray matrix $U_\tau$ for $\Sigma_\tau$ as follows. 
We replace $U$ by $U\cdot T_\tau$, where $T_\tau$ is an invertible $d\times d$ matrix such that the rows of $U \cdot T_\tau$ labeled by $\tau(1)$ are $e_1, \ldots, e_k$. The matrix $T_\tau$ determines a change of coordinates in $\mathbb{R}^d$. Let $U_\tau$ be the submatrix of $U \cdot T_\tau$ consisting of the last $d-k$ columns and the rows indexed by ${\rm nb}(\tau)$. Notice that, for each $\sigma \in \Sigma(d)$ such that $\tau \subseteq \sigma$, we have $|\det (U_\tau)_{\overline{\sigma}}| = c_\tau \cdot  |\det U_\sigma|$, where $c_\tau = |\det T_\tau|$. 
The adjoint ${\rm Adj}_{\Sigma_\tau, U_\tau}$ is a polynomial in $ R_{\Sigma_\tau} = \mathbb{C}[x_{\overline{\rho}} : \rho \in {\rm nb}(\tau)]$, which we view as a subring of $R_\Sigma$. Define 
\begin{equation} \label{eq:Lambdatau} \Lambda_{\tau} \, = \, V(x_\rho \, : \, \rho \in \tau(1) ) \, = \, \{ x \in \mathbb{P}^{n-1} \, : \, x_\rho = 0 \text{ for all } \rho \in \tau(1) \}.\end{equation}

\begin{lemma} \label{lem:restrict}
    The restriction of ${\rm Adj}_{\Sigma,U}(x)$ to the $(n-k-1)$-dimensional subspace $\Lambda_\tau$ equals 
    \[ ({\rm Adj}_{\Sigma,U})_{|\Lambda_\tau} \, = \, c_\tau^{-1} \cdot \Big ( \prod_{\substack{\rho \notin {\rm nb}(\tau) \\ \rho \not \in \tau(1)} }  x_{\rho} \Big ) \cdot {\rm Adj}_{\Sigma_\tau,U_\tau}. \]
\end{lemma}
\begin{proof} 
By the definition of ${\rm Adj}_{\Sigma,U}$, we have that 
    \[ ({\rm Adj}_{\Sigma,U})_{|\Lambda_\tau} \, = \,  \Big ( \prod_{\substack{\rho \notin {\rm nb}(\tau) \\ \rho \not \in \tau(1)} } x_{\rho} \Big ) \cdot \Big (\sum_{\substack{ \sigma \in \Sigma(d) \\ \tau \subseteq \sigma} } | \det U_\sigma | \cdot \prod_{\substack{\rho' \in {\rm nb}(\tau)\\ \rho' \not \in \sigma(1)}} x_{\rho'} \Big ). \]
The product in the first factor is over $\Sigma(1) \setminus ({\rm nb}(\tau) \cup \tau(1))$, and the sum in the second factor is over all $\sigma \in \Sigma(d)$ which contain $\tau$ as a face. By construction, the minors satisfy $|\det (U_\tau)_{\overline{\sigma}}| = c_\tau \cdot  |\det U_\sigma|$, for each $\sigma \in \Sigma(d)$ such that $\tau \subseteq \sigma$. Hence, the second factor is 
\[ \sum_{\substack{ \sigma \in \Sigma(d) \\ \tau \subseteq \sigma} } | \det U_\sigma | \cdot \prod_{\substack{\rho' \in {\rm nb}(\tau)\\ \rho' \not \in \sigma(1)}} x_{\rho'}  \, = \, c_\tau^{-1} \cdot \sum_{\overline{\sigma} \in \Sigma_\tau(d-k)} | \det (U_\tau)_{\overline{\sigma}} | \cdot \prod_{ \overline{\rho'} \not \in \overline{\sigma}(1)} x_{\rho'} \, = \, c_\tau^{-1} \cdot {\rm Adj}_{\Sigma_\tau,U_\tau}. \qedhere\]
\end{proof}

\begin{example} \label{ex:square}
The fan $\Sigma$ in the left part of Figure \ref{fig:fansSec2}, with ray matrix $U = \left ( \begin{smallmatrix}
    1 & 0 & -1 & 0 \\ 0 & 1 & 0 & -1
\end{smallmatrix}\right )^t$, has the following universal adjoint: ${\rm Adj}_{\Sigma,U} = x_3x_4 + x_1x_4 + x_1x_2 + x_2x_3$. This depends on the choice of $U$ and on the fan structure: setting $U' = \left (\begin{smallmatrix}
    3 & 0 & -1 & 0 \\ 0 & 1 & 0 & -1
\end{smallmatrix}\right)^t $ gives ${\rm Adj}_{\Sigma,U'} = 3\, x_3x_4 + x_1x_4 + x_1x_2 + 3 \, x_2x_3$, and using the fan $\Sigma'$ in the middle of Figure \ref{fig:fansSec2} instead of $\Sigma$ gives ${\rm Adj}_{\Sigma',U} = x_3x_4 + x_1x_2$. 

Setting $x_1 = 0$ in ${\rm Adj}_{\Sigma,U}$ gives $x_3x_4 + x_2x_3 = x_3(x_2+x_4)$. We match this with Lemma \ref{lem:restrict}. We have ${\rm nb}(\rho_1) = \{ \rho_2, \rho_4 \}$. Since the first row of $U$ is $e_1 = (1,0)$, we can use $T_{\rho_1} = {\rm id}_{2 \times 2}$ and $c_{\rho_1} = 1$. The star fan $\Sigma_{\rho_1}$ of $\Sigma$ at $\rho_1$ is the complete fan in $\mathbb{R}^1$, whose ray matrix $U_{\rho_1}$ is a submatrix of the second column of $U$. The corresponding adjoint is ${\rm Adj}_{\Sigma_{\rho_1},U_{\rho_1}} = x_2 + x_4$. 

Setting $x_1 = 0$ in ${\rm Adj}_{\Sigma',U}$ gives $x_3x_4$. Rays 3 and 4 do not belong to ${\rm nb}(\rho_1)$. The star fan of $\Sigma'$ at $\rho_1$ consists of a ray and its face $\{0\}$. Its universal adjoint is the constant $1$. 
\begin{figure}
        \centering
        \includegraphics[height= 3.0cm]{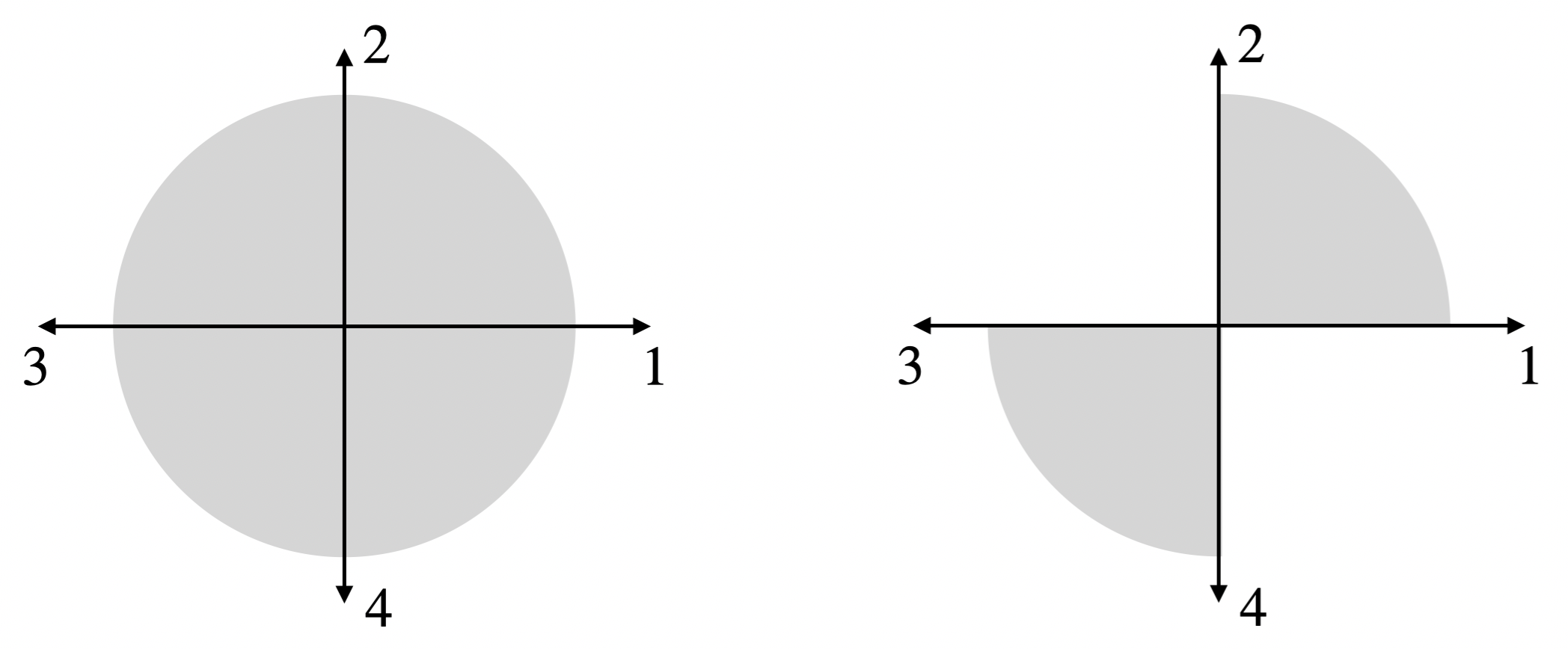}
        \quad \quad 
        \includegraphics[height = 3.0cm]{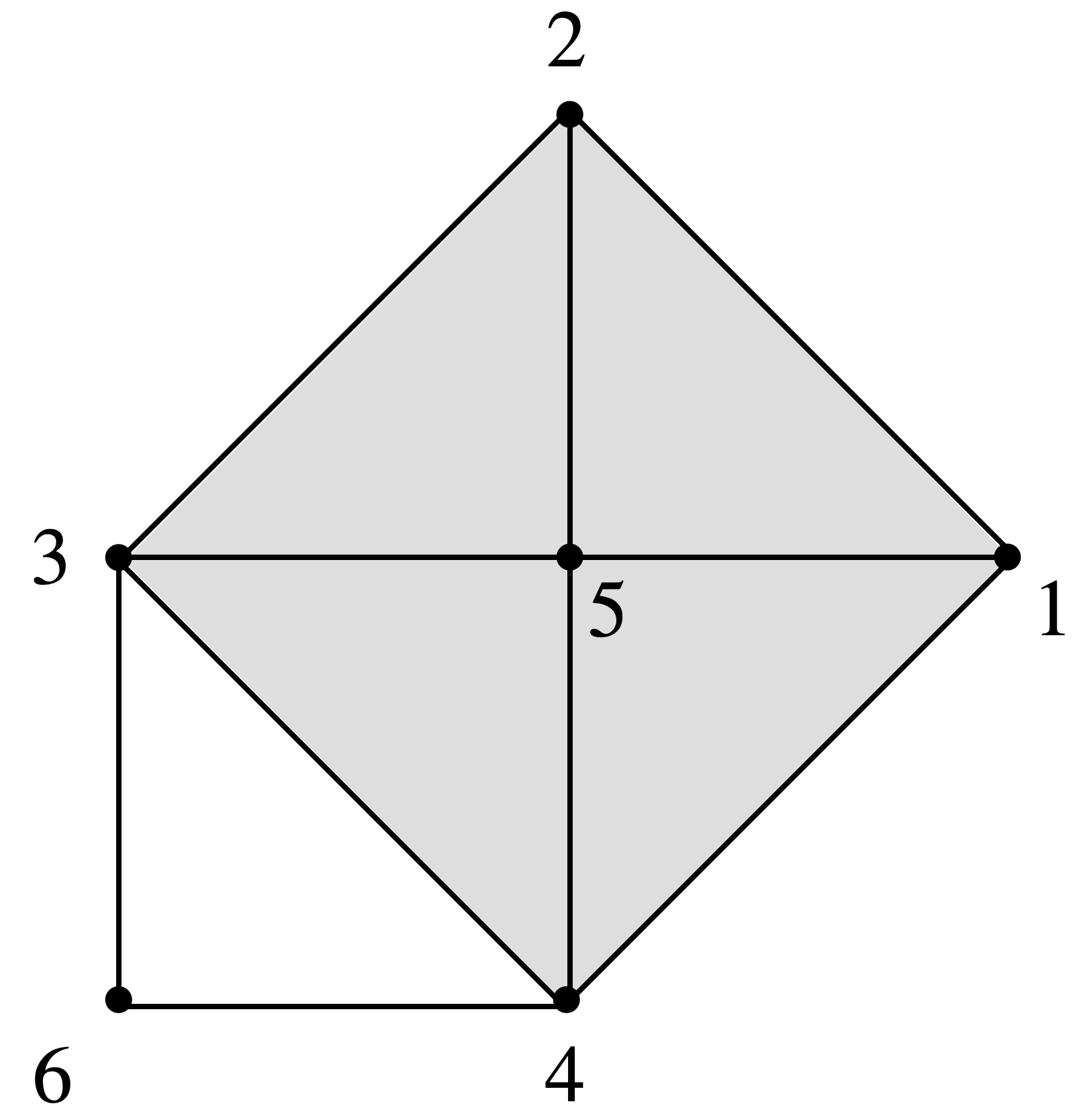}
        \caption{Three simplicial fans.}
        \label{fig:fansSec2}
    \end{figure}
\end{example}

\begin{example} \label{ex:noncomplete}
    We consider the fan $\Sigma$ in $\mathbb{R}^3$ obtained by taking the cone over the right part of Figure \ref{fig:fansSec2}. It has four three-dimensional cones, ten two-dimensional cones and six rays: 
    \[ U \, = \, \begin{pmatrix}
        e_1 + e_2 & e_1 + e_3 & e_1 - e_2 & e_1 - e_3 & e_1 & e_1 - e_2 - e_3
    \end{pmatrix}^t \, \, \in \mathbb{R}^{6 \times 3}. \]
    The universal adjoint ${\rm Adj}_\Sigma = {\rm Adj}_{\Sigma,U}$ is a cubic polynomial in six variables with four terms: 
    \begin{equation} \label{eq:AdjIncomplete} {\rm Adj}_\Sigma \, = \, x_3x_4x_6 + x_1x_4x_6 + x_1x_2x_6 + x_2x_3x_6.\end{equation}
    Restricting to $x_6 = 0$ gives the zero polynomial. This agrees with the fact that the star fan $\Sigma_{\rho_6}$ has no full-dimensional cones (see Example \ref{ex:simpleex}). Setting $x_5 = 0$, we obtain 
    \[ ({\rm Adj}_\Sigma)_{|x_5 = 0} \, = \, {\rm Adj}_\Sigma \, = \, x_6 \cdot (x_3x_4 + x_1x_4 + x_1x_2 + x_2x_3).\]
    The factor $x_6$ is explained by $\Sigma(1) \setminus {\rm nb}(\rho_5) = \{\rho_5, \rho_6\}$, and the quadratic factor is the universal adjoint of the star fan of $\Sigma$ at $\rho_5$, which is the fan in the left part of Figure \ref{fig:fansSec2}.
\end{example}
Lef $f \in \mathbb{C}(x_\rho  : \rho \in \Sigma(1))$ be a rational function with poles of order at most one along the coordinate hyperplanes. We define the \emph{residue} of $f$ along $\Lambda_\tau$ as
\begin{equation} \label{eq:residue} {\rm res}_{\Lambda_\tau} f \, = \, \Big ( \Big (\prod_{\rho \in \tau(1)} x_\rho \Big ) \cdot f \Big )_{|\Lambda_\tau} \, \in \, \mathbb{C}(x_\rho \, : \, \rho \in \Sigma(1) \setminus \tau(1)).\end{equation}
Lemma \ref{lem:restrict} has the following easy consequence for the residues of the toric amplitude.
\begin{corollary} \label{cor:resamp}
    The residue of the toric amplitude ${\rm Amp}_{\Sigma,U}(x)$ along $\Lambda_\tau$ equals 
    \[ {\rm res}_{\Lambda_\tau} \, {\rm Amp}_{\Sigma,U} \, = \, c_\tau^{-1} \cdot {\rm Amp}_{\Sigma_\tau,U_\tau}. \]
\end{corollary}

We rephrase Lemma \ref{lem:restrict} in the important case where $\Sigma = \Sigma_P$ is the normal fan of a simple polytope $P$. The $x$-variables are now indexed by facets $Q$ of $P$. Let $\Delta \subseteq P$ be a face of codimension $k$ and let $T_\Delta \in \mathbb{R}^{d \times d}$ be such that the rows of $U \cdot T_{\Delta}$ indexed by the facets of $P$ containing ${\Delta}$ are $e_1, \ldots, e_k$. The matrix $U_{\Delta}$ consists of the last $d-k$ columns of $U \cdot T_{\Delta}$ and the rows indexed by facets $Q$ for which $Q \cap \Delta$ is a facet of ${\Delta}$. With this definition, the rows of $U_{\Delta}$ generate the rays of $\Sigma_{\Delta}$, the normal fan of ${\Delta}$ in the coordinates defined by $T_{\Delta}$. We set $c_{\Delta} = |\det (T_{\Delta})|$. We shall restrict to the linear space 
\begin{equation} \label{eq:LambdaDelta} \Lambda_{\Delta} \, = \, V(x_Q \, : \, {\Delta} \subseteq Q) = \{ x \in \mathbb{P}^{n-1}\, : \, x_Q = 0 \text{ for all } Q \supseteq {\Delta} \}\end{equation}
and define a residue in analogy with \eqref{eq:residue} as follows: ${\rm res}_{\Lambda_{\Delta}} \,  f = ((\prod_{ Q \supseteq {\Delta}} x_Q) \cdot f)_{|\Lambda_\Delta}$. 

\begin{lemma} \label{lem:restrictpolytope}
    The restriction of ${\rm Adj}_{P,U}(x)$ to the coordinate subspace $\Lambda_{\Delta}$ equals 
    \[ ({\rm Adj}_{P,U})_{|\Lambda_{\Delta}} \, = \, c_{\Delta}^{-1} \cdot \Big ( \prod_{Q \cap {\Delta} = \emptyset } x_{Q} \Big ) \cdot {\rm Adj}_{{\Delta},U_{\Delta}}. \]
    Here ${\rm Adj}_{{\Delta},U_{\Delta}} = {\rm Adj}_{\Sigma_{\Delta},U_{\Delta}}$. Moreover, we have ${\rm res}_{\Lambda_{\Delta}} \, {\rm Amp}_{P,U} \, = \, c_{\Delta}^{-1} \cdot {\rm Amp}_{{\Delta},U_{\Delta}}$.
\end{lemma}

\begin{example}
    Let $P \subset \mathbb{R}^3$ be the ABHY associahedron from Section \ref{subsec:assoc}. Let $\Delta$ be the facet corresponding to the variable $x_{13}$. That is, $\Delta$ is a pentagon. We find that 
    \[ {\rm res}_{\Lambda_\Delta} \, {\rm Amp}_P \, = \, ( x_{13} \cdot {\rm Amp}_P )_{|x_{13} = 0} \,= \, \frac{1}{x_{14}x_{46}} + \frac{1}{x_{14}x_{15}} + \frac{1}{x_{15}x_{35}} + \frac{1}{x_{35}x_{36}} + \frac{1}{x_{36}x_{46}},  \]
    which is the pentagonal amplitude from \eqref{eq:amppentagon}. The surviving terms correspond to the five triangulations of the hexagon which use the diagonal $(1,3)$. For the facet of $x_{14}$ we have 
    \[ ( x_{14} \cdot {\rm Amp}_P )_{|x_{14} = 0} \,= \,  \frac{1}{x_{13}x_{46}} + \frac{1}{x_{13}x_{15}} + \frac{1}{x_{15}x_{24}} + \frac{1}{x_{24}x_{46}} \, = \, \Big ( \frac{1}{x_{13}} + \frac{1}{x_{24}} \Big) \Big(  \frac{1}{x_{15}} + \frac{1}{x_{46}} \Big). \]
    This is the amplitude of a degenerate quadrilateral; it factors because of Lemma \ref{lem:factorfans}.
\end{example}

We can use Lemma \ref{lem:restrict} to show that, if $\Sigma$ is complete, then the polynomial ${\rm Adj}_{\Sigma,U}(x)$ vanishes at any point in the column span of $U$. In the language of projective geometry, the hypersurface ${\cal A}_{\Sigma,U} = \{ x \in \mathbb{P}^{n-1} \, : \, {\rm Adj}_{\Sigma,U}(x) = 0 \}$ contains the $(d-1)$-dimensional linear space spanned by the $d$ columns of $U$. We denote this linear space by $\mathbb{P}({\rm im}(U)) \subset \mathbb{P}^{n-1}$.

\begin{theorem} \label{thm:containsU}
    If $\Sigma$ is a complete fan in $\mathbb{R}^d$, then the universal adjoint ${\rm Adj}_{\Sigma,U}$ vanishes on the $(d-1)$-dimensional linear space $\mathbb{P}({\rm im}(U)) \subset \mathbb{P}^{n-1}$. In particular, if $\Sigma = \Sigma_P$ for a $d$-dimensional simple polytope $P \subset \mathbb{R}^d$, then we have $\mathbb{P}({\rm im}(U)) \subseteq {\cal A}_P$.
\end{theorem}
\begin{proof}
    We prove this by induction on the dimension $d$. For $d = 1$, $\Sigma$ is the normal fan of a line segment with ray matrix $U = \begin{pmatrix}
        a & -b
    \end{pmatrix}^t$ for some positive real numbers $a, b$. The adjoint is ${\rm Adj}_{\Sigma,U} = a \, x_2 + b \, x_1$, which vanishes on $\mathbb{P}({\rm im}(U))$. Suppose that the theorem holds in dimension $d-1$. Fix any ray $\rho \in \Sigma(1)$. As above, we let $T_\rho$ be an invertible matrix such that the row of $U \cdot T_\rho$ labeled by $\rho$ is $e_1$. By Lemma \ref{lem:restrict} and the induction hypothesis, ${\rm Adj}_{\Sigma,U}$ vanishes on the $(d-2)$-dimensional linear subspace of $\mathbb{P}({\rm im}(U))$ spanned by the last $d-1$ columns of $U \cdot T_\rho$, and this is true for each $\rho \in \Sigma(1)$. If no two rays of $\Sigma$ have the same $\mathbb{R}$-span, then this implies that the restriction of ${\rm Adj}_{\Sigma,U}$ to $\mathbb{P}({\rm im}(U))$ vanishes on $n$ different hyperplanes in $\mathbb{P}({\rm im}(U))$. However, ${\rm Adj}_{\Sigma,U}$ has degree $n-d<n$, which implies that ${\rm Adj}_{\Sigma,U}$ has to vanish identically on $\mathbb{P}({\rm im}(U))$. If two rays are linearly dependent, then we can construct a family of fans $\Sigma_\epsilon$ with ray matrices $U_\epsilon $ such that $\lim_{\epsilon \to 0} U_\epsilon = U$, $\Sigma_\epsilon$ has pairwise independent rays and $\Sigma_\epsilon$ has the same combinatorial type as $\Sigma$ for $\epsilon \neq 0$. We have shown that ${\rm Adj}_{\Sigma_\epsilon, U_\epsilon}$ vanishes on ${\rm im}(U_\epsilon)$ for $\epsilon \neq 0$. By continuity, ${\rm Adj}_{\Sigma,U}$ vanishes on $\mathbb{P}({\rm im}(U))$. 
\end{proof}

\begin{example}
    The following fan $\Sigma$ in $\mathbb{R}^3$ is taken from \cite[page 71]{fulton1993introduction}. It is smooth and complete but not projective, i.e., it is not the normal fan of a polytope. The ray matrix is 
    \[ U \, = \,  \begin{pmatrix}
        -e_1 & -e_2 & -e_3 & e_1+e_2+e_3 & e_1 + e_2 & e_2 + e_3 & e_1 + e_3
    \end{pmatrix}^t \, \in \, \mathbb{R}^{7 \times 3} . \]
    There are ten cones in $\Sigma(3)$: 467, 457, 456, 357, 267, 237, 156, 135, 126, and 123. The adjoint 
    \[{\rm Adj}_\Sigma  \, = \, x_1x_2x_3x_5 + x_1x_2x_3x_6 + x_1x_2x_3x_7 + \,  \ldots \, + x_2x_4x_6x_7 + x_3x_4x_5x_7 + x_4x_5x_6x_7\]
    is checked to vanish identically on the plane $\mathbb{P}({\rm im}(U)) \subset \mathbb{P}^6$. See the code available at \cite{mathrepo}.
\end{example}

\section{Dual volumes} \label{sec:dualvolume}

For $U \in \mathbb{R}^{n \times d}$, consider the polyhedron $P_x$ with the following facet representation: 
\begin{equation} \label{eq:Px} P_x \, = \, \{ y \in \mathbb{R}^d \, : \, U\, y + x \, \geq \, 0 \}, \quad x \in \mathbb{R}^{n}. \end{equation}
Here $U \, y + x \geq 0$ means $u_\rho \cdot y + x_\rho \geq 0$ for each row $u_\rho$ of $U$ and corresponding entry $x_\rho$ of $x$. 
The normal fan of $P_x$ depends on $x$. Its rays are among the rows of $U$. The different normal fans obtained by varying $x$ are indexed by cones in the \emph{chamber complex} ${\rm Ch}(U)$. That is, ${\rm Ch}(U)$ is a fan in $\mathbb{R}^n$ whose cones $C$ are such that, for each $x \in {\rm relint}(C)$, the polyhedron $P_x$ has the same normal fan. We recall the construction. For $I \subseteq \{1, \ldots, n\}$, let $\tilde{C}_I \subseteq \mathbb{R}^n$ be the cone generated by the standard basis vectors $\{e_i \, : \, i \in I\}$, and let $C_I = \tilde{C}_I + {\rm im}(U)$, where ${\rm im}(U) \simeq \mathbb{R}^d$ is the column span of $U$. For each point $x \in \mathbb{R}^n$, let $C_x$ be the intersection of all cones $C_I$ containing $x$. If $x \notin C_I$ for all $I$, then $C_x = \emptyset$. The chamber complex ${\rm Ch}(U)= \{ C_x \, : \, x \in \mathbb{R}^n\}$ is the set of cones obtained in this manner. It is a fan with lineality space ${\rm im}(U)$.
For each cone $C \in {\rm Ch}(U)$, let $\Sigma_C$ be the normal fan of $P_x$ for $x \in {\rm relint}(C)$ and let $U_C$ be the submatrix of $U$ consisting of the rows labeled by rays of $\Sigma_C$. If $\dim(C) = n$, then $P_x$ is $d$-dimensional and simple for $x \in {\rm int}(C)$, and $\Sigma_C$ is simplicial. 

\begin{example} \label{ex:chambercomplex}
    Consider a pentagon $P \subset \mathbb{R}^2$, each of whose interior angles is greater than $90^\circ$. The rows of $U \in \mathbb{R}^{5 \times 2}$ generate the rays of its normal fan $\Sigma$, see Figure \ref{fig:chambercomplexpentagon} (right). Modulo the two-dimensional subspace ${\rm im}(U)$ generated by the columns of $U$, the chamber complex is a collection of pointed cones in $\mathbb{R}^3$. It is obtained by taking the cone over Figure \ref{fig:chambercomplexpentagon} (left). The pointed fan ${\rm Ch}(U)/{\rm im}(U)$ has eleven three-dimensional cones, twenty two-dimensional cones and ten rays. The rays labeled by $e_1, \ldots, e_5$ are the images of the standard basis vectors under the quotient by ${\rm im}(U)$. The grey polygons inside each cell show the combinatorial type of $P_x$ for $x$ in that cell. As $x$ moves in the chamber complex, the edge lines of $P_x$ are translated in the direction of their normal vectors. For instance, if $x \in \mathbb{R}^5$ is such that $x \, {\rm mod} \,  {\rm im}(U)$ lies in the central pentagonal cell, shaded in blue, then $P_x = \{ y \in \mathbb{R}^2 \, : \, U \, y + x \geq 0 \}$ is a pentagon. As $x$ moves from the central cell, across the cone generated by rays $e_1$ and $e_2$, into the triangular cell adjacent to ray $e_4$, the edge line of $P_x$ corresponding to the fourth row of $U$ is translated in such a way that $P_x$ becomes a quadrilateral. For each $C \in {\rm Ch}(U)$, the cone ${\rm int}(C) \cap \mathbb{R}^5_+$ represents polygons with fixed normal fan containing $0$ in their interior. 
    \begin{figure}
        \centering
        \includegraphics[height = 5.0cm]{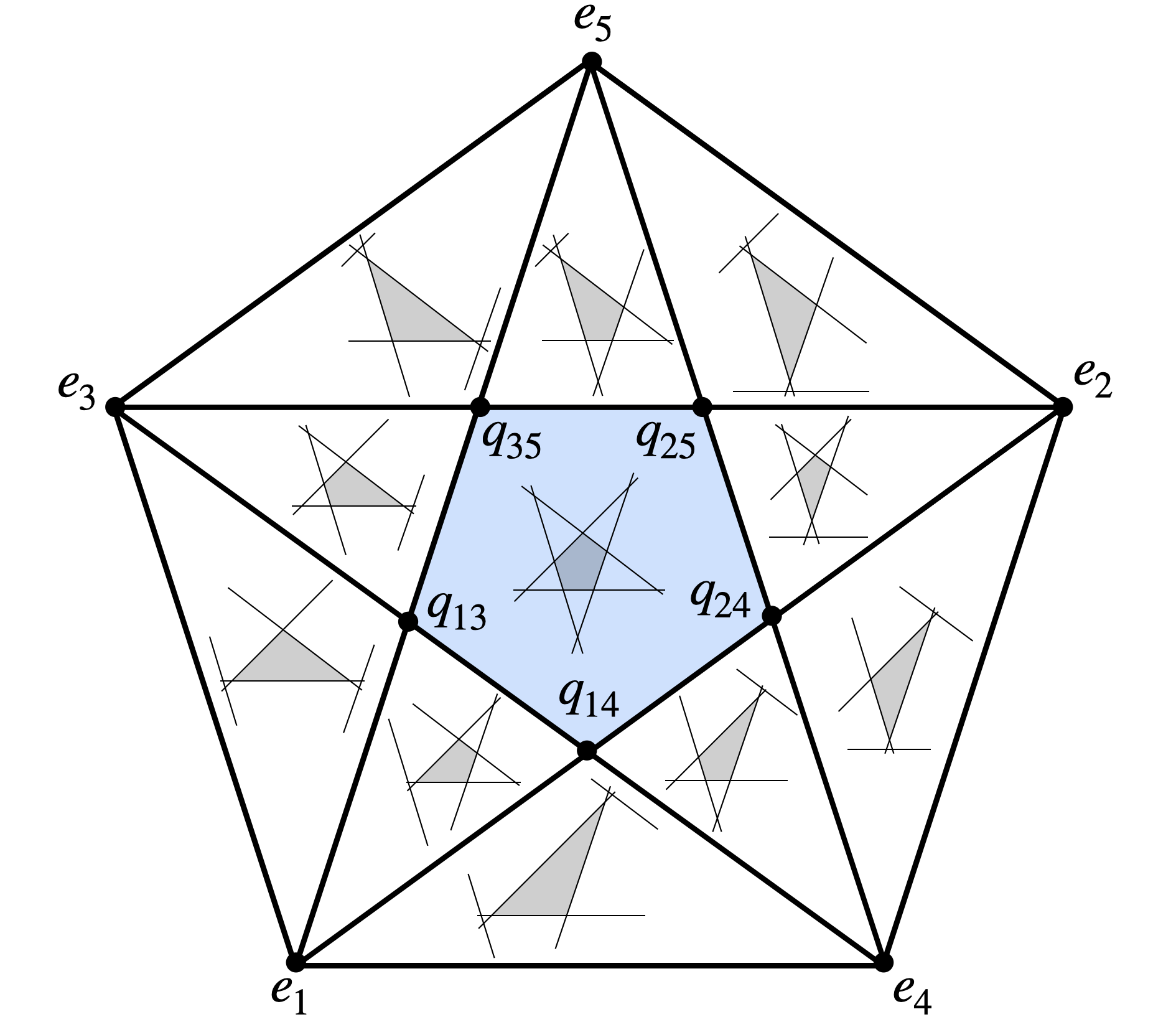}
        \quad \quad 
        \includegraphics[height = 5.0cm]{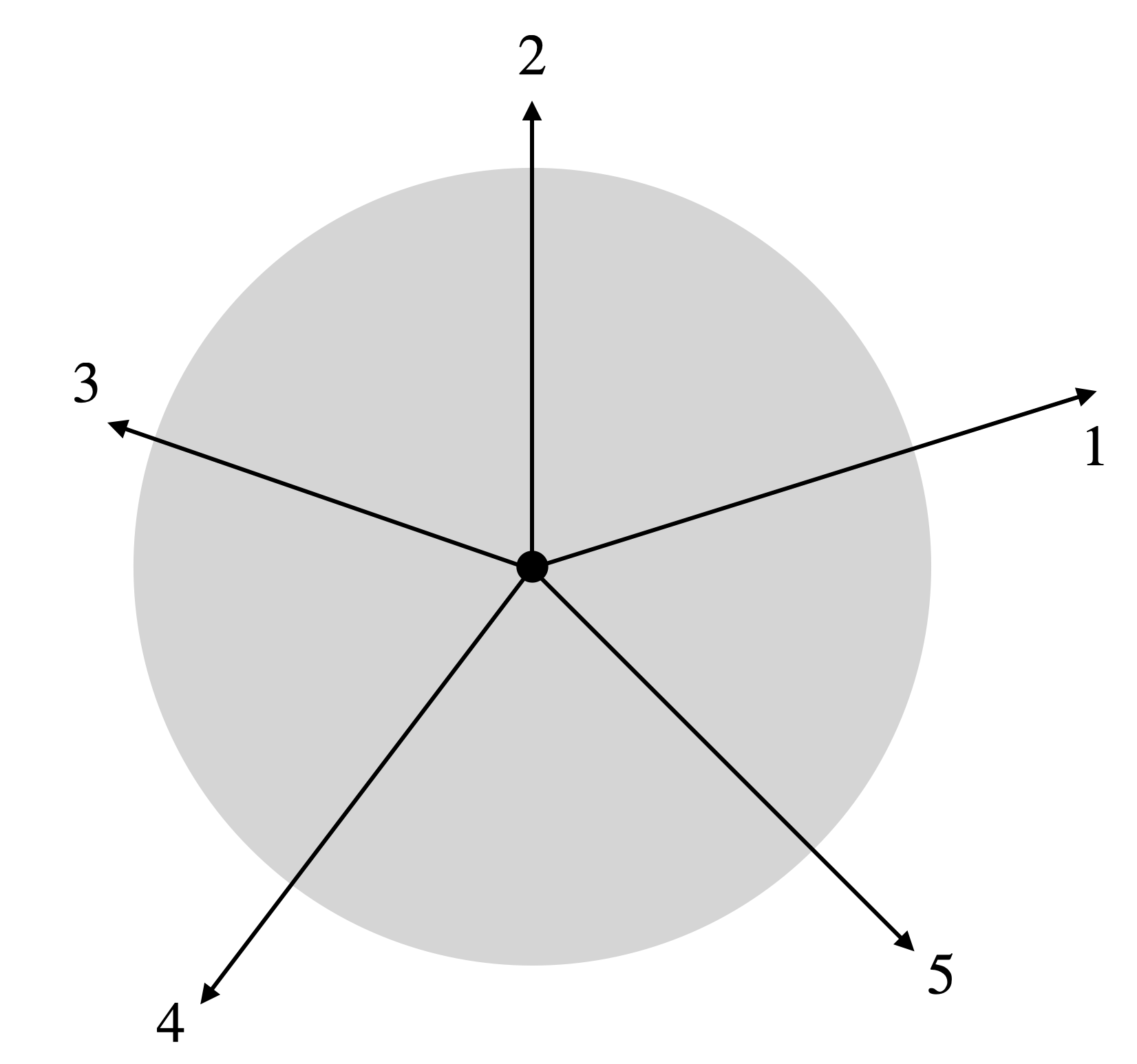}
        \caption{The chamber complex of a pentagon.}
        \label{fig:chambercomplexpentagon}
    \end{figure}
\end{example}

Let ${\rm Vol}(\cdot)$ denote the $d$-dimensional normalized volume. That is, ${\rm Vol}(B)$ is given by 
\[{\rm Vol}(B) \, = \, d! \cdot \int_{B} 1 \, {\rm d} y_1 \cdots {\rm d}y_d\]
for any bounded set $B \subset \mathbb{R}^d$. The volume function $x \mapsto {\rm Vol}(P_x)$ was studied, for instance, in \cite{lasserre1983analytical}. Here, we are interested in the \emph{dual volume function} $x \mapsto {\rm Vol}(P_x^{\, \circ})$ instead, where $x \in \mathbb{R}^n_+$. The polytope $P_x^{\, \circ}$ is the \emph{polar dual} of $P_x$. It lives in the dual vector space $(\mathbb{R}^d)^\vee$:
\[ P_x^{\, \circ} \, = \, \{ u \in (\mathbb{R}^d)^\vee \, : \, u \cdot y \geq -1, \text{ for all } y \in P_x \}. \]
If $P_x$ is simple, then $P_x^{\, \circ}$ is simplicial. Notice that $P_x$ contains the origin $0 \in \mathbb{R}^d$ in its interior if and only if $x \in \mathbb{R}^n_+$. In that case, the dual polytope $P_x^{\, \circ}$ is bounded.  
The relation between the dual volume function and ${\rm Amp}_\Sigma$ can be found in \cite[Corollary 2.6]{pavlov2025santalo} and \cite[Section 2.2]{gao2024dual}.

\begin{lemma} \label{lem:dualvolfromamp}
    Let $C \in {\rm Ch}(U)$ be an $n$-dimensional cone in the chamber complex of $U$. For $x \in C \cap \mathbb{R}^n_+$, we have ${\rm Amp}_{\Sigma_C,U_C}(x) = {\rm Vol}(P_x^{\, \circ})$.
\end{lemma}

\begin{proof}
    For each $x \in {\rm int}(C)$, the polytope $P_x^{\, \circ}$ is simplicial and given by 
    \begin{equation} \label{eq:decomp} P_x^{\, \circ} \, = \, \bigcup_{\sigma \in \Sigma_C(d)} {\rm Conv} \Big ( \,  \{ 0\} \cup \Big \{ \frac{u_\rho}{x_{\rho}} \, : \, \rho \in \sigma(1) \Big \} \, \Big).\end{equation}
    This is a triangulation of $P_x^{\, \circ}$, and the normalized volumes of these simplices are precisely the terms in ${\rm Amp}_{\Sigma_C, U_C}$. For $x \in C \setminus {\rm int}(C)$, the statement follows by continuity.
\end{proof}

\begin{example} 
    We reconsider Example \ref{ex:pentagonintro}. The dual polygon $P^\circ_x$ for $x = (1,1,1,1,1)$ is shown in the left part of Figure \ref{fig:dualpentagon}. It decomposes into five simplices as in \eqref{eq:decomp}.
    \begin{figure}
        \centering
        \includegraphics[height=4.0cm]{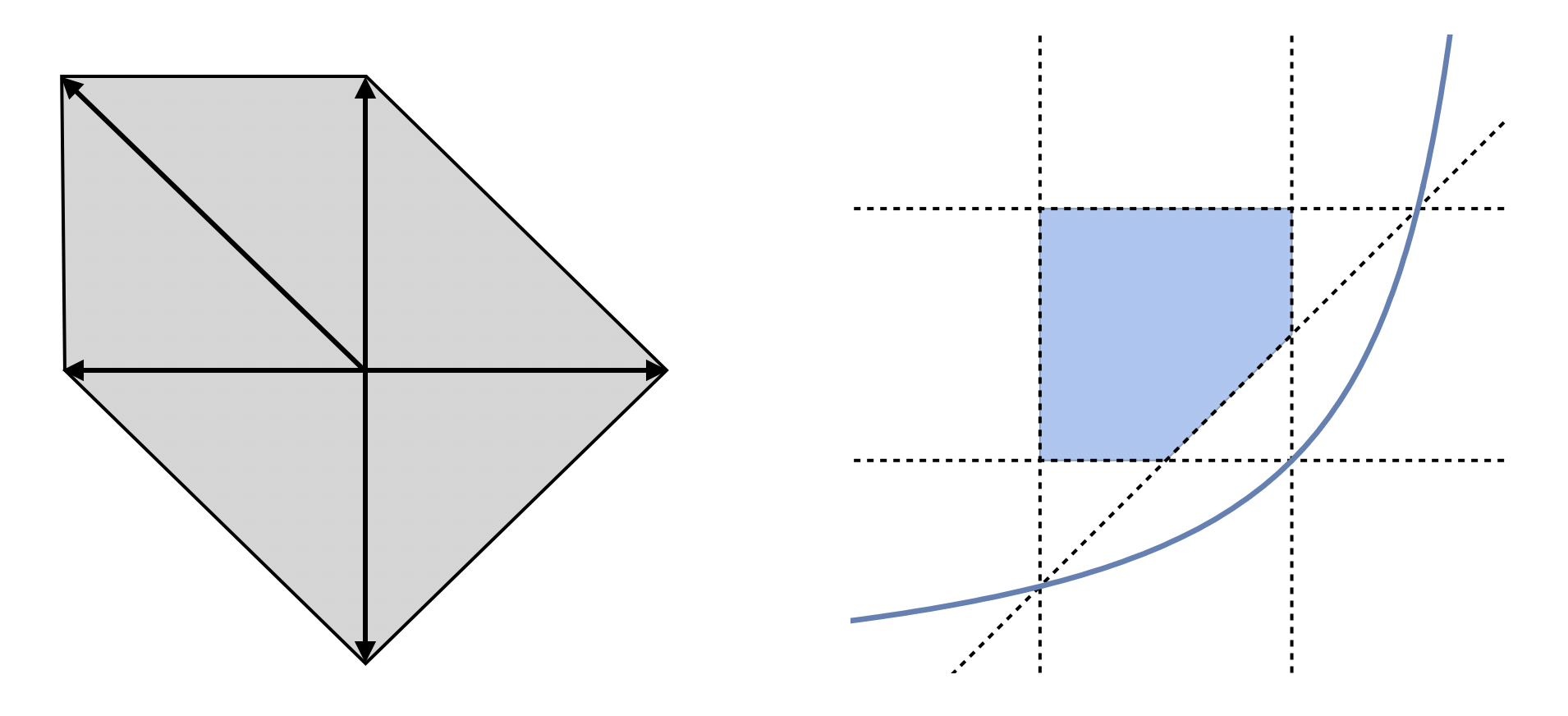}
        \quad \,  
        \includegraphics[height = 4.2cm]{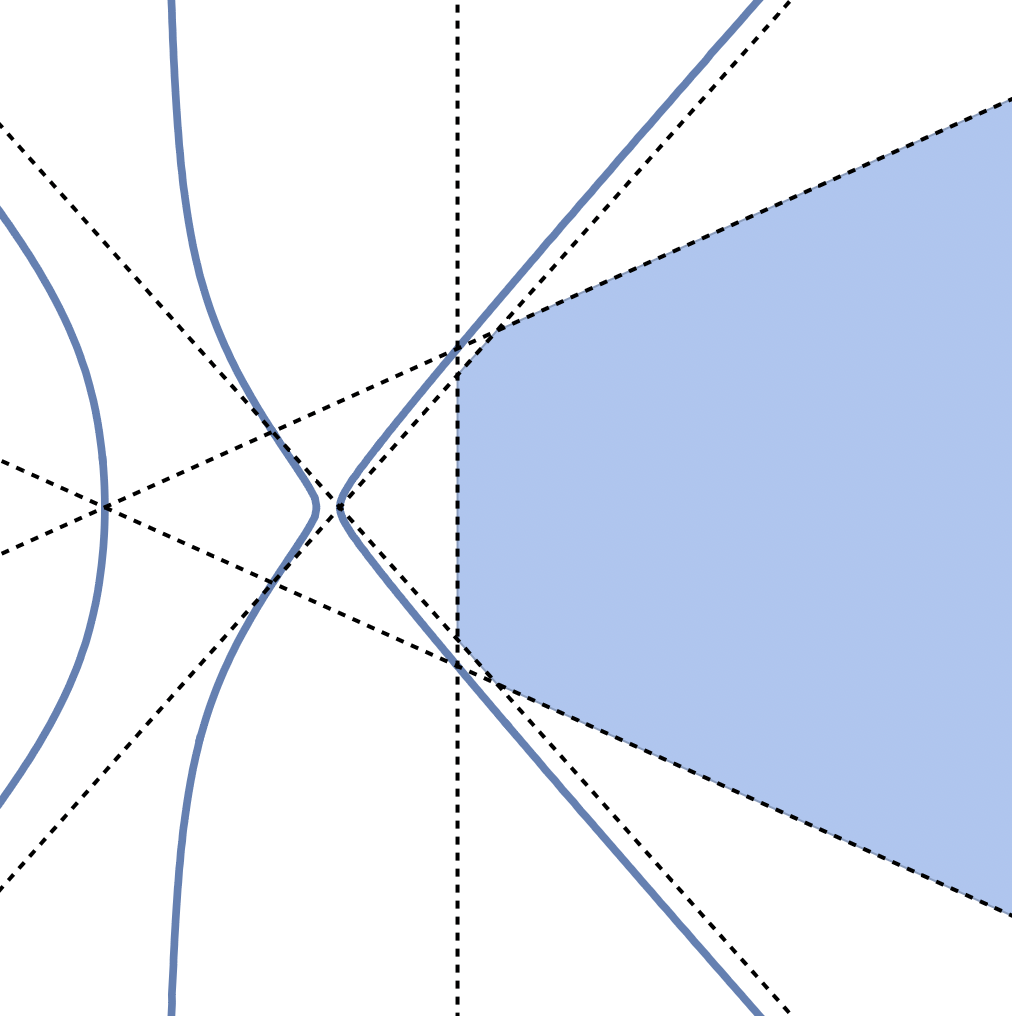}
        \caption{Left: the polar dual polygon of the pentagon in Figure \ref{fig:fanpentagon}. Middle: the adjoint curve of that pentagon. Right: the adjoint curve of an unbounded polyhedron.}
        \label{fig:dualpentagon}
    \end{figure}
\end{example}

 Below, we fix a cone $C \in {\rm Ch}(U)$ of dimension $n$ and set $U = U_C$, $\Sigma = \Sigma_C$.
We explain the name \emph{universal adjoint} for the polynomial ${\rm Adj}_\Sigma(x) = x_1 \cdots x_n \cdot {\rm Amp}_\Sigma(x)$. Adjoints of polygons were studied by Wachspress \cite{Wachspress1975}. They were introduced by Warren \cite{warren1996barycentric} in his construction of barycentric coordinates on polytopes of arbitrary dimension. Adjoints appear naturally when studying dual volumes \cite[Theorem 4.3]{gao2024dual}. Fix $z \in {\rm int}(C)$ and for each $y \in \mathbb{R}^d$, let $P_{z} - y = \{ y' - y \, : \, y' \in P_{z} \}$ be the translated polytope. For $y \in {\rm int}(P_{z})$ we have 
\begin{equation} \label{eq:adj} {\rm Vol}(P_{z}-y)^\circ \, = \, \frac{{\rm adj}_{P_{z}}(y)}{(u_1 \cdot y + z_1)(u_2 \cdot y + z_2) \cdots (u_n \cdot y + z_n)} .\end{equation}
Here ${\rm adj}_{P_{z}}(y)$ is a polynomial, called the \emph{adjoint polynomial} of $P_{z}$. We use this as our definition. 
The rational function \eqref{eq:adj} is the \emph{canonical function} of $P_z$ as a \emph{positive geometry} \cite{arkani2017positive,lam2024invitation}. 
The universal adjoint ${\rm Adj}_\Sigma$ encodes ${\rm adj}_{P_{z}}(y)$ for each $z \in {\rm int}(C)$, in the following~sense. 

\begin{lemma} \label{lem:universaladjoint}
    For any $z \in {\rm int}(C)$, we have  ${\rm Adj}_\Sigma(U\, y+z) = {\rm adj}_{P_{z}}(y)$ in $\mathbb{R}[y_1,\ldots,y_d]$. 
\end{lemma}

\begin{proof}
    As a consequence of Lemma \ref{lem:dualvolfromamp}, for any $z \in {\rm int}(C)$ and $y \in {\rm int}(P_{z})$, we have the equality ${\rm Vol}(P_{z}-y)^{\, \circ} = {\rm Amp}_\Sigma(U\, y+z)$. Therefore, the rational functions ${\rm Amp}_\Sigma(U \, y + z)$ and the righthand side of \eqref{eq:adj} agree on the $d$-dimensional open set ${\rm int}(P_z)$. Hence, they are equal as rational functions. The lemma follows by multiplying with $\prod_{i=1}^n (u_i \cdot y + z_i)$. 
\end{proof}

\begin{proposition} \label{prop:degadj}
    If $\Sigma$ is complete, i.e., if $P_z$ is bounded for all $z \in {\rm int}(C)$, then the adjoint polynomial ${\rm adj}_{P_z}(y)$ has degree at most $n-d-1$. 
\end{proposition}
\begin{proof}
    The polynomial ${\rm Adj}_{\Sigma}(U \, y + z \, y_0)$ is homogeneous of degree $n-d$ in $y_0, \ldots, y_d$. It vanishes identically on the hyperplane $y_0 = 0$ by Theorem \ref{thm:containsU}. Hence ${\rm Adj}_{\Sigma}(U \, y + z \, y_0)$ has a factor $y_0$. After dehomogenizing by setting $y_0 = 1$ we obtain a polynomial of degree at most $n-d-1$ in $y_1, \ldots, y_d$. By Lemma \ref{lem:universaladjoint}, that polynomial is ${\rm adj}_{P_z}(y)$. 
\end{proof}

\begin{example} \label{ex:geometric}
Lemma \ref{lem:universaladjoint} has a nice geometric interpretation. 
The $(d-1)$-plane $\mathbb{P}({\rm im}(U))$ is contained in the universal adjoint surface ${\cal A}_P$ by Theorem \ref{thm:containsU}. Fix $z \in {\rm int}(C)$ and think of it as a point in $\mathbb{P}^{n-1}$. The $d$-plane $H$ spanned by $\mathbb{P}({\rm im}(U))$ and $z$ intersects the nonnegative points of $\mathbb{P}^{n-1}$, i.e., the points which can be represented by $n$ nonnegative coordinates, in a polytope $P_z$. The intersection of ${\cal A}_P$ with $H$ is the union of $\mathbb{P}({\rm im}(U))$ and a hypersurface of degree $n-d-1$. That hypersurface in $\mathbb{C}^d \simeq H \setminus \mathbb{P}({\rm im}(U))$ is Warren's adjoint hypersurface given by ${\rm adj}_{P_z}(y) = 0$. Figure \ref{fig:adjquadrilateral} illustrates this for a quadrilateral $P \subset \mathbb{R}^2$. This property of ${\cal A}_P$ helps to show that the adjoint curve of a generic polygon is~smooth (Corollary \ref{cor:genericsmoothngon}).
    \begin{figure}
        \centering
        \includegraphics[height = 6.0
        cm]{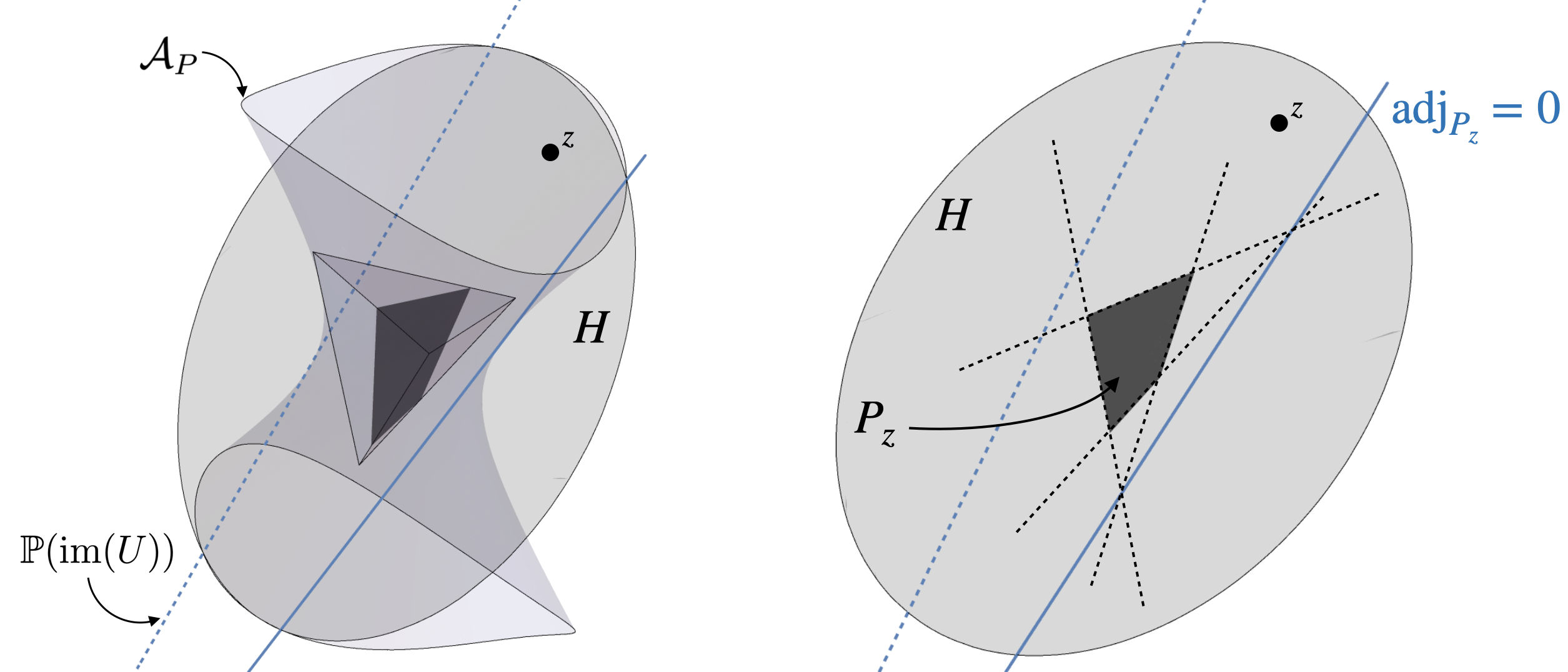}
        \caption{Obtaining the adjoint line of a quadrilateral via its universal adjoint quadric. }
        \label{fig:adjquadrilateral}
    \end{figure}
\end{example}

\begin{example}
    Let $U$ be as in Example \ref{ex:pentagonintro}. The polygon $P_z$ with $z = (1,1,1,1,1)$ is shown in Figure \ref{fig:fanpentagon}. Its adjoint polynomial ${\rm adj}_{P_z}(y) = 5 - 3\, y_1 +3 \, y_2 -y_1y_2$ defines a hyperbola in $\mathbb{R}^2$, as shown in the middle part of Figure \ref{fig:dualpentagon}. This polynomial is obtained by substituting 
    \[ x_1 \, = \, y_1 + 1, \quad x_2 \,= \, y_2 + 1, \quad x_3 \, = \,  -y_1 + y_2 + 1, \quad x_4 \, = \, -y_1 +1, \quad x_5 \, = \, -y_2 + 1\]
    into \eqref{eq:adjpentagon}, see \cite{mathrepo}. The degree drops from three to two, as predicted by Proposition \ref{prop:degadj}.

    The polyhedron shown in Figure \ref{fig:dualpentagon} (right) is unbounded. It is given by the inequalities 
    \[ 2 \, y_1 + 5 \, y_2 + 10 \geq 0, \, \, \,   y_1 + y_2 + 3 \geq 0, \, \, \,   y_1 + 2 \geq 0, \, \, \, y_1-y_2 +3 \geq 0, \, \, \,   2 \, y_1 - 5 \, y_2 +10 \geq 0. \]
    Labeling the rays of the normal fan in that order, the universal adjoint is
    $ {\rm Adj}_\Sigma(x) \, = \, 3 \, x_1x_2x_3 + x_1x_2x_5 + x_1x_4x_5 + 3 \, x_3x_4x_5$.
    Substituting $x_1 = 2 \, y_1 + 5 \, y_2 + 10$ and so on, we obtain the adjoint of $P_z$ with $z = (10,3,2,3,10)$:
    $ {\rm adj}_{P_z}(y) \, = \, 20 \, y_1^3 + 224 \, y_1^2 - 20\, y_1y_2^2 + 812\, y_1 - 90 \, y_2^2 + 960$.
    Since $\Sigma$ is not complete, Proposition \ref{prop:degadj} does not apply: the degree is $n - d = 3$. 
\end{example}

Figures \ref{fig:dualpentagon} and \ref{fig:adjquadrilateral} suggest that the adjoint hypersurface of $P_z$ interacts with its facet hyperplane arrangement in an interesting way. This is made precise by the following fact, proved in \cite[Theorem 1 and Proposition 2]{kohn2020projective}. View $P_z$ as a convex polytope in $\mathbb{RP}^d$ and let ${\cal H}$ be its projective hyperplane arrangement. Define the \emph{residual arrangement} ${\cal R}_{P_z}$ as the union of all flats of ${\cal H}$ which do not intersect $P_z$. The polynomial ${\rm adj}_{P_z}(y)$ vanishes on ${\cal R}_{P_z}$. Furthermore, if the arrangement ${\cal H}$ is simple, meaning that no $d+1$ hyperplanes meet, then ${\rm adj}_{P_z}(y)$ is uniquely determined by these interpolation conditions (up to scaling). In Figure \ref{fig:dualpentagon} (middle) the blue adjoint curve is the unique conic passing through the three residual points seen in the picture and the two residual points at infinity. In Figure \ref{fig:dualpentagon} (right), $P_z$ is a hexagon in $\mathbb{RP}^2$. The residual arrangement ${\cal R}_{P_z}$ consists of nine points, three of which lie at infinity. This leads us to ask which linear spaces are contained in the universal adjoint hypersurface ${\cal A}_\Sigma$, and whether it is similarly determined by interpolation conditions. This is our next topic.

\section{Fano schemes} \label{sec:fano}

Let $X \subset \mathbb{P}^{n-1}$ be a hypersurface. The \emph{Fano scheme of $k$-planes} of $X$ is a subscheme of the Grassmannian ${\rm Gr}(k+1,n)$ of $k$-planes in $\mathbb{P}^{n-1}$ \cite[\S 6.1]{eisenbud20163264}. Its underlying algebraic set is
\[ F_k(X) \, = \, \{ [\Lambda] \in {\rm Gr}(k+1,n) \, : \, \Lambda \subseteq X \}. \]
Here $\Lambda \subseteq \mathbb{P}^{n-1}$ is a linear subspace of dimension $k$, and $[\Lambda]$ is its point in the Grassmannian. We clearly have $F_0(X) = X$ and $F_k(X) = \emptyset$ for $k > n-2$. Moreover, $F_{n-2}(X) = \emptyset$ unless the defining equation of $X$ has a linear factor. We are mainly interested in $1 \leq k \leq n-3$.

We study the Fano schemes of the adjoint hypersurface ${\cal A}_\Sigma$. Except for some small examples, such as those in Sections \ref{subsec:quadrilateral} and \ref{subsec:pentagon}, it is out of reach to compute these schemes explicitly. We limit ourselves to finding points in $F_k({\cal A}_\Sigma)$ from the combinatorics of $\Sigma$. Throughout the section, $U \in \mathbb{R}^{n \times d}$ is any matrix whose rows are ray generators for $\Sigma(1)$. The hypersurface ${\cal A}_\Sigma = {\cal A}_{\Sigma,U}$ depends on $U$, but we drop $U$ from the notation for simplicity. We have already established a point in the Fano scheme of $(d-1)$-planes (Theorem \ref{thm:containsU}). 
\begin{corollary}
    If $\Sigma$ is a complete fan, then we have $[\mathbb{P}({\rm im}(U))] \in F_{d-1}({\cal A}_\Sigma)$. 
\end{corollary}
To a cone $\sigma \in \Sigma$ we associate the monomial $x^{\hat{\sigma}} = \prod_{\rho \notin \sigma(1)} x_\rho$. The \emph{irrelevant ideal} of $\Sigma$ is 
\begin{equation} \label{eq:irrel} B(\Sigma) \, = \, \langle \,  x^{\hat{\sigma}} \, : \, \sigma \in \Sigma \, \rangle \, \, \subseteq \, R_\Sigma. \end{equation}
This monomial ideal is also known as the \emph{Stanley-Reisner ideal} of the Alexander dual of $\Sigma$. A minimal set of generators for $B(\Sigma)$ is given by $x^{\hat{\sigma}}$, where $\sigma$ ranges over the maximal cones of $\Sigma$ with respect to inclusion. The name \emph{irrelevant ideal} is motivated by toric geometry, where $B(\Sigma)$ arises in Cox's GIT construction of the normal toric variety $X_\Sigma$, see \cite{Cox1995} and \cite[\S 5.1]{CoxLittleSchenck2011}. We follow these references in writing $Z(\Sigma) = V(B(\Sigma))$ for the variety defined by $B(\Sigma)$. We think of $Z(\Sigma)$ as a variety in $\mathbb{P}^{n-1}$, which is unnatural in the setting from \cite{Cox1995} but convenient in ours. Let $\Sigma^c$ be the set of subsets of $\Sigma(1)$ which are not contained in a cone of $\Sigma$. We have 
\begin{equation} \label{eq:Z} Z(\Sigma) \, = \, \bigcup_{J \in \Sigma^c} V(x_\rho \, : \, \rho \in J) \, \, \subset \, \mathbb{P}^{n-1}.\end{equation}
This becomes a minimal irreducible decomposition of $Z(\Sigma)$ if $J \in \Sigma^c$ ranges over \emph{primitive collections} in the union \eqref{eq:Z}. These are the minimal elements of $\Sigma^c$ \cite[Definition 5.1.5]{CoxLittleSchenck2011}. 

\begin{example} \label{ex:irrel}
The irrelevant ideal $B(\Sigma)$ for $\Sigma$ from the right part of Figure \ref{fig:fansSec2} is given by 
$B(\Sigma) \, = \, \langle x_3x_4x_6, \,  x_1x_4x_6, \,  x_1x_2x_6,  \, x_2x_3x_6, \, x_1x_2x_3x_5, \,  x_1x_2x_4x_5 \rangle$.
Its variety is 
\begin{equation} \label{eq:ZIncomplete} Z(\Sigma) \, = \, \Lambda_{13} \cup \Lambda_{24} \cup \Lambda_{16} \cup \Lambda_{26} \cup \Lambda_{56} \cup \Lambda_{346}, \end{equation}
where $\Lambda_J = V(x_j \, : \, j \in J)$. The primitive collections of $\Sigma$ are $13, 24, 16, 26, 56$ and $346$.
\end{example}

Example \ref{ex:irrel} introduces convenient notation for coordinate subspaces: for $J \subseteq \Sigma(1)$, we~set 
\begin{equation} \label{eq:LambdaJ} \Lambda_J \, = \, V(x_\rho \, : \, \rho \in J) \, = \, \{ x \in \mathbb{P}^{n-1} \, : \, x_\rho = 0 \text{ for all } \rho \in J \}. \end{equation}
The following statement follows directly from the definitions of ${\rm Adj}_\Sigma$ and $B(\Sigma)$.
\begin{proposition} \label{prop:containedinB}
    The hypersurface ${\cal A}_\Sigma$ contains each coordinate subspace $\Lambda_J$ for $J \in \Sigma^c$. We have $[\Lambda_J] \in F_{n-1-|J|}({\cal A}_\Sigma)$, where $|J|$ denotes the cardinality of $J$, and ${\rm Adj}_\Sigma \in B(\Sigma)$.
\end{proposition}

If $\Sigma$ contains a cone which is not a face of a $d$-dimensional cone, we can improve Proposition \ref{prop:containedinB}. Let $\overline{\Sigma} \subseteq \Sigma$ be the subfan of $\Sigma$ consisting of all $d$-dimensional cones $\Sigma(d)$ and their faces. 
Its irrelevant ideal $B(\overline{\Sigma})$ lives in a subring of $R_\Sigma$. It defines a union of coordinate subspaces \[ V(B(\overline{\Sigma})) \, = \, \bigcup_{J \in \overline{\Sigma}^c} \Lambda_J \, \, \subseteq \, \mathbb{P}^{n-1},  \]
where $\overline{\Sigma}^c$ is the set of subsets of $\overline{\Sigma}(1)$ which do not form a cone in $\overline{\Sigma}$. 
\begin{proposition} \label{prop:coordinatehyperplanes}
    The hypersurface ${\cal A}_\Sigma$ contains the following union of coordinate subspaces: 
    \begin{equation} \label{eq:containedinsmallerB} \Big ( \bigcup_{\rho \in \Sigma(1) \setminus \overline{\Sigma}(1)} \Lambda_\rho \Big) \cup V(B(\overline{\Sigma})) \, \subseteq \, {\cal A}_\Sigma. \end{equation}
    That is, we have $[\Lambda_\rho] \in F_{n-2}({\cal A}_\Sigma)$ for $\rho \in \Sigma(1) \setminus \overline{\Sigma}(1)$ and $[\Lambda_J] \in F_{n-1-|J|}({\cal A}_\Sigma)$ for $J \in \overline{\Sigma}^c$.
\end{proposition}
\begin{proof}
    A ray of $\Sigma$ which does not belong to $\overline{\Sigma}(1)$ is not contained in any $d$-dimensional cone. Hence, $x_\rho$ is a factor of ${\rm Adj}_\Sigma$. If $\overline{U}$ is the submatrix of $U$ with rows indexed by $\overline{\Sigma}(1)$, then
    \begin{equation} \label{eq:looserays}
        {\rm Adj}_{\Sigma, U} \, = \, \Big ( \prod_{\rho \in \Sigma(1) \setminus \overline{\Sigma}(1) } x_\rho \Big ) \cdot {\rm Adj}_{\overline{\Sigma},\overline{U}}.
    \end{equation}
    The inclusion $V(B(\overline{\Sigma})) \subseteq {\cal A}_\Sigma$ now follows from ${\rm Adj}_{\overline{\Sigma},\overline{U}} \in B(\overline{\Sigma})$, see Proposition \ref{prop:containedinB}. 
\end{proof}

\begin{example}
    The polynomial \eqref{eq:AdjIncomplete} has $x_6$ as a factor, which corresponds to the fact that $\rho_6$ does not belong to any cone of $\Sigma(3)$. The inclusion \eqref{eq:containedinsmallerB} reads $\Lambda_6 \cup \Lambda_{13} \cup \Lambda_{24} \subseteq {\cal A}_\Sigma$.
    In particular, by \eqref{eq:ZIncomplete}, this implies that $Z(\Sigma) \subseteq {\cal A}_\Sigma$, as stated in Proposition \ref{prop:containedinB}.
\end{example}

In the rest of the section, we focus on the case where $\Sigma = \Sigma_P$ is the complete normal fan of a $d$-dimensional simple polytope $P \subseteq \mathbb{R}^d$. To emphasize this, we denote the adjoint hypersurface by ${\cal A}_P = {\cal A}_{\Sigma_P} \subseteq \mathbb{P}^{n-1}$. The set $\Sigma^c$ now consists of subsets $J$ of the set of facets of $P$ such that $\bigcap_{Q \in J} Q = \emptyset$. We have $\Sigma = \overline{\Sigma}$ and Proposition \ref{prop:containedinB} says that if $J \in \Sigma^c$, then $\Lambda_J = V(x_Q \, : \, Q \in J) \subseteq {\cal A}_P$. For each edge $e$ of $P$, the subspace $\Lambda_e \, = \, V(x_Q \, : \, e \subseteq Q) \, \subseteq \, \mathbb{P}^{n-1}$ has dimension $n-d$. Let $Q_e^1$ and $Q_e^2$ be the two facets of $P$ which do not contain $e$, but do intersect $e$ in its vertices $v_e^1$ and $v_e^2$ respectively. Finally, our next statement uses 
\begin{equation} \label{eq:He} H_e \, = \, \{ x\in \mathbb{P}^{n-1} \, : \, u_{v_e^1} \, x_{Q_e^2} + u_{v_e^2} \, x_{Q_e^1} \, = \, 0 \}, \end{equation}
where $u_v = |\det(U_{\sigma_v})|$ and $\sigma_v \in \Sigma_P$ is the $d$-dimensional cone corresponding to the vertex $v$.

\begin{theorem} \label{thm:interpol}
    Let $P\subseteq \mathbb{R}^d$ be a full-dimensional simple convex polytope and let $\Sigma = \Sigma_P$ be its normal fan. We have that ${\cal A}_P = \{ x \in \mathbb{P}^{n-1} \, : \, {\rm Adj}_P(x) = 0 \}$ is the unique hypersurface of degree $n-d$ satisfying the following properties:
    \begin{enumerate}
        \item $[\Lambda_J] \in F_{n-1-|J|}({\cal A}_P)$ for each primitive collection $J \in \Sigma^c$ and 
        \item $[\Lambda_e \cap H_e] \in F_{n-d-1}({\cal A}_P)$ for each edge $e$ of $P$.
    \end{enumerate}
\end{theorem}

\begin{proof}
    Note that ${\cal A}_P$ satisfies the first property by Proposition \ref{prop:containedinB}. The second property follows from the fact that, setting $x_Q = 0$ for each facet $Q$ containing $e$, we obtain 
    \begin{equation} \label{eq:restricttoedge} ({\rm Adj}_P)_{|\Lambda_e} \, = \, \Big (\prod_{Q \cap e = \emptyset} x_Q \Big ) \cdot (u_{v_e^1} \, x_{Q_e^2} + u_{v_e^2} \, x_{Q_e^1}), \end{equation}
    (Lemma \ref{lem:restrictpolytope}). Additionally setting $x_Q = 0$ with $Q \cap e = \emptyset$ gives $\Lambda_J$ for some $J \in \Sigma^c$. Setting $u_{v_e^1} \, x_{Q_e^2} + u_{v_e^2} \, x_{Q_e^1} = 0$ instead shows that $\Lambda_e \cap H_e \subseteq {\cal A}_P$. It remains to show uniqueness. Suppose that $X = \{ x \in \mathbb{P}^{n-1} \, : \, f(x) = 0 \}$ is a hypersurface of degree $n-d$ satisfying properties 1 and 2. Property 1 implies that the defining equation $f$ of $X$ is contained in the degree-$(n-d)$ part of the irrelevant ideal $B(\Sigma)_{n-d}$. This is spanned as a $\mathbb{C}$-vector space by $\{ x^{\hat{\sigma}} \, : \, \sigma \in \Sigma(d) \}$. In other words, we must have $f = \sum_{\sigma \in \Sigma(d)} z_\sigma \, x^{\hat{\sigma}}$ for some $z_\sigma \in \mathbb{C}$. The condition $\Lambda_e \cap H_e \subseteq X$ fixes the ratio between the coefficients $z_{\sigma_{v_e^1}}$ and $z_{\sigma_{v_e^2}}$. Since the edge graph of $P$ is connected, this fixes $f$ up to scaling. Hence, $X$ is unique, and it equals ${\cal A}_P$.
\end{proof}

\begin{remark} One can replace \emph{each edge of $P$} by \emph{each edge in a path on the edge graph of $P$ which visits all vertices} in property 2 of Theorem \ref{thm:interpol}. 
\end{remark}

Related to the study of Fano schemes is the question \emph{for which linear restrictions does ${\rm Adj}_\Sigma$ factor as a product of $n-d$ linear forms}? This is interesting in the physics application as well, see \cite{arkani2024hidden,umbert2025splitting}. We define the \emph{$k$-th split variety} of a hypersurface $X \subset \mathbb{P}^{n-1}$ as follows: 
\[ {\rm Split}_k(X) \, = \, \{ [\Lambda] \in {\rm Gr}(k+1,n) \, : \, X \cap \Lambda \text{ is a union of $(k-1)$-planes } \}. \]
For each face $\Delta \subseteq P$, we define the linear space $\Lambda_\Delta = V(x_Q \, : \, Q \supseteq \Delta)$ as in \eqref{eq:LambdaDelta}.
\begin{corollary} \label{cor:simplexfaces}
    Let $P\subseteq \mathbb{R}^d$ be a full-dimensional simple convex polytope. If $\Delta$ is a $(d-k)$-dimensional face of $P$ which is a product of simplices, then we have $[\Lambda_\Delta] \in {\rm Split}_{n-k-1}({\cal A}_P)$.
\end{corollary}
\begin{proof}
    The universal adjoint of a simplex is linear, and that of a product of simplices is a product of linear forms by Lemma \ref{lem:factorfans}. The statement follows from Lemma~\ref{lem:restrictpolytope}.
\end{proof}

\begin{example}
    Equation \eqref{eq:restricttoedge} confirms that $[\Lambda_e] \in {\rm Split}_{n-d}({\cal A}_P)$ for each edge $e$ of $P$.
\end{example}

\begin{example}
    The quadrilateral faces of the associahedron in Figure \ref{fig:ABHY3} are products of line segments. The restriction $({\rm Adj}_P)_{|x_{14} = 0}$ factors as $x_{25}x_{26}x_{36}x_{35} (x_{13} + x_{24}) (x_{15} + x_{46})$.
\end{example}

The following statement says that the Fano schemes of $P$ contain those of its faces.  For a face $\Delta \subseteq P$, let $\Delta^c = \{Q \subseteq P \, : \, \text{$Q$ is a facet and $Q \cap \Delta = \emptyset$} \}$ and let $|\Delta^c|$ be its cardinality. 

\begin{proposition} \label{prop:embedding}
    Let $P\subseteq \mathbb{R}^d$ be a full-dimensional simple convex polytope. Let $\Delta$ be a face of $P$. For each $k$, there is an injective map $F_k({\cal A}_\Delta) \hookrightarrow F_{k + |\Delta^c|}({\cal A}_P)$. 
\end{proposition}

\begin{proof}
    Let $\ell$ be the dimension of $\Delta$ and let $n_\Delta$ be its number of facets. Let $\Lambda$ be a $k$-plane contained in the hypersurface ${\cal A}_\Delta \subseteq \mathbb{P}^{n_\Delta-1}$. It is defined by $n_\Delta - k - 1$ linear equations in the variables $\{x_Q \, : \, Q \cap \Delta \text{ is a facet of $\Delta$}\}$. These same equations define a linear subspace $\Lambda'$ of dimension $n-1-(d-\ell)-(n_\Delta - k -1) = k + |\Delta^c|$ in $\Lambda_\Delta \simeq \mathbb{P}^{n-1-(d-\ell)}$. By Lemma \ref{lem:restrictpolytope}, the polynomial ${\rm Adj}_P(x)$ vanishes on this linear subspace. The map is $\Lambda \mapsto \Lambda'$. 
\end{proof}

\begin{example}
    The adjoint of an edge $e$ of $P$ is a linear polynomial in two variables. It defines one point ${\cal A}_e = \{\Lambda\} \subset \mathbb{P}^1$, so that $[\Lambda] \in F_0({\cal A}_e)$. Its image in $F_{n-d-1}({\cal A}_P)$ under the embedding from Proposition \ref{prop:embedding} is the $(n-d-1)$-plane $[\Lambda_e \cap H_e]$ from Theorem \ref{thm:interpol}.
\end{example}

In the next section, we will use the results of Section \ref{sec:dualvolume} to explain some of the linear spaces contained in ${\cal A}_P$ from the chamber complex and deformation cone of the polytope $P$.

\section{Deformations} \label{sec:chambercomplex}

In this section, $\Sigma$ is the normal fan of a simple convex polytope $P \subset \mathbb{R}^d$ of dimension $d$, and $U$ is its ray matrix. The \emph{deformation cone} of $\Sigma$ is the cone $C$ in the chamber complex ${\rm Ch}(U)$ such that the normal fan of $P_x$ from \eqref{eq:Px} equals $\Sigma$ for all $x \in {\rm int}(C)$ \cite[Section 2]{CastilloLiu2020NestedBraidFans}. We recall its characterization in terms of convex piecewise linear functions. This is standard in toric geometry, see for instance \cite[Section 6.1]{CoxLittleSchenck2011} and \cite{cox2009primitive}.  A function $\phi: \mathbb{R}^d \rightarrow \mathbb{R}$ is called \emph{piecewise linear} on $\Sigma$ if its restriction to each $\sigma \in \Sigma(d)$ is given by a linear function. In particular, $\phi$ is continuous. Each such piecewise linear function is specified by its values at $u_\rho, \rho \in \Sigma(1)$. Hence, the piecewise linear functions on $\Sigma$ form a vector space ${\rm PL}(\Sigma) \simeq \mathbb{R}^n$. Below, we will consistently identify ${\rm PL}(\Sigma)$ with $\mathbb{R}^n$ in this way: $\phi \sim (\phi(u_\rho))_{\rho \in \Sigma(1)}$. The \emph{deformation cone} ${\rm Def}(\Sigma) \subseteq {\rm PL}(\Sigma)$ is the $n$-dimensional cone of \emph{convex} piecewise linear functions on $\Sigma$.
The interior of ${\rm Def}(\Sigma)$ consists of \emph{strictly convex} functions $\phi \in {\rm PL}(\Sigma)$, which means that $\phi(u) + \phi(v) > \phi(u+v)$ when $u$ and $v$ belong to the interior of different maximal cones of $\Sigma$. This open cone is called the \emph{type cone} of $\Sigma$. In terms of polytopes, we have ${\rm Def}(\Sigma) \, = \, \{ x \in \mathbb{R}^n \, : \, \Sigma_{P_x} \text{ is refined by } \Sigma\}$ and ${\rm int}({\rm Def}(\Sigma)) \, = \, \{ x \in \mathbb{R}^n \, : \, \Sigma_{P_x} = \Sigma \}$. The polytope $P_x$ is that from \eqref{eq:Px}, and a fan $\Sigma'$ is said to be \emph{refined} by $\Sigma$ if each cone of $\Sigma'$ is a union of cones in $\Sigma$. Moving $x$ around in the deformation cone corresponds to translating the facets of $P_x$ in the direction of their normal ray, without ever crossing a vertex of $P_x$.

The cone ${\rm Def}(\Sigma)$ is not pointed. Its lineality space consists of the (global) linear functions $\phi: \mathbb{R}^d \rightarrow \mathbb{R}$. Under the identification ${\rm PL}(\Sigma) \simeq \mathbb{R}^n$, this lineality space is the $d$-dimensional vector space ${\rm im}(U)$. The quotient by ${\rm im}(U)$ leads to an $(n-d)$-dimensional pointed cone ${\rm Nef}(\Sigma) \, = \, {\rm Def}(\Sigma) \, / \, {\rm im}(U)$.
This is called the \emph{nef cone} of $\Sigma$. Here \emph{nef} is short for \emph{numerically effective}. The name comes from divisor theory on normal toric varieties \cite[Section 6.2]{CoxLittleSchenck2011}. 

\begin{example} 
The nef cone of $\Sigma$ in Example \ref{ex:chambercomplex} is the three-dimensional cone over the blue pentagon in Figure \ref{fig:chambercomplexpentagon}. For $x \in {\rm int}({\rm Def}(\Sigma))$, the piecewise linear function $\phi: \mathbb{R}^2 \rightarrow \mathbb{R}$ defined by $\phi(u_\rho) = x_\rho$ is strictly convex on $\Sigma$. If $x$ lies on the relative interior of a facet of the deformation cone, then the function $\phi$ is not strictly convex; it is linear on the union of two adjacent cones of $\Sigma(2)$. This union of two cones is a cone in the normal fan of the quadrilateral $P_x$ seen in the corresponding triangular cell of Figure \ref{fig:chambercomplexpentagon}. The fan $\Sigma$ refines~$\Sigma_{P_x}$.
\end{example}

We derive inequalities for ${\rm Def}(\Sigma)$ from faces of $P$. For a face $\Delta \subseteq P$, let $\sigma_\Delta \in \Sigma$ be its normal cone. Let $\Delta \subseteq P$ be a face which is a simplex of dimension $d-k>0$. The set ${\rm nb}(\Delta)$ consists of the $d-k+1$ facets of $P$ which intersect $\Delta$ in one of its $d-k+1$ facets. It is identified with ${\rm nb}(\sigma_\Delta)$ from \eqref{eq:nb}. 
We will need the following $(d+1) \times (d+1)$ determinant:
\begin{equation} \label{eq:WDelta} \quad W_\Delta(x) \, = \, \det \begin{pmatrix}
    U_{{\rm nb}(\Delta)} & x_{{\rm nb}(\Delta)}\\
    U_{\sigma_\Delta} & x_{\sigma_\Delta} 
\end{pmatrix} \, = \, \sum_{\rho \in {\rm nb}(\Delta)} c_\Delta^{\, \rho} \, x_{\rho} +  \sum_{\rho' \in \sigma_\Delta(1)} c_\Delta^{ \, \rho'} \, x_{\rho'}.\end{equation}
Here, $U_{{\rm nb}(\Delta)}$ is the submatrix of $U$ with rows indexed by ${\rm nb}(\Delta)$, and $U_{\sigma_\Delta}$ is indexed by the rays of $\sigma_\Delta$. The vectors $x_{{\rm nb}(\Delta)}$ and $x_{\sigma_\Delta}$ are subvectors of the column vector $x$, indexed compatibly with $U_{{\rm nb}(\Delta)}$ and $U_{\sigma_\Delta}$. Since the rows can be re-ordered, $W_\Delta(x)$ is defined up to a sign. 

\begin{proposition}
    Let $\Delta \subseteq P$ be a face of $P$ which is a simplex of dimension $d-k>0$. The coefficients $c_\Delta^{\, \rho}, \rho \in {\rm nb}(\Delta)$ appearing in \eqref{eq:WDelta} are nonzero and have the same sign denoted by ${\rm sign}(\Delta) \in \{ \pm 1\}$. We have that ${\rm sign}(\Delta) \cdot W_\Delta(x) \geq 0$ for all $x \in {\rm Def}(\Sigma)$. Moreover, we have 
    \begin{equation} \label{eq:wallrel} {\rm Def}(\Sigma) \, = \, \{x \in \mathbb{R}^n \, : \, {\rm sign}(e) \cdot W_e(x) \geq 0 \text{ for all edges } e \subseteq P \}.\end{equation}
\end{proposition}

\begin{proof}
    If $\tilde{U} \in \mathbb{R}^{(\ell+1) \times \ell}$ is the ray matrix of the normal fan of an $\ell$-dimensional simplex, then $\det ( \tilde{U} ~ y)$ is a linear form in $y$ whose coefficients all have the same sign. Indeed, up to sign, it is the universal adjoint of the simplex (Example \ref{ex:simpleex}). We apply this in our situation.
    
    Below Corollary \ref{cor:resamp}, we defined $T_\Delta$ to be a matrix for which $U_{\sigma_\Delta} \cdot T_\Delta$ has rows $e_1, \ldots, e_k$. The last $d-k$ columns of $U_{{\rm nb}(\Delta)} \cdot T_\Delta$ form the matrix $U_\Delta$, which is a ray matrix for the normal fan of $\Delta$. We set $x_{\rho} = 0$ for $\rho \in \sigma_\Delta(1)$ in the matrix from \eqref{eq:WDelta} and observe that
    \begin{equation}  \label{eq:proof5.3} \det  \left ( \begin{pmatrix}
    U_{{\rm nb}(\Delta)} & x_{{\rm nb}(\Delta)}\\
    U_{\sigma_\Delta} & 0
\end{pmatrix} \cdot \begin{pmatrix}
    T_\Delta & 0 \\ 0 & 1
\end{pmatrix} \right ) \, = \, c \cdot \det \begin{pmatrix}
    U_\Delta & x_{{\rm nb}(\Delta)}
\end{pmatrix} \end{equation}
for some nonzero constant $c$. By the above considerations, the resulting linear form has coefficients of constant sign. That linear form is, up to a constant factor, equal to $\sum_{\rho \in {\rm nb}(\Delta)} c_\Delta^{\, \rho} \, x_{\rho}$. 

The inequality ${\rm sign}(\Delta) \cdot W_\Delta(x) \geq 0$ is the condition for the function $\phi \in {\rm PL}(\Sigma)$ represented by $x \in \mathbb{R}^n$ to be convex along $\sigma_\Delta$, so it must be satisfied for $x \in {\rm Def}(\Sigma)$. The claim \eqref{eq:wallrel} follows from the fact that it suffices to check convexity along the $(d-1)$-dimensional cones, which are the boundaries of the domains of linearity of $\phi$. The conditions ${\rm sign}(e) \cdot W_e(x) \geq 0$ are called \emph{wall-crossing inequalities}, see for instance \cite[Section 6.1]{CoxLittleSchenck2011} and \cite[Theorem 1.6]{cox2009primitive}.
\end{proof}

We note that a wall-crossing inequality ${\rm sign}(e) \cdot W_e(x) \geq 0$ does not necessarily define a facet of ${\rm Def}(\Sigma)$, see Example \ref{ex:cuboid}. Below we write $x_{\sigma_\Delta} = 0$ as a shorthand for ``$x_\rho = 0$ for all $\rho \in \sigma_\Delta(1)$''. The \emph{$(n-1)$-skeleton} ${\rm Ch}(U)_{n-1}$ of the chamber complex is the union of its $(n-1)$-dimensional cones. Let $\overline{{\rm Ch}(U)}_{n-1} \subset \mathbb{P}^{n-1}$ be the projectivization of its Zariski~closure.

\begin{proposition} \label{prop:linspacefromcham}
     Let $\Delta \subseteq P$ be a face of $P$ which is a simplex of dimension $d-k>0$. The $(n-k-2)$-dimensional linear space $W_{\Delta}(x) = x_{\sigma_\Delta} = 0$ is contained in ${\cal A}_\Sigma \cap \overline{{\rm Ch}(U)}_{n-1}$. 
\end{proposition}
\begin{proof}
    From Equation \eqref{eq:proof5.3} one sees that $(W_{\Delta}(x))_{|\Lambda_\Delta} = \pm c_\Delta \cdot {\rm Adj}_{\Delta, U_\Delta}(x)$. Here $\Lambda_\Delta = \{ x_{\sigma_\Delta} = 0 \}$ is as in \eqref{eq:LambdaDelta} and $c_\Delta = \det (T_\Delta)$. The containment in ${\cal A}_\Sigma$ follows from Lemma \ref{lem:restrictpolytope}. We need to show that the linear space $W_{\Delta}(x) = x_{\sigma_\Delta} = 0$ is also contained in $\overline{{\rm Ch}(U)}_{n-1}$. 
    The set $K_\Delta = \Sigma(1) \setminus ({\rm nb}(\Delta) \cup \sigma_\Delta(1))$ consists of $n-d-1$ rays. These are the rays indexing the variables which do not appear in $W_\Delta(x)$. 
    The hyperplane in $\mathbb{R}^n$ defined by $W_{\Delta}(x)=0$ is ${\rm span}_{\mathbb{R}}(e_\rho, \rho \in K_\Delta) + {\rm im}(U)$, the linear span of an $(n-1)$-dimensional cone in ${\rm Ch}(U)$.
\end{proof}

We have seen the linear spaces from Proposition \ref{prop:linspacefromcham} before: if $\Delta = e$ is an edge, then $W_{e}(x) = x_{\sigma_e} = 0$ is the linear space $\Lambda_e \cap H_e$ from Theorem \ref{thm:interpol}. Moreover, we have seen in Corollary \ref{cor:simplexfaces} that $[\Lambda_\Delta] \in {\rm Split}_{n-k-1}({\cal A}_\Sigma)$ for any simplex $\Delta$, and $W_{\Delta}(x) = x_{\sigma_\Delta} = 0$ defines a component of $\Lambda_\Delta \cap {\cal A}_\Sigma$. Proposition \ref{prop:linspacefromcham} relates these linear spaces to the chamber complex. 

\begin{example} \label{ex:Hepentagon} In Example \ref{ex:chambercomplex}, let $\Delta$ be the edge corresponding to the fourth row of $U$. The set $K_\Delta$ consists of the rays $\rho_1$ and $\rho_2$. The projection of the hyperplane $W_{\Delta}(x) = 0$ along ${\rm im}(U)$ is represented by the line connecting $e_1$ and $e_2$ in Figure \ref{fig:chambercomplexpentagon}. Repeating this for all edges, we find the five facet hyperplanes $W_{\Delta}(x) = 0$ of the deformation cone ${\rm Def}(\Sigma)$, shaded in blue in Figure \ref{fig:chambercomplexpentagon}. The interplay between the linear spaces $\Lambda_e \cap H_e$ from Theorem \ref{thm:interpol} and the chamber complex is nicely seen from the overlap of Figures \ref{fig:config} (right) and \ref{fig:chambercomplexpentagon} (left). 
\end{example}

We can use Proposition \ref{prop:linspacefromcham} to study the behavior of the adjoint polynomial ${\rm adj}_P(y)$ from \eqref{eq:adj} under deformations of $P$. 
For each face $\Delta \subseteq P$ which is a simplex of dimension $d-k > 0$ and each $x \in \mathbb{R}^n$ satisfying $W_\Delta(x) = 0$, we define $v_\Delta(x) \in \mathbb{R}^d$ as the unique point satisfying 
\[\begin{pmatrix}
    U_{{\rm nb}(\Delta)} & x_{{\rm nb}(\Delta)}\\
    U_{\sigma_\Delta} & x_{\sigma_\Delta} 
\end{pmatrix} \cdot \begin{pmatrix}
v_\Delta(x) \\ 1
\end{pmatrix} \, = \, 0. \]
This is well-defined, as the first block column of our $(d+1) \times (d+1)$ matrix has rank $d$. Suppose that, in addition to $W_\Delta(x) = 0$, we have $x \in {\rm Def}(\Sigma)$. Then $x$ corresponds to a convex piecewise linear function which is linear on the union of all $d$-dimensional cones in $\Sigma(d)$ containing $\sigma_\Delta$; it is given by $u \mapsto -u \cdot v_\Delta(x)$ on this union. This implies that $v_\Delta(x)$ is a vertex of $P_x$, and $P_x$ is a deformation of $P$ in which the face $\Delta$ shrinks to the vertex $v_\Delta(x)$. 

\begin{proposition} \label{prop:degenadjoint}
Let $\Delta \subseteq P$ be a face of $P$ which is a simplex of dimension $d-k>0$. Suppose that there exists a point $z_0 \in {\rm Def}(\Sigma)$ such that $W_\Delta(z_0) = 0$ and $P_{z_0}$ has dimension $d$. If $P_{z_0}$ has $n$ facets, then the adjoint polynomial ${\rm adj}_{P_{z_0}}(y)$ vanishes at any point $y$ satisfying $U_{\sigma_{\Delta}} \, y + (z_0)_{\sigma_{\Delta}} = 0$. In particular, it vanishes at the vertex $v_{\Delta}(z_0) \in P_{z_0}$. 
\end{proposition}
\begin{proof}
Let $z: [0,1] \rightarrow {\rm Def}(\Sigma)$ be a smooth path, such that $z(0) =z_0$ and $z(t) \in {\rm int}({\rm Def}(\Sigma))$ for $t \in (0,1]$. Let $\Sigma_0(1) = \Sigma_{P_{z_0}}(1)$. By \eqref{eq:adj} we have that, for each $t \in (0,1]$,
    \[ {\rm Vol}(P_{z(t)}-y)^\circ \, = \, \frac{{\rm adj}_{P_{z(t)}}(y)}{\prod_{\rho \in \Sigma(1)}(u_{\rho} \cdot y + z_{\rho}(t) )} \quad \text{and} \quad {\rm Vol}(P_{z_0}-y)^\circ \, = \, \frac{{\rm adj}_{P_{z_0}}(y)}{\prod_{\rho \in \Sigma_0(1)}(u_{\rho} \cdot y + z_{\rho}(0)) } .\]
    We assume that our path is such that $\bigcap_{t \in [0,1]} {\rm int}(P_{z(t)})$ contains a small open ball ${\cal B} \subset \mathbb{R}^d$. By continuity of the dual volume, the limit $\lim_{t \rightarrow 0}{\rm Vol}(P_{z(t)}-y)^\circ $ agrees with ${\rm Vol}(P_{z_0}-y)^\circ$ for $y \in {\cal B}$. Since ${\cal B}$ has dimension $d$, this implies an equality of rational functions in $y$:
    \[ \lim_{t \rightarrow 0} \,  \frac{{\rm adj}_{P_{z(t)}}(y)}{\prod_{\rho \in \Sigma(1)}(u_{\rho} \cdot y + z_{\rho}(t) )} \, = \, \frac{{\rm adj}_{P_{z_0}}(y)}{\prod_{\rho \in \Sigma_0(1)}(u_{\rho} \cdot y + z_{\rho}(0))} .\]
Notice that on the righthand side, the product in the denominator is over all rays which represent a facet of $P_{z_0}$. Comparing these two formulas we find that,  for all $y$, 
    \[ \lim_{t \rightarrow 0} {\rm adj}_{P_{z(t)}}(y) \, = \, {\rm adj}_{P_{z_0}}(y) \cdot \Big (\prod_{\rho \in \Sigma(1) \setminus \Sigma_0(1)} u_{\rho} \cdot y + z_{\rho}(0) \Big). \]
    If $P_{z(t)}$ loses a facet for $t \rightarrow 0$, then the adjoint splits off a linear factor. By Lemma \ref{lem:universaladjoint},
    \begin{equation} \label{eq:proof5.7} \lim_{t \rightarrow 0} {\rm Adj}_\Sigma(U \,y + z(t)) \, = \, {\rm Adj}_\Sigma(U \, y + z_0) \, = \, {\rm adj}_{P_{z_0}}(y)\cdot \Big (\prod_{\rho \in \Sigma(1) \setminus \Sigma_0(1)} u_{\rho} \cdot y + z_{\rho}(0) \Big).\end{equation}
The point $U \, y + z_0 \in \mathbb{R}^n$ lies on the hyperplane $W_{\Delta}(x) = 0$ for any $y$, because that hyperplane contains ${\rm im}(U)$. If $U_{\sigma_{\Delta}} \, y + (z_0)_{\sigma_{\Delta}} = 0$, then $U \, y + z_0$ satisfies the equations $W_{\Delta}(x) = x_{\sigma_\Delta} = 0$, and ${\rm Adj}_\Sigma(U \, y + z_0) = 0$ by Proposition \ref{prop:linspacefromcham}. Under the assumption that $P_{z_0}$ has $n$ facets, we have $\Sigma(1) \setminus \Sigma_0(1) = \emptyset$ and \eqref{eq:proof5.7} implies that ${\rm adj}_{P_{z_0}}(y) = 0$ when $U_{\sigma_{\Delta}} \, y + (z_0)_{\sigma_{\Delta}} = 0$. 
\end{proof}

\begin{example} \label{ex:cuboid}
We consider the simple polytope $P = \{ y \in \mathbb{R}^3\, : \, U \, y + z \geq 0 \}$ given by 
\[ U \, = \, \left (  \begin{matrix} -8 &10 & -4 & 2 & 0 & 0 \\ -11 &  12 &  -3 & 0 & 2 & 0 \\  -7 &6 &  -1 & 0 &  0 & 2\end{matrix} \normalsize \right )^t, \quad z \, = \, \begin{pmatrix} 479 &  336 &  69 &  78 &  208 &  78 \end{pmatrix}^t. \]
Its normal fan is $\Sigma$. Two of the facets of $P$ are triangles, two are quadrilaterals, and two are pentagons. The chamber complex ${\rm Ch}(U)$ modulo ${\rm im}(U)$ is the cone over the hexagon in Figure \ref{fig:d3n6}. The point $z$ lies in the grey pentagonal cell, which is ${\rm Def}(\Sigma) \in {\rm Ch}(U)$. The figure also shows the Schlegel diagram with respect to the quadrilateral facet of $\rho_2 \in \Sigma(1)$ for three different cells of ${\rm Ch}(U)$. 
\begin{figure}
\centering
\includegraphics[height = 5cm]{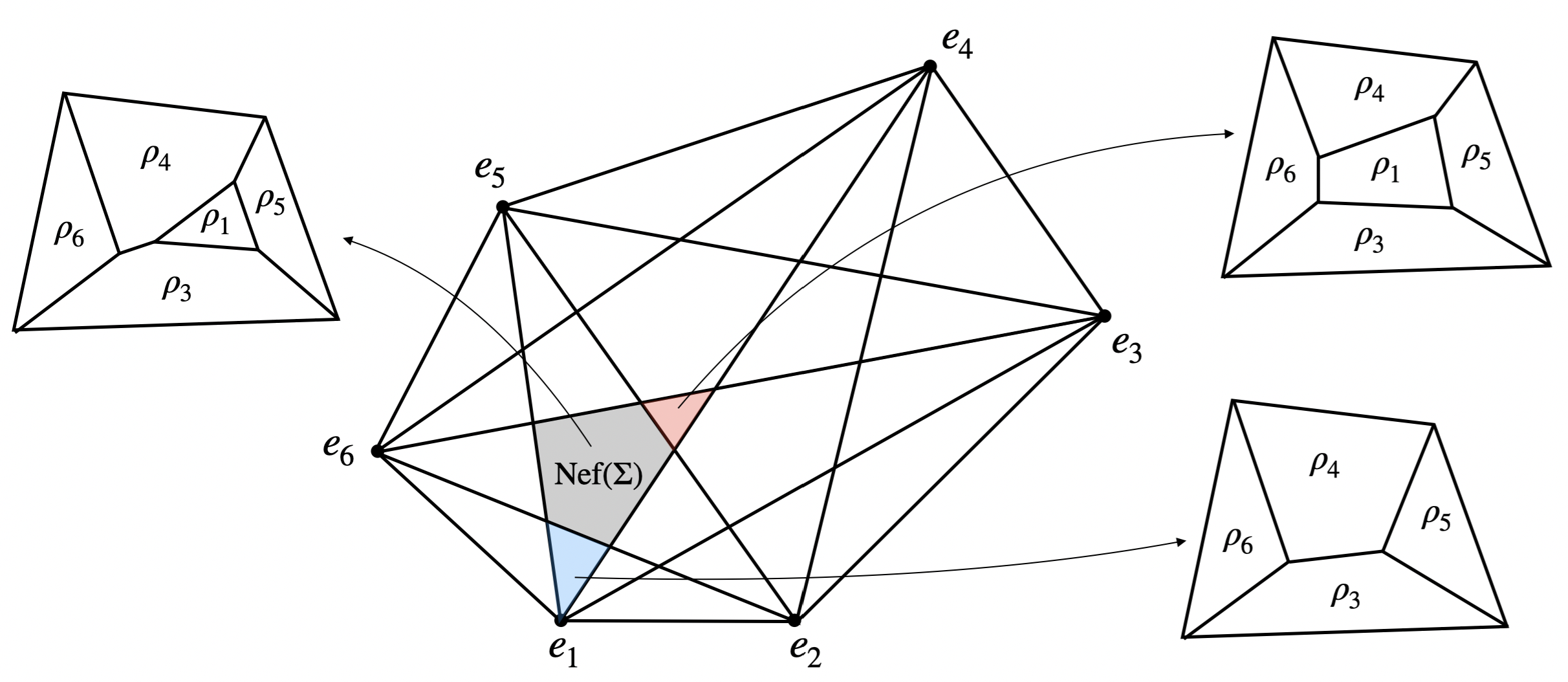}
\caption{The chamber complex of a simple $3$-polytope.}
\label{fig:d3n6}
\end{figure}
The adjoint ${\rm Adj}_{\Sigma}$ is a cubic with eight terms, one for each vertex: 
\[
12 \, x_1x_2x_5 + 48 \,  x_1x_3x_5 + 24 \,  x_1x_3x_6 + \,  36 \, x_1x_4x_5 
+ 28 \, x_1x_4x_6 + 28 \,  x_2x_3x_6 + 40 \, x_2x_4x_6 + 20 \, x_2x_5x_6. \]
As predicted by Proposition \ref{prop:linspacefromcham}, it vanishes on the following linear spaces: 
\[\det \left (\begin{smallmatrix}
-4 & -3 & -1 & x_3\\
2 & 0 & 0 & x_4 \\ 
0 & 2 & 0 & x_5 \\ 
-8 & -11 & -7 & x_1 
\end{smallmatrix} \right )\, = \, x_1 \, = \, 0, \quad \, \det \left ( \begin{smallmatrix}
-8 & -11 & -7 & x_1 \\
0 & 0 & 2 & x_6 \\ 
-4 & -3 & -1 & x_3\\
2 & 0 & 0 & x_4 \\ 
\end{smallmatrix} \right ) \, = \, x_3 \, = \, x_4 \, = \, 0.\]
The first is $W_\Delta(x) = x_{\sigma_{\Delta}} = 0$ for the triangular facet $\Delta$ corresponding to $\rho_1$, and the second corresponds to the edge formed by $\rho_3, \rho_4$. There are $12$ more such linear spaces: $11$ for the other edges, and one for the triangle $\rho_6$. 
The $2$-plane corresponding to the edge $\rho_3, \rho_5$ projects to the cone generated by $e_4, e_6$ in Figure \ref{fig:d3n6}. Note that this is not a facet hyperplane of ${\rm Def}(\Sigma)$. 

The point $z_0 = (7,0,1,0,0,2)$ lies on the boundary of the deformation cone ${\rm Def}(\Sigma)$ corresponding to the blue cell in Figure \ref{fig:d3n6}. Substituting $x = U\, y + z_0$ in ${\rm Adj}_\Sigma$, we find that the quadratic adjoint ${\rm adj}_{P_{z}}$ factors into two linear forms when $z \rightarrow z_0$. One of these linear forms is $u_1 \cdot y + (z_0)_1 = -8 \, y_1 -11 \, y_2 - 7 \, y_3 + 7$, as predicted by \eqref{eq:proof5.7}. The other linear form defines the adjoint plane of $P_{z_0}$. Using $z_1 = (33,-26,9,0,0,0)$ instead, we find that the quadratic surface defined by ${\rm adj}_{P_{z_1}} = -16 \cdot (120 \, y_1 y_2  - 38 \, y_1 y_3  + 165 \, y_2 ^2 + 79 \, y_2 y_3  - 495 \, y_2  + 8 \, y_3 ^2 - 72 \, y_3)$ passes through the non-simple vertex $y = (0,3,0)$ defined by $\rho_1,\rho_3, \rho_4,\rho_6$. The point $z_1$ lies on the facet of ${\rm Def}(\Sigma)$ shared with the red triangle in Figure \ref{fig:d3n6}.
\end{example}

\section{Singular locus} \label{sec:singlocus}

We work with a simplicial fan $\Sigma$ in $\mathbb{R}^d$ which, for now, is not necessarily complete or projective. The singular locus ${\rm Sing}({\cal A}_\Sigma) \subset {\cal A}_\Sigma$ is defined by $n$ equations of degree $n-d-1$: 
\begin{equation} \label{eq:Sing} {\rm Sing}({\cal A}_\Sigma) \, = \, \Big \{ x \in \mathbb{P}^{n-1} \, : \, \frac{\partial {\rm Adj}_\Sigma(x)}{\partial x_\rho} \, = \, 0, \text{ for all } \rho \in \Sigma(1) \Big \}. \end{equation}
We have seen a couple of scenarios in which ${\rm Adj}_\Sigma$ is reducible, see Lemma \ref{lem:factorfans} and Proposition \ref{prop:coordinatehyperplanes}. In these cases, one can easily find $(n-3)$-dimensional components of ${\rm Sing}({\cal A}_\Sigma)$ by intersecting irreducible components of ${\cal A}_\Sigma$. The rest of the singular locus consists of the singular loci of individual factors. We illustrate this with an example.    

\begin{example} \label{ex:threefactors}
We consider the fan from Example \ref{ex:noncomplete}, whose universal adjoint defines a cubic four-fold ${\cal A}_\Sigma \subset \mathbb{P}^5$, see Equation \eqref{eq:AdjIncomplete}. We have ${\rm Adj}_\Sigma \, = \, x_6 \cdot ( x_1+x_3) \cdot (x_2 + x_4)$.
By Proposition \ref{prop:coordinatehyperplanes}, the factor $x_6$ is seen from $\Sigma(1) \setminus \overline{\Sigma}(1) = \{ \rho_6 \}$. The other two factors form the universal adjoint of the star fan $\Sigma_{\rho_5}$, which is a product of two one-dimensional fans, see Lemma \ref{lem:factorfans}. The singular locus of ${\cal A}_\Sigma$ consists of three $3$-planes: 
\[ {\rm Sing}({\cal A}_\Sigma) \, = \, V(x_6, x_1 + x_3) \cup V(x_6, x_2 + x_4) \cup V(x_1+x_3,x_2+x_4). \qedhere \]
\end{example}

The following simple observation implies that ${\cal A}_\Sigma$ is singular when $n-d \geq 3$.

\begin{proposition}
    For each subset $J \subseteq \Sigma(1)$ with $|J| \geq d+2$, we have $\Lambda_J \subseteq {\rm Sing}({\cal A}_\Sigma)$.
\end{proposition}
\begin{proof}
    Since ${\rm Adj}_\Sigma$ has only squarefree monomials, so do its partial derivatives. These are the defining equations of ${\rm Sing}({\cal A}_\Sigma)$, see \eqref{eq:Sing}. Squarefree monomials of degree $n-d-1$ vanish trivially when at least $ d+2$ out of $n$ coordinates are zero, which implies the proposition.   
\end{proof}

Recall that $\Sigma^c$ is the set of subsets of $\Sigma(1)$ which do not generate a cone of $\Sigma$. The minimal elements of $\Sigma^c$ are called \emph{primitive collections}. We encountered these in Section \ref{sec:fano}. Let $J \in \Sigma^c$ be a primitive collection. We investigate the intersection of ${\rm Sing}({\cal A}_\Sigma)$ with the coordinate subspace $\Lambda_J = V(x_\rho \, : \, \rho \in J)$. We start from the following observation: 
\begin{equation} \label{eq:cases1} \Big ( \frac{\partial {\rm Adj}_\Sigma}{\partial x_\rho} \Big )_{|\Lambda_J} \, = \, \begin{cases}
    \frac{\partial}{\partial x_\rho} \Big ( ({\rm Adj}_{\Sigma})_{|\Lambda_J} \Big ) & \text{ if } \rho \notin J \\
    \frac{\partial}{\partial x_\rho} \Big ( ({\rm Adj}_{\Sigma})_{|\Lambda_{J \setminus \{ \rho \}}} \Big ) & \text{ if } \rho \in J \\
\end{cases}.\end{equation}
The expression for $\rho \in J$ follows from the fact that all monomials of ${\rm Adj}_\Sigma$ are squarefree. By Proposition \ref{prop:containedinB} and the fact that $\Lambda_J$ is an irreducible component of $Z(\Sigma)$, the restriction $({\rm Adj}_\Sigma)_{|\Lambda_J}$ is zero. Since $J \in \Sigma^c$ is a primitive collection, we have that for each $\rho \in J$, the set $J \setminus \{ \rho \} \subset \Sigma(1)$ generates a cone $\tau_\rho \in \Sigma$ of dimension $|J|-1$. Hence, we have $\Lambda_{J \setminus \{ \rho \}} = \Lambda_{\tau_\rho}$, where $\Lambda_{\tau_\rho}$ is as in Equation \eqref{eq:Lambdatau}. The restriction of ${\rm Adj}_\Sigma$ to $\Lambda_{\tau_\rho}$ was derived in Lemma \ref{lem:restrict}: 
\[ ({\rm Adj}_{\Sigma})_{|\Lambda_{\tau_\rho}} \, = \, c_{\tau_\rho}^{-1} \cdot \Big ( \prod_{\substack{\rho' \notin {\rm nb}(\tau_\rho) \\ \rho' \not \in \tau_\rho(1)} }  x_{\rho'} \Big ) \cdot {\rm Adj}_{\Sigma_{\tau_\rho},U_{\tau_\rho}}. \]
Clearly $\rho \notin \tau_\rho(1)$ and, since $J$ is primitive, we have $\rho \notin {\rm nb}(\tau_\rho)$. We conclude that \eqref{eq:cases1} gives
\begin{equation} \label{eq:cases2} \Big ( \frac{\partial {\rm Adj}_{\Sigma,U}}{\partial x_\rho} \Big )_{|\Lambda_J} \, = \, \begin{cases}
    0 & \text{ if } \rho \notin J \\
    c_{\tau_\rho}^{-1} \cdot \Big ( \prod_{\substack{\rho' \notin {\rm nb}(\tau_\rho) \\ \rho' \not \in J} }  x_{\rho'} \Big ) \cdot {\rm Adj}_{\Sigma_{\tau_\rho},U_{\tau_\rho}} & \text{ if } \rho \in J \\
\end{cases}.\end{equation}
The above discussion leads immediately to a proof of the following statement. 
\begin{proposition} \label{prop:singirrel}
    Let $J \in \Sigma^c$ be a primitive collection. For $\rho \in J$, let $\tau_\rho \in \Sigma(|J|-1)$ be the cone generated by $J \setminus \{ \rho \}$. The intersection ${\rm Sing}({\cal A}_\Sigma) \cap \Lambda_J$ is given by $2\cdot |J|$ equations:
    \begin{equation} \label{eq:singprimitive} 
    x_\rho \, = \, 0 \quad \text{ and } \quad  \Big ( \prod_{\substack{\rho' \notin {\rm nb}(\tau_\rho) \\ \rho' \not \in J} }  x_{\rho'} \Big ) \cdot {\rm Adj}_{\Sigma_{\tau_\rho},U_{\tau_\rho}}  \, = \, 0 \quad \text{  for all  } \rho \in J.
    \end{equation}
    In particular, if $\Sigma$ is the normal fan of a simple $d$-dimensional polytope $P$, then $J$ is a set of facets of $P$. For $Q \in J$, let $\Delta_Q = \bigcap_{Q' \in J \setminus \{ Q\}} Q'$. The variety ${\rm Sing}({\cal A}_P) \cap \Lambda_J$ is given by
    \begin{equation} \label{eq:singprimitivepolytope} x_Q \, = \, 0 \quad \text{ and } \quad  \Big ( \prod_{\substack{Q'\cap \Delta_Q = \emptyset \\ Q' \not \in J} }  x_{Q'} \Big ) \cdot {\rm Adj}_{{\Delta_Q},U_{\Delta_Q}}  \, = \, 0 \quad \text{  for all  } Q \in J. \end{equation}
\end{proposition}
Notice that the algebraic variety defined by \eqref{eq:singprimitive} (or \eqref{eq:singprimitivepolytope}) visibly splits up into several simple components, which makes decomposing the corresponding ideal in $R_\Sigma$ relatively easy. 

\begin{example}
    The set $J = \{ \rho_5, \rho_6 \}$ is a primitive collection for the fan $\Sigma$ in Example \ref{ex:noncomplete}. We have $\tau_{\rho_6} = \rho_5$ and ${\rm nb}(\rho_5) \setminus J = \emptyset$. The formula \eqref{eq:cases2} gives 
    $\frac{\partial {\rm Adj}_{\Sigma,U}}{\partial x_6} \, = \, {\rm Adj}_{\Sigma_{\rho_5}, U_{\rho_5}} \, = \, x_3x_4 + x_1x_4 + x_1x_2 + x_2x_3 \, = \, 0 $,    which is checked from \eqref{eq:AdjIncomplete}. We also have $\tau_{\rho_5} = \rho_6$, whose star fan has no full-dimensional cones, so ${\rm Adj}_{\Sigma_{\rho_6}, U_{\rho_6}} = 0$. The equations \eqref{eq:singprimitive} read $x_5 = x_6 = (x_1+x_3)(x_2+x_4) = 0$. This defines a union of two planes, given by ${\rm Sing}({\cal A}_\Sigma) \cap \{ x_5 = x_6 = 0\}$ by Example \ref{ex:threefactors}. 
\end{example}

\begin{example} \label{ex:hexagon}
    Consider the hexagon $P \subset \mathbb{R}^2$ with normal fan given by 
    \[ U \, = \, \begin{pmatrix}
        e_1 & e_1+e_2 & e_2 & -e_1 & -e_1 -e_2 & -e_2\end{pmatrix}^t \, \, \in \, \mathbb{R}^{6 \times 2}.\]
    The universal adjoint ${\cal A}_P$ is a quartic hypersurface in $\mathbb{P}^5$. The set $J = \{ \rho_1, \rho_5\}$ is a primitive collection, and the equations \eqref{eq:singprimitive} read $x_1 \, = \, x_2x_3(x_4+x_6) \, = \, x_5 \, = \, x_3x_4(x_2+x_6) \, = \, 0$. This defines a union of a plane $\{x_1 = x_3 = x_5 = 0\}$ and four lines in $\mathbb{P}^5$. Repeating this for the other eight primitive collections, we obtain 30 lines and two planes in total. All of these are contained in ${\rm Sing}({\cal A}_P)$. A computation in \texttt{Oscar.jl} \cite{OSCAR,mathrepo} shows that this constitutes the full singular locus. We shall prove that this is no coincidence (Theorem \ref{thm:singngon}). 
\end{example}

\begin{example} \label{ex:singassoc}
    Let $P \subset \mathbb{R}^3$ be the associahedron in Figure \ref{fig:ABHY3}. We have computed in Section \ref{subsec:assoc} that ${\rm Sing}({\cal A}_P) \cap Z(\Sigma)$ has 133 irreducible components. The nonlinear components have degree three or seven. The twelve components of degree three are explained as follows. A pentagonal facet of $P$ does not intersect one of the quadrilateral facets. Each such a pair forms a primitive collection $J$, e.g., $J = \{ Q_{36}, Q_{15}\}$. The intersection ${\rm Sing}({\cal A}_P) \cap \Lambda_J$ is 
    \[\renewcommand\arraystretch{1.2} \begin{matrix} x_{36} \, = \, x_{15} \, = \,  x_{14}   x_{25}  x_{24}  (x_{35} + x_{46})     (x_{13} + x_{26}) \, = \, 0\\
   x_{46}   x_{26} (x_{13} x_{14} x_{24} + x_{13} x_{14} x_{35} + x_{13} x_{25} x_{35} + x_{14} x_{24} x_{25} + x_{24} x_{25} x_{35})    \, = \, 0 \end{matrix}\]
   This involves the reducible quadratic adjoint of $Q_{36}$, and the cubic pentagonal adjoint of $Q_{15}$. There are 20 four-dimensional components, 18 of which are linear. The two degree-three components are $x_{36} \, = \, x_{15} \, = \, {\rm Adj}_{Q_{36}} \, = \, {\rm Adj}_{Q_{15}} \, = \, 0$. Repeating this for each pentagon gives twelve cubic four-folds. The six components of degree seven in ${\rm Sing}({\cal A}_P)$ come from the six pairs of non-adjacent pentagonal facets. Such a pair also forms a primitive collection. For instance, for $J = \{Q_{46},Q_{35} \}$, Proposition \ref{prop:singirrel} gives the following equations for ${\rm Sing}({\cal A}_P) \cap \Lambda_J$: 
   \[ x_{46} \, = \, x_{35} \, = \, x_{15}      x_{25} \, {\rm Adj}_{Q_{46}} \, = \, 
   x_{14}  x_{24} \, {\rm Adj}_{Q_{35}} \, = \, 0.   \]
   This contains the degree-nine four-fold $x_{46} \, = \, x_{35} \, = \, {\rm Adj}_{Q_{46}} \, = \,  {\rm Adj}_{Q_{35}} \, = \, 0$, which is checked to decompose into two four-planes and one component of degree seven. 
\end{example}

We proceed by studying singular points of ${\cal A}_\Sigma$ which are not contained in $\Lambda_J$ for any primitive collection. That is, we want to characterize points $x \in {\rm Sing}({\cal A}_\Sigma) \setminus Z(\Sigma)$. We work under the mild assumption that $\Sigma = \overline{\Sigma}$, with $\overline{\Sigma}$ as in Proposition \ref{prop:coordinatehyperplanes}. Define 
\[ \phi_\Sigma : \mathbb{P}^{n-1} \setminus Z(\Sigma) \longrightarrow \mathbb{P}^{|\Sigma(d)| -1}, \quad \text{with} \quad \phi_\Sigma(x) \, = \, (x^{\hat{\sigma}})_{\sigma \in \Sigma(d)}.\]
The coordinates of $\phi_\Sigma$ are the minimal generators of the irrelevant ideal $B(\Sigma)$ (here we need $\Sigma = \overline{\Sigma}$). The closure of the image of this map is a projective toric variety $Y_\Sigma \subseteq \mathbb{P}^{|\Sigma(d)| -1}$, which is \emph{not} the abstract normal toric variety $X_\Sigma$ usually associated to $\Sigma$. We define a matrix $M$ of size $|\Sigma(1)| \times |\Sigma(d)|$ whose rows and columns are indexed by the rays and $d$-dimensional cones of $\Sigma$ respectively. The entry $M_{\rho,\sigma}$ is $|\det(U_\sigma)|$ if $\rho \notin \sigma$, and $0$ otherwise. Here is an~example. 

\begin{example} \label{ex:MandY}
    For a pentagon and a hexagon, the matrix $M$ takes the following form:
    \begin{equation} \label{eq:MpentMhex} M_{\rm pent} \, = \, \begin{pmatrix}
        0 & u_{23} & u_{34} & u_{45} & 0 \\ 
        0 & 0 & u_{34} & u_{45} & u_{15} \\ 
        u_{12} & 0 & 0 & u_{45} & u_{15} \\
        u_{12} & u_{23} & 0 & 0 & u_{15} \\
        u_{12} & u_{23} & u_{34} & 0 & 0
    \end{pmatrix}, \quad M_{\rm hex} \, = \,  \begin{pmatrix}
        0 & u_{23} & u_{34} & u_{45} & u_{56} & 0 \\ 
        0 & 0 & u_{34} & u_{45} & u_{56} & u_{16} \\ 
        u_{12} & 0 & 0 & u_{45} & u_{56} & u_{16} \\
        u_{12} & u_{23} & 0 & 0 & u_{56} & u_{16} \\
        u_{12} & u_{23} & u_{34} & 0 & 0 & u_{16} \\
        u_{12} & u_{23} & u_{34} & u_{45} & 0 & 0
    \end{pmatrix}. \end{equation}
    Here $u_{ij} = | \det U_{ij}|$. In Examples \ref{ex:pentagonintro} and \ref{ex:hexagon}, all nonzero entries are one. The toric varieties $Y_\Sigma$ are $\mathbb{P}^4$ for the pentagon, and the hypersurface $\{ y_{12}y_{34}y_{56} = y_{16} y_{23}y_{45} \} \subset \mathbb{P}^5$ for the hexagon. Here $\mathbb{P}^5$ has homogeneous coordinates $(y_{12}: y_{23}: y_{34}: \cdots: y_{16})$ and $\phi_\Sigma$ is the map 
    \[ (x_1: \ldots, x_6) \, \longmapsto \, (x_3x_4x_5x_6: x_1x_4x_5x_6: \cdots: x_2x_3x_4x_5). \qedhere \]
\end{example}

Below, the homogeneous coordinates of a point $y = (y_\sigma)_{\sigma \in \Sigma(d)} \in \mathbb{P}^{|\Sigma(d)|-1}$ are indexed by the cones of $\Sigma(d)$. We may regard $y$ as a column vector of length $|\Sigma(d)|$ and write $M \cdot y$ for the matrix-vector product. The $n$ linear equations $M \cdot y = 0$ are well defined on $\mathbb{P}^{|\Sigma(d)|-1}$.

\begin{lemma} \label{lem:sing}
    If $\Sigma = \overline{\Sigma}$ and $x \in \mathbb{P}^{n-1}$ belongs to ${\rm Sing}({\cal A}_\Sigma) \setminus Z(\Sigma)$, then we have $M \cdot \phi_\Sigma(x) = 0$.
\end{lemma}
\begin{proof}
    By construction, the entries of the vector $M \cdot \phi_\Sigma(x)$ are $x_\rho \cdot \frac{\partial {\rm Adj}_\Sigma}{\partial x_\rho}$ for $\rho \in \Sigma(1)$. 
\end{proof}
Lemma \ref{lem:sing} leads to a sufficient criterion for checking that the singular locus of ${\cal A}_\Sigma$ is contained in $Z(\Sigma)$. If this holds, then the singular locus is completely described (set-theoretically) by the equations \eqref{eq:singprimitive}-\eqref{eq:singprimitivepolytope}. More precisely, for each primitive collection $J \in \Sigma^c$, let $I_J$ be the ideal generated by the $2 \cdot |J|$ polynomials in \eqref{eq:singprimitive}. If ${\rm Sing}({\cal A}_\Sigma) \subseteq Z(\Sigma)$, then Proposition \ref{prop:singirrel} implies that ${\rm Sing}({\cal A}_\Sigma) =\bigcup_J V(I_J)$. The structure of the generators of $I_J$ makes the irreducible decomposition of ${\rm Sing}({\cal A}_\Sigma) \cap Z(\Sigma)$ almost entirely combinatorial. 
\begin{corollary} \label{cor:checkIsatB}
    If $\Sigma = \overline{\Sigma}$ and $\{y \in Y_\Sigma \, : \, M \cdot y = 0\} = \emptyset$, then we have ${\rm Sing}({\cal A}_\Sigma) \subseteq Z(\Sigma)$.
\end{corollary}

\begin{example}
    For the fan $\Sigma$ from Example \ref{ex:cuboid}, the variety $Y_\Sigma$ is a 5-fold in $\mathbb{P}^7$ given by $y_{145}y_{235} - y_{245}y_{135} \, = \, y_{236}y_{145} - y_{246}y_{135} \, = \, y_{236}y_{245} - y_{246}y_{235} \, = \, 0$.
    Here $y_\sigma$ corresponds to the coordinate $x^{\hat{\sigma}}$ of $\phi_\Sigma$. The equations $M \cdot y = 0$ define a line in $\mathbb{P}^7$, which is not incident to $Y_\Sigma$. By Corollary \ref{cor:checkIsatB}, we have ${\rm Sing}({\cal A}_\Sigma) \subseteq Z(\Sigma)$. This is verified in the code \cite{mathrepo}. The singular locus is a union of six curves, four are lines and two have degree two. The degree-two components are explained by the primitive collections $\{ \rho_1,\rho_2\}, \{\rho_5,\rho_6\}$ via Proposition~\ref{prop:singirrel}.
\end{example}

\begin{theorem} \label{thm:singngon}
    Let $\Sigma$ be the normal fan of a convex $n$-gon with ray matrix $U$. Let $u_{ij} = | \det U_{ij} |$ and assume that the $n$ maximal cones of $\sigma$ are indexed by rows $12, 23, 34, \ldots, 1n$. If $n$ is not a multiple of $4$ or $u_{12}u_{34}\cdots u_{n-1,n} \neq u_{1n}u_{23}\cdots u_{n-2,n-1}$, then we have
    \begin{enumerate}
        \item ${\rm Sing}({\cal A}_\Sigma)$ equals the union of all solution sets to \eqref{eq:singprimitive}-\eqref{eq:singprimitivepolytope}, where $J$ runs over all $\frac{n(n-3)}{2}$ primitive collections of $\Sigma(1)$, 
        \item ${\cal A}_\Sigma$ is irreducible and $\dim {\rm Sing}({\cal A}_\Sigma) \leq n-4$.
    \end{enumerate}
\end{theorem}
\begin{proof}
    By Proposition \ref{prop:singirrel}, the first statement will follow from the inclusion ${\rm Sing}({\cal A}_\Sigma) \subseteq Z(\Sigma)$. 
    Notice that we can indeed always order the rays of $\Sigma(1)$ so that the maximal cones are indexed by $12, 23, \ldots, 1n$. The coefficients $u_{i,i+1}$ of the adjoint are strictly positive. 
    
    We claim that the matrix $M$ constructed above has rank $n$ if $n$ is odd, and rank $n-1$ when $n$ is even. These matrices are shown in \eqref{eq:MpentMhex} for $n = 5$ and $n = 6$. To show this claim, note that $M$ has the same rank as the matrix $M_1$ in which we replace $u_{i,i+1}$ by $1$ for all $i$. For $\ell = 0, \ldots, n-1$, let $v_\ell \in \mathbb{C}^n$ be the column vector $({\rm exp}(\sqrt{-1} \frac{2 \pi}{n} k \ell))_{k = 0, \ldots, n-1}$. One checks~that 
    \[ M_1 \cdot v_\ell \, = \, \Big ( \sum_{k = 1}^{n-2} e^{\sqrt{-1} \frac{2\pi}{n}k \ell} \Big ) \cdot v_\ell. \]
    This identifies the eigenvalues and eigenvectors of $M_1$. In particular, $M_1$ has orthogonal eigenvectors (it is a normal matrix). The eigenvalues $\lambda_\ell, \, \ell = 0, \ldots, n-1$ are 
    \[ \lambda_\ell \, = \, \sum_{k = 1}^{n-2} e^{\sqrt{-1} \frac{2\pi}{n}k \ell}  \, = \, \begin{cases}
        n-2 & \ell = 0 \\ 
        -1 - e^{\sqrt{-1}\frac{2\pi(n-1)}{n} \ell} & \ell = 1, \ldots, n-1
    \end{cases}.\]
    If $n$ is odd, then all eigenvalues are nonzero. If $n$ is even, then $\lambda_{n/2} = 0$, and all other eigenvalues are nonzero. This proves our claim about ${\rm rank}(M) = {\rm rank}(M_1)$. 

    If $n$ is odd and $x \in {\rm Sing}({\cal A}_\Sigma) \setminus Z(\Sigma)$ then $\phi_\Sigma(x)$ is a non-trivial kernel vector of $M$ by Lemma \ref{lem:sing}. But this contradicts ${\rm rank}(M) = n$, so part 1 of the theorem is proved for odd $n$.

    If $n$ is even, one checks that the one-dimensional kernel of $M$ is spanned by the vector
    \begin{equation} \label{eq:yproof} y \, = \, \begin{pmatrix}
        u_{12}^{-1} & - u_{23}^{-1} & u_{34}^{-1} & - u_{45}^{-1}& \cdots & - u_{1n}^{-1}
    \end{pmatrix}^t.\end{equation}
    Hence, if $x \in {\rm Sing}({\cal A}_\Sigma) \setminus Z(\Sigma)$, then we must have $\phi_\Sigma(x) = (u_{12}^{-1}: -u_{23}^{-1}: \cdots: -u_{1n}^{-1})$ by Lemma \ref{lem:sing}. We must also have $\phi_\Sigma(x) \in  Y_\Sigma$. The monomial parametrization of the toric variety $Y_\Sigma$ is encoded by the columns of the matrix $M_1$. It follows from basic toric geometry, see for instance \cite[Propositions 1.1.8 and 2.1.4]{CoxLittleSchenck2011} that $Y_\Sigma$ is a hypersurface, and its binomial defining equation is given by $y_{12}y_{34} \cdots y_{n-1,n} = y_{1n}y_{23} \cdots y_{n-2,n-1}$. Plugging in \eqref{eq:yproof} gives precisely the condition $u_{12}u_{34}\cdots u_{n-1,n} = (-1)^{\frac{n}{2}} u_{1n}u_{23}\cdots u_{n-2,n-1}$. By positivity of $u_{i,i+1}$, this equality cannot hold unless $n$ is a multiple of $4$. It is a \emph{genericity} condition because, even if $n$ is a multiple of $4$, the equality only holds for special $n$-gons (see Example \ref{ex:octagon}). 

    With our ordering of the rays of $\Sigma$, the primitive collections $J$ are $\{\rho_i, \rho_j\}$ for $1 \leq i < j-1 \leq n-1$ and $\{i,j\} \neq \{1,n\}$. There are indeed $\frac{n(n-3)}{2}$ of them. Each component $\Lambda_J \subseteq \mathbb{P}^{n-1}$ is of dimension $n-3$, and the equations \eqref{eq:singprimitivepolytope} define a strict subvariety of $\Lambda_J$. Hence we have proved that $\dim {\rm Sing}({\cal A}_\Sigma) \leq n-4$. If ${\cal A}_\Sigma$ were reducible, then the singular locus would contain the intersection of two of its components, which has dimension $n-3$. 
\end{proof}

\begin{example} \label{ex:octagon}
    We consider an octagon whose normal fan $\Sigma$ has the following ray matrix: 
    \[ U \, = \, \begin{pmatrix}
        e_1 & e_1+e_2 & e_2 & -e_1+e_2 & -e_1 & -e_1 - e_2 & -e_2 & e_1 - \alpha \, e_2
    \end{pmatrix}^t \, \, \in \, \mathbb{R}^{8 \times 2}.\]
    Here $\alpha$ is a positive real number. The universal adjoint ${\cal A}_\Sigma$ is a hypersurface of degree 6 in~$\mathbb{P}^7$ defined by
    ${\rm Adj}_\Sigma(x) \,  = \,x_1 x_2 x_3 x_4 x_5 x_6  + \cdots  + \alpha \cdot x_2 x_3 x_4 x_5 x_6 x_7  + x_3 x_4 x_5 x_6 x_7 x_8$.
    We have $u_{12}u_{34}u_{56}u_{78} \neq u_{18}u_{23}u_{45}u_{67}$ unless $\alpha = 1$. If $\alpha = 1$, then ${\rm Sing}({\cal A}_\Sigma)$ contains the~line 
    \[ V(x_6 + x_8, \, x_5 + x_7, \, x_4 - x_8, \,  x_3 - x_7, \, x_2 + x_8, \, x_1 + x_7). \]
    This line is not contained in $Z(\Sigma)$. For $\alpha \neq 1$, the singular locus ${\rm Sing}({\cal A}_\Sigma) \subsetneq Z(\Sigma)$ has 56 components. All of them have degree one. There are 40 three-planes, and 16 four-planes.  
\end{example}
We will combine Theorem \ref{thm:singngon} with a Bertini argument to show generic smoothness of Warren's adjoint \eqref{eq:adj} for polygons. This relies on the geometric observation in Example \ref{ex:geometric}. Here is a general statement for a $d$-dimensional polytope $P_z = \{ y \in \mathbb{R}^d \, : \, U \, y + z \geq 0 \}$.

\begin{theorem} \label{thm:bertini}
    Let $\Sigma$ be the normal fan of a full-dimensional simple polytope $P \subset \mathbb{R}^d$ and let $U \in \mathbb{R}^{n \times d}$ be its ray matrix. If $\dim {\rm Sing}({\cal A}_\Sigma) < n - d - 1$, then Warren's adjoint hypersurface $\{ y \in \mathbb{C}^d \, : \, {\rm adj}_{P_z}(y) = 0 \}$ is smooth for generic $z \in {\rm int}({\rm Def}(\Sigma)) = \{ z \in \mathbb{R}^n \, : \, \Sigma_{P_z}= \Sigma \}$. 
\end{theorem}

\begin{proof}
    A $d$-plane $H \simeq \mathbb{P}^d$ satisfying $H \supset \mathbb{P}({\rm im}(U))$ is obtained as the span of $\mathbb{P}({\rm im}(U))$ and $z$, where $z$ is viewed as a point in $\mathbb{P}^{n-1}$. Once we show that $({\cal A}_\Sigma \setminus \mathbb{P}({\rm im}(U))) \cap H$ is smooth for generic $z$, we know that this holds in particular for $z$ in a dense open subset of the deformation cone ${\rm Def}(\Sigma)$. This implies the theorem, since by Lemma \ref{lem:universaladjoint} we have 
    \[ ({\cal A}_\Sigma \setminus \mathbb{P}({\rm im}(U))) \cap H \, \simeq \,  \{ y \in \mathbb{C}^d \, : \, {\rm adj}_{P_z}(y) = 0\} \quad \text{ for } z \in {\rm int}({\rm Def}(\Sigma)).\]
    Let $\ell_1(x), \ldots, \ell_{n-d}(x)$ be a basis for the linear forms vanishing on $\mathbb{P}({\rm im}(U))$. The fibers of the morphism $\ell: {\cal A}_\Sigma \setminus \mathbb{P}({\rm im}(U)) \rightarrow \mathbb{P}^{n-d-1}$ given by $\ell(x) = (\ell_1(x): \cdots : \ell_{n-d}(x))$ are precisely the intersections $({\cal A}_\Sigma \setminus \mathbb{P}({\rm im}(U))) \cap H$, where $H \supset \mathbb{P}({\rm im}(U))$. We must show that generic fibers are smooth. We will do so by applying a version of Bertini's theorem \cite{JouanolouBertini}. First of all, since $\dim {\rm Sing}({\cal A}_\Sigma) < n-d-1$ by assumption, point 1 in \cite[Theor\`eme 6.10]{JouanolouBertini} assures that generic fibers of $\ell$ do not intersect ${\rm Sing}({\cal A}_\Sigma) \setminus \mathbb{P}({\rm im}(U))$. Hence, to study generic fibers, we may restrict $\ell$ to the smooth quasi-projective variety ${\cal A}_\Sigma \setminus ({\rm Sing}({\cal A}_\Sigma) \cup \mathbb{P}({\rm im}(U)))$. Point 2 in \cite[Theor\`eme 6.10]{JouanolouBertini} says that generic fibers of this restriction, and hence of $\ell$, are~smooth. 
\end{proof}

\begin{corollary} \label{cor:genericsmoothngon}
    In the situation of Theorem \ref{thm:singngon}, if $n$ is not a multiple of $4$ or $u_{12}u_{34}\cdots u_{n-1,n} \neq u_{1n}u_{23}\cdots u_{n-2,n-1}$, then Warren's adjoint curve $\{ y \in \mathbb{C}^2 \, : \, {\rm adj}_{P_z}(y) = 0\}$ of the polygon $P_z = \{ y \in \mathbb{R}^2\, : \, U \, y + z \geq 0 \}$ is smooth for generic $z \in {\rm Def}(\Sigma)$. 
\end{corollary}

\paragraph{Acknowledgements. }
I am grateful to Bernd Sturmfels and Bruno Gim\'enez Umbert for interesting discussions on zeros of amplitudes. Bernd Sturmfels suggested the name ``toric amplitudes'' as well as the examples in Sections \ref{subsec:pentagon}-\ref{subsec:assoc}, and he pointed me to \cite{Dolgachev}. Theorem \ref{thm:containsU} and Section \ref{sec:fano} were inspired by questions from Mario Kummer and Carolina Figueiredo respectively. Thanks to Nathan Ilten for his guidance on computing Fano schemes. Theorem \ref{thm:bertini} and Corollary \ref{cor:genericsmoothngon} were added to a previous version of this manuscript after Mario Kummer and Dmitrii Pavlov encouraged me to use Theorem \ref{thm:singngon} to show smoothness of the adjoint curve of a polygon, and Rainer Sinn pointed me to Jouanolou's book \cite{JouanolouBertini}.

\footnotesize
\bibliographystyle{abbrv}
\bibliography{references}

@article{lam2024matroids,
  title={Matroids and amplitudes},
  author={Lam, Thomas},
  journal={arXiv:2412.06705},
  year={2024}
}

@article{ranestad2025positive,
  author        = {Kristian Ranestad and Bernd Sturmfels and Simon Telen},
  title         = {What is Positive Geometry?},
  journal       = {Le Matematiche, special volume on Positive Geometry},
  volume        = {80},
  number        = {1},
  pages         = {3--16},
  year          = {2025},
  doi           = {10.4418/2025.80.1.1},
  eprint        = {2502.12815},
  archivePrefix = {arXiv},
  primaryClass  = {math.AG}}

@article{gao2024dual,
  title={Dual Mixed Volume},
  author={Gao, Yibo and Lam, Thomas and Xue, Lei},
  journal={arXiv:2410.21688},
  year={2024}
}

@article{pavlov2025santalo,
  title={Santal{\'o} Geometry of Convex Polytopes},
  author={Pavlov, Dmitrii and Telen, Simon},
  journal={SIAM Journal on Applied Algebra and Geometry},
  volume={9},
  number={1},
  pages={58--82},
  year={2025},
  publisher={SIAM}
}

@article{arkani2018scattering,
  title={Scattering forms and the positive geometry of kinematics, color and the worldsheet},
  author={Arkani-Hamed, Nima and Bai, Yuntao and He, Song and Yan, Gongwang},
  journal={Journal of High Energy Physics},
  volume={2018},
  number={5},
  pages={1--78},
  year={2018},
  publisher={Springer}
}

@inproceedings{Dolgachev,
  title={Corrado {S}egre and nodal cubic threefolds},
  author={Dolgachev, Igor},
  booktitle={From Classical to Modern Algebraic Geometry: Corrado Segre's Mastership and Legacy},
  pages={429--450},
  year={2016},
  organization={Springer}
}

@article{lasserre1983analytical,
  title={An analytical expression and an algorithm for the volume of a convex polyhedron in $\mathbb{R}^n$},
  author={Lasserre, Jean B},
  journal={Journal of optimization theory and applications},
  volume={39},
  pages={363--377},
  year={1983},
  publisher={Springer}
}

@book{fulton1993introduction,
  title={Introduction to toric varieties},
  author={Fulton, William},
  number={131},
  year={1993},
  publisher={Princeton university press}
}

@article{warren1996barycentric,
  title={Barycentric coordinates for convex polytopes},
  author={Warren, Joe},
  journal={Advances in Computational Mathematics},
  volume={6},
  number={1},
  pages={97--108},
  year={1996},
  publisher={Springer}
}

@article{kohn2020projective,
  title={Projective geometry of {W}achspress coordinates},
  author={Kohn, Kathl{\'e}n and Ranestad, Kristian},
  journal={Foundations of Computational Mathematics},
  volume={20},
  number={5},
  pages={1135--1173},
  year={2020},
  publisher={Springer}
}

@book{Wachspress1975,
  author    = {E. L. Wachspress},
  title     = {A Rational Finite Element Basis},
  year      = {1975},
  publisher = {Academic Press},
  address   = {New York},
  isbn      = {978-0-12-728950-0}
}

@article{arkani2017positive,
  title={Positive geometries and canonical forms},
  author={Arkani-Hamed, Nima and Bai, Yuntao and Lam, Thomas},
  journal={Journal of High Energy Physics},
  volume={2017},
  number={11},
  pages={1--124},
  year={2017},
  publisher={Springer}
}

@article{lam2024invitation,
  title={An invitation to positive geometries},
  author={Lam, Thomas},
  journal={Open Problems in Algebraic Combinatorics},
  volume={110},
  pages={159--180},
  year={2024},
  publisher={American Mathematical Society}
}

@book{eisenbud20163264,
  title={3264 and all that: A second course in algebraic geometry},
  author={Eisenbud, David and Harris, Joe},
  year={2016},
  publisher={Cambridge University Press}
}

@book{CoxLittleSchenck2011,
  author    = {David A. Cox and John B. Little and Henry K. Schenck},
  title     = {Toric Varieties},
  year      = {2011},
  publisher = {American Mathematical Society},
  series    = {Graduate Studies in Mathematics},
  volume    = {124},
  address   = {Providence, RI},
  isbn      = {978-0-8218-4813-4},
  doi       = {10.1090/gsm/124}
}

@article{Cox1995,
  author    = {David A. Cox},
  title     = {The Homogeneous Coordinate Ring of a Toric Variety},
  journal   = {Journal of Algebraic Geometry},
  volume    = {4},
  number    = {1},
  year      = {1995},
  pages     = {17--50},
  publisher = {American Mathematical Society},
  issn      = {1056-3911},
  url       = {https://www.ams.org/jag/1995-04-01/S1056-3911-1995-1299003-7/},
}

@article{umbert2025splitting,
  author       = {Bruno Gim{\'e}nez Umbert and Bernd Sturmfels},
  title        = {Splitting {CEGM} Amplitudes},
  journal      = {Journal of High Energy Physics},
  year         = {2025},
  volume       = {2025},
  number       = {4},
  pages        = {76},
  doi          = {10.1007/jhep04(2025)076},
  url          = {https://link.springer.com/article/10.1007/JHEP04(2025)076}
}

@article{arkani2024hidden,
  title={Hidden zeros for particle/string amplitudes and the unity of colored scalars, pions and gluons},
  author={Arkani-Hamed, Nima and Cao, Qu and Dong, Jin and Figueiredo, Carolina and He, Song},
  journal={Journal of High Energy Physics},
  volume={2024},
  number={10},
  pages={1--56},
  year={2024},
  publisher={Springer}
}

@article{cox2009primitive,
  title={Primitive collections and toric varieties},
  author={Cox, David A and von Renesse, Christine},
  journal={Tohoku Mathematical Journal, Second Series},
  volume={61},
  number={3},
  pages={309--332},
  year={2009},
  publisher={Mathematical Institute, Tohoku University}
}

@misc{OSCAR,
  key          = {OSCAR},
  organization = {The OSCAR Team},
  title        = {OSCAR -- Open Source Computer Algebra Research system, Version 1.3.1},
  year         = {2025},
  url          = {https://www.oscar-system.org},
}

@article{lee1989associahedron,
  title     = {The Associahedron and Triangulations of the $n$-Gon},
  author    = {Lee, Carl W.},
  journal   = {European Journal of Combinatorics},
  volume    = {10},
  number    = {6},
  pages     = {551--560},
  year      = {1989},
  publisher = {Elsevier}
}

@misc{mathrepo,
    author = {Telen, S },
    note = "MathRepo page \url{https://mathrepo.mis.mpg.de/ToricAmplitudes}, 2025"
}

@book{JouanolouBertini,
  author    = {Jean-Pierre Jouanolou},
  title     = {Théorèmes de Bertini et applications},
  publisher = {Birkhäuser},
  year      = {1983},
  series    = {Progress in Mathematics},
  volume    = {42},
  address   = {Boston, MA},
  isbn      = {978-3-0348-6231-2},
}

@article{cachazo2022smoothly,
  title={Smoothly splitting amplitudes and semi-locality},
  author={Cachazo, Freddy and Early, Nick and Umbert, Bruno Gim{\'e}nez},
  journal={Journal of High Energy Physics},
  volume={2022},
  number={8},
  pages={1--46},
  year={2022},
  publisher={Springer}
}

@article{cachazo2019scattering,
  title={Scattering equations: from projective spaces to tropical {G}rassmannians},
  author={Cachazo, Freddy and Early, Nick and Guevara, Alfredo and Mizera, Sebastian},
  journal={Journal of High Energy Physics},
  volume={2019},
  number={6},
  pages={1--33},
  year={2019},
  publisher={Springer}
}

@article{CastilloLiu2020NestedBraidFans,
  author    = {Federico Castillo and Fu Liu},
  title     = {Deformation Cones of Nested Braid Fans},
  journal   = {International Mathematics Research Notices},
  volume    = {2022},
  number    = {3},
  pages     = {1973--2026},
  year      = {2020},
  publisher = {Oxford University Press},
  doi       = {10.1093/imrn/rnaa090},
}

@book{WeinzierlFeynmanDiagrams,
  author    = {Weinzierl, Stefan},
  title     = {Feynman Diagrams},
  series    = {Cambridge Monographs on Particle Physics, Nuclear Physics and Cosmology},
  publisher = {Cambridge University Press},
  address   = {Cambridge},
  year      = {2022},
  isbn      = {978-1-108-84298-8},
}

\noindent{\bf Author's address:}
\medskip
\noindent Simon Telen, MPI-MiS Leipzig
\hfill {\tt simon.telen@mis.mpg.de}

\end{document}